\documentclass[11pt]{amsart}
\usepackage{geometry}                % See geometry.pdf to learn the layout options. There are lots.F
\usepackage{longtable}
\usepackage{graphicx}
\usepackage{amssymb}
\usepackage{epstopdf}
\usepackage{amscd}
\usepackage{amsmath}
\usepackage[colorlinks=true,urlcolor=black,citecolor=black,linkcolor=black]{hyperref}
\usepackage{pstricks}
\usepackage{multido}

\calclayout
\newtheorem{theorem}{Theorem}[section]
\numberwithin{equation}{subsection}
\newtheorem{lemma}[theorem]{Lemma}
\newtheorem{proposition}[theorem]{Proposition}
\newtheorem{corollary}[theorem]{Corollary}

\theoremstyle{definition}

\newtheorem{definition}[theorem]{Definition}
\newtheorem{definitions}[theorem]{Definitions}
\theoremstyle{remark}

\newtheorem{remark}[theorem]{Remark}

\newtheorem{example}[theorem]{Example}

\newtheorem{problem}[theorem]{Problem}
\newcommand{\mysubsection}[1]{\refstepcounter{subsection}\vspace{0.3cm}\noindent \arabic{section}.\arabic{subsection}.}

\newcommand\dd{{d_2}}

\newcommand\F{{\mathcal F}}
\newcommand\C{{\mathbb C}}
\newcommand\R{{\mathbb R}}

\newcommand\Z{{\mathbb Z}}
\newcommand\vep{{e}}

\newcommand{\binomial}[2]{\genfrac{(}{)}{0pt}{}{#1}{#2}}

\newcommand{\Id}{{ I}}
\newcommand{\Inv}[1]{{\Phi({#1})}}
\newcommand{\CatA}[1]{{\text{Cat}_\text{A}(#1)}}
\newcommand{\CatB}[1]{{\text{Cat}_\text{B}(#1)}}
\newcommand{\W}{{\mathcal W}}
\newcommand{\X}{{\mathcal X}}
\newcommand{\Y}{{\mathcal Y}}
\newcommand{\ep}{\varepsilon}
\newcommand{\phantomelement}{\varnothing}

\newcommand{\CC}{\mathbb{C}}
\newcommand{\ZZ}{\mathbb{Z}}
\newcommand{\GL}{\operatorname{GL}}
\renewcommand{\geq}{\geqslant}
\renewcommand{\leq}{\leqslant}

\makeatletter
\long\def\tlist@if@empty@nTF #1{%
\expandafter\ifx\expandafter\\\detokenize{#1}\\%
\expandafter\@firstoftwo
\else
\expandafter\@secondoftwo
\fi
}
\providecommand*\wo[1][]{%
J
\tlist@if@empty@nTF{#1}{}{ _#1}% only the false code executed
}
\makeatother
\newcommand{\resetdiagramdefaults}{%
\newgray{minusfill}{0.80}
\newgray{minusedge}{0.65}
\newgray{numcolor}{0}
\psset{unit=0.75cm}
}

\newcommand{\darkergrays}{%
\newgray{minusfill}{0.80}
\newgray{minusedge}{0.65}
}

\newcommand{\putplus}[2]{%
\pspolygon[fillstyle=none,linecolor=minusedge]%
(!#1 #2 [0.5 0.5 0.5 -0.5 -1.5 1.0] transform)%
(!#1 #2 [0.5 0.5 0.5 -0.5 -2.0 0.5] transform)%
(!#1 #2 [0.5 0.5 0.5 -0.5 -1.5 0.0] transform)%
(!#1 #2 [0.5 0.5 0.5 -0.5 -1.0 0.5] transform)
\rput(!#1 #2 [0.5 0.5 0.5 -0.5 -1.5 0.5] transform){\tiny $+$}
}

\newcommand{\putplusbox}[2]{%  Same as putplus put without the +
\pspolygon[fillstyle=none,linecolor=minusedge]%
(!#1 #2 [0.5 0.5 0.5 -0.5 -1.5 1.0] transform)%
(!#1 #2 [0.5 0.5 0.5 -0.5 -2.0 0.5] transform)%
(!#1 #2 [0.5 0.5 0.5 -0.5 -1.5 0.0] transform)%
(!#1 #2 [0.5 0.5 0.5 -0.5 -1.0 0.5] transform)
}

\newcommand{\putminus}[2]{%
\pspolygon[fillstyle=solid,fillcolor=minusfill,linecolor=minusedge]%
(!#1 #2 [0.5 0.5 0.5 -0.5 -1.5 1.0] transform)%
(!#1 #2 [0.5 0.5 0.5 -0.5 -2.0 0.5] transform)%
(!#1 #2 [0.5 0.5 0.5 -0.5 -1.5 0.0] transform)%
(!#1 #2 [0.5 0.5 0.5 -0.5 -1.0 0.5] transform)
\rput(!#1 #2 [0.5 0.5 0.5 -0.5 -1.5 0.5] transform){\tiny $-$}
}

\newcommand{\putminusbox}[2]{% Same as putminus but without the -
\pspolygon[fillstyle=solid,fillcolor=minusfill,linecolor=minusedge]%
(!#1 #2 [0.5 0.5 0.5 -0.5 -1.5 1.0] transform)%
(!#1 #2 [0.5 0.5 0.5 -0.5 -2.0 0.5] transform)%
(!#1 #2 [0.5 0.5 0.5 -0.5 -1.5 0.0] transform)%
(!#1 #2 [0.5 0.5 0.5 -0.5 -1.0 0.5] transform)
}

\newcommand{\putbigminus}[4]{%
\pspolygon[fillstyle=solid,fillcolor=minusfill,linecolor=minusedge]%
(!#3 #2 [0.5 0.5 0.5 -0.5 -1.5 1.0] transform)%  Up
(!#1 #2 [0.5 0.5 0.5 -0.5 -2.0 0.5] transform)%  Left
(!#1 #4 [0.5 0.5 0.5 -0.5 -1.5 0.0] transform)%  Down
(!#3 #4 [0.5 0.5 0.5 -0.5 -1.0 0.5] transform) % Right
\rput(!#1 #3 add 2 div #2 #4 add 2 div [0.5 0.5 0.5 -0.5 -1.5 0.5] transform){\huge $-$}
}

\newcommand{\putbigminusbox}[4]{%
\pspolygon[fillstyle=solid,fillcolor=minusfill,linecolor=minusedge]%
(!#3 #2 [0.5 0.5 0.5 -0.5 -1.5 1.0] transform)%  Up
(!#1 #2 [0.5 0.5 0.5 -0.5 -2.0 0.5] transform)%  Left
(!#1 #4 [0.5 0.5 0.5 -0.5 -1.5 0.0] transform)%  Down
(!#3 #4 [0.5 0.5 0.5 -0.5 -1.0 0.5] transform) % Right
}

\newcommand{\putbigplus}[4]{%
\pspolygon[fillstyle=solid,fillcolor=white,linecolor=minusedge]%
(!#3 #2 [0.5 0.5 0.5 -0.5 -1.5 1.0] transform)%  Up
(!#1 #2 [0.5 0.5 0.5 -0.5 -2.0 0.5] transform)%  Left
(!#1 #4 [0.5 0.5 0.5 -0.5 -1.5 0.0] transform)%  Down
(!#3 #4 [0.5 0.5 0.5 -0.5 -1.0 0.5] transform) % Right
\rput(!#1 #3 add 2 div #2 #4 add 2 div [0.5 0.5 0.5 -0.5 -1.5 0.5] transform){\huge $+$}
}

\newcommand{\putbigplusbox}[4]{%
\pspolygon[fillstyle=solid,fillcolor=white,linecolor=minusedge]%
(!#3 #2 [0.5 0.5 0.5 -0.5 -1.5 1.0] transform)%  Up
(!#1 #2 [0.5 0.5 0.5 -0.5 -2.0 0.5] transform)%  Left
(!#1 #4 [0.5 0.5 0.5 -0.5 -1.5 0.0] transform)%  Down
(!#3 #4 [0.5 0.5 0.5 -0.5 -1.0 0.5] transform) % Right
}

\newcommand{\putinsquare}[3]{%
\rput(!#1 #2 [0.5 0.5 0.5 -0.5 -1.5 0.5] transform){#3}
}

\newcommand{\NWSEline}[2]{%
\psline[linecolor=minusedge]%
(!#1 #1 1 add [0.5 0.5 0.5 -0.5 -1.5 1.0] transform)%  Up
(!#1 #2 1 add [0.5 0.5 0.5 -0.5 -1.0 0.5] transform) % Right
}

\newcommand{\NESWline}[1]{%
\psline[linecolor=minusedge]%
(! 1 #1 1 add [0.5 0.5 0.5 -0.5 -2.0 0.5] transform)%  Left
(!#1 #1 1 add [0.5 0.5 0.5 -0.5 -1.5 1.0] transform) %  Up
}

\newcommand{\drawgrid}[1]{%   Alas, we have to call draw grid with one less then the actual number...
\multido{\n=1+1}{#1}{%        (this makes sense from the Lie point of view -- it's the A_n Weyl group!)
\NWSEline{\n}{#1}
\NESWline{\n}
}
\psline[linecolor=minusedge]%
(!1 2 [0.5 0.5 0.5 -0.5 -2.0 0.5] transform)%  Left
(!1 #1 1 add [0.5 0.5 0.5 -0.5 -1.5 0.0] transform) %  Down
\psline[linecolor=minusedge]%
(! 1 #1 1 add [0.5 0.5 0.5 -0.5 -1.5 0.0] transform)%  Down
(!#1 #1 1 add [0.5 0.5 0.5 -0.5 -1.0 0.5] transform) % Right
}

\newcommand{\putnum}[1]{%
\rput(!#1 #1 [0.5 0.5 0.5 -0.5 -1.5 0.5] transform){\tiny\color{numcolor} #1}
}

\newcommand{\putnums}[1]{%
\multido{\n=1+1}{#1}{%
\putnum{\n}
}
}

\newcommand{\Youngbox}[2]{%
\pspolygon[fillstyle=solid,fillcolor=minusfill,linecolor=minusedge]%
(!#1 #2 [0.7071067811  0 0 0.7071067811 -0.35355339059  0.35355339059] transform)% Top left
(!#1 #2 [0.7071067811  0 0 0.7071067811 -0.35355339059 -0.35355339059] transform)% Lower left
(!#1 #2 [0.7071067811  0 0 0.7071067811  0.35355339059 -0.35355339059] transform)% Lower left
(!#1 #2 [0.7071067811  0 0 0.7071067811  0.35355339059  0.35355339059] transform) % Top left
}

\newcommand{\Youngrow}[2]{%
\multido{\n=1+1}{#2}{%
\Youngbox{\n}{#1}
}
}

\newcommand{\GraphPoint}[2]{%
\pscircle[fillstyle=solid,fillcolor=white](#1,#2){0.23}
}

\title[Decomposing Inversion Sets and the Littlewood-Richardson Cone]{Decomposing Inversion Sets of Permutations and Applications to Faces of the Littlewood-Richardson Cone}

\author[Dewji]{R. Dewji}
\address{Department of Mathematics and Statistics
Queen's University \\ King\-ston, Ontario, Canada \\ K7L 3N6
}
\email{rian.dewji@gmail.com}

\author[Dimitrov]{I. Dimitrov}
\address{Department of Mathematics and Statistics
Queen's University \\ King\-ston, Ontario, Canada \\ K7L 3N6
}
\email{dimitrov@mast.queensu.ca}

\author[McCabe]{A. McCabe}
\address{760 Lawrence Ave. West\\ Toronto, Ontario, Canada\\ M6A 1B7}
\email{adam.r.mccabe@gmail.com}

\author[Roth]{M. Roth}
\address{Department of Mathematics and Statistics
Queen's University \\ King\-ston, Ontario, Canada \\ K7L 3N6
}
\email{mikeroth@mast.queensu.ca}

\author[Wehlau]{D. Wehlau}
\address{Department of Mathematics and Computer Science \\
Royal Military College \\ King\-ston, Ontario, Canada \\ K7K 5L0
}
\email{wehlau@rmc.ca}

\author[Wilson]{J. Wilson}
\address{ Department of Mathematics \\
Stanford University \\
Building 380, Sloan Hall \\
Stanford, California 94305}
\email{jchw@stanford.edu}

\date{\today}
                                          % Activate to display a given date or no date

\begin{document}
\maketitle

\begin{abstract}
If $\alpha \in S_n$ is a permutation of $\{1, 2, \ldots, n\}$, the inversion set of $\alpha$ is
$\Phi(\alpha) = \{ (i, j) \, | \,  1 \leq i < j \leq n, \alpha(i) > \alpha(j)\}$. We describe all $r$-tuples
$\alpha_1, \alpha_2, \ldots, \alpha_r \in S_n$ such that $\Delta_n^+ = \{ (i, j) \, | \,  1 \leq i < j \leq n\}$ is
the disjoint union of $\Phi(\alpha_1), \Phi(\alpha_2), \ldots, \Phi(\alpha_r)$. Using this description we
prove that certain faces of the Littlewood-Richardson cone are simplicial and provide an algorithm for
writing down their sets of generating rays. We also discuss analogous problems for
the Weyl groups of root systems of types $B$, $C$ and $D$ providing solutions for types $B$ and $C$.
Finally we provide some enumerative results and introduce a useful tool for visualizing inversion sets.

\end{abstract}

\section{Introduction}\label{Introduction}

\mysubsection{} Given a positive integer $n$, we set
$$
\Delta_n^+ := \{(i, j) \ | \, 1 \leq  i< j \leq n\}.
$$
In accordance with terminology from Lie Theory, we will refer to the elements of $\Delta^+_n$ as {\it positive roots},
the element $(1,n)$ is the {\it highest root}, and the elements $(i,i+1)$ with $1\leq i \leq n-1$ are the {\it simple roots}.

We describe an element $\alpha \in S_n$ as a function, writing
$\alpha=(\alpha(1),\alpha(2),\dots,\alpha(n))$, and define {\it the inversion set of $\alpha$}, $\Phi(\alpha)$, by
$$
\Phi(\alpha) := \{ (i,j) \in \Delta_n^+ \, | \, \alpha(i) > \alpha(j)\}.
$$
We use
$\Id_n$ for the identity permutation: $\Id_n = (1,2,\dots,n)$, and $\wo[n]$ for the longest element:
$\wo[n] = (n,n-1,\dots,1) \in S_n$. (The term ``longest element'' also comes from Lie Theory.)
Note that
$\Inv{\Id_n} = \emptyset$ and
$\Inv{\wo[n]} = \Delta_n^+$.
It is not hard to see that the element $\alpha \in S_n$ is determined by its inversion set $\Inv\alpha$.
Thus there are exactly $n!$ subsets of $\Delta_n^{+}$ which are inversion sets.

Throughout the paper we use the following notational conventions. We use the symbol $\sqcup$ to denote a disjoint union.
Often we will write $\Id$, $\wo$, $\Delta^+$, etc. instead of  $\Id_n$, $\wo[n]$, $\Delta^+_n$, etc. when the value of $n$ is clear from the context.

\begin{definition}
A {\it decomposition} of an inversion set $\Inv\alpha$ is a set of disjoint inversion sets
$\Inv{\alpha_1},\Inv{\alpha_2},\dots,\Inv{\alpha_r}$ such that
$$\Inv\alpha = \Inv{\alpha_1} \sqcup \Inv{\alpha_2} \sqcup \dots \sqcup \Inv{\alpha_r}\ .$$
Note that $\Inv{\Id}=\emptyset$ may occur in a decomposition.
The decomposition is called {\it trivial} if $\Inv\alpha=\Inv{\alpha_a}$ for some $a$ with $1 \leq a \leq r$, and
hence that $\alpha_i=\Id$ for all $i\neq a$.
\end{definition}

We say that an element $\alpha \in S_n$ (and its inversion set $\Inv\alpha$) is {\it reducible} if there exists a
non-trivial decomposition of $\Inv\alpha$.
Otherwise we say that $\alpha$ (and $\Inv\alpha$) is {\it irreducible}.
We call a decomposition as above an {\em irreducible decomposition} if each $\Inv{\alpha_i}$ is irreducible.

Solving the following problem was the motivation for this article.
\begin{problem}\label{main problem}
Describe all decompositions of $\Delta_n^{+}$:
$$\Delta_n^+ = \Inv{\alpha_1} \sqcup \Inv{\alpha_2} \sqcup \dots \sqcup \Inv{\alpha_r}\ .$$
The order of the inversion sets above is  irrelevant.
\end{problem}

\mysubsection{}
We are interested in this problem because of its relation to two other problems:
\begin{enumerate}
\item[(i)] determining the regular codimension $n$ faces of the Littlewood-Richardson cone;
\item[(ii)] studying the cup product of the cohomology of line bundles on homogeneous varieties.
\end{enumerate}
%We briefly describe the relation to the Littlewood-Richardson cone in the next paragraph.
%More details and the relation of Problem~\ref{main problem}
%to the cup product of line bundles appear in \S\ref{Motivation}.

We briefly describe these two problems in the next paragraphs.

\medskip
\noindent
{\bf The Littlewood-Richardson cone.}
If $A$ is a Hermitian matrix, denote by $\lambda = (\lambda_1 \geq \lambda_2 \geq \ldots
\geq \lambda_n) \in \R^n$ its eigenvalues and let
$\R^{3n}_+ = \{(\lambda, \mu, \nu) \, | \, \lambda_i \geq \lambda_{i+1}, \mu_i \geq \mu_{i+1}, \nu_i \geq \nu_{i+1} {\text { for }} 1 \leq i \leq n-1\}$.
In 1912 H. Weyl posed the following question: For which triples $(\lambda, \mu, \nu) \in \R^{3n}_+$ do there
exist Hermitian matrices $A,B,C$ such that $C = A+ B$ and whose eigenvalues are $\lambda, \mu,\nu$ respectively. In 1962 A. Horn proved that
the set of such triples is a polyhedral cone $\mathcal{C'}$ and conjectured inequalities determining $\mathcal{C'}$.
Horn's conjecture was proved in the 1990's by
Klyachko and Knutson and Tao, see \cite{F} for a nice exposition on Horn's conjecture.  It is worth mentioning that the lattice points of $\mathcal{C'}$
are exactly the triples $(\lambda, \mu, \nu)$ for which the corresponding Littlewood-Richardson coefficient $c_{\lambda, \mu}^\nu$
is nonzero. It is often convenient to study the cone $\mathcal{C''}$ corresponding to the relation $A+B+C = 0$ instead of $\mathcal{C'}$
corresponding to $C = A+B$ thus symmetrizing the roles of $\lambda, \mu$, and $\nu$.
N. Ressayre \cite{R} described all {\it regular faces of $\mathcal{C''}$}, i.e. faces that intersect the interior of $\R^{3n}_+$.
The regular faces of codimension $n$ are in a bijection with triples $\alpha_1, \alpha_2, \alpha_3$ of elements of $S_n$
with the property that $\Delta_n^+ = \Inv{\alpha_1} \sqcup \Inv{\alpha_2} \sqcup \Inv{\alpha_3}$: each regular face of codimension $n$ is
the intersection of $\R^{3n}_+$ with the subspace of codimension $n$ defined by
$$
\alpha_1^{-1} \lambda +\alpha_2^{-1} \mu + \alpha_3^{-1} \nu = 0
$$
for the corresponding triple $\alpha_1, \alpha_2, \alpha_3$.

\medskip
\noindent
{\bf Cup products of line bundles on homogeneous varieties.} Let $G = GL_n(\C)$, let $B \subset G$ be a Borel
subgroup, and let $X = G/B$.  The Picard group of $X$ is isomorphic to $\Z^n$ and hence the line bundles on $X$
are parametrized by $\Z^n$. We denote by $\mathcal{L}_\lambda$ the line bundle on $X$ which corresponds to the
$B$-character $-\lambda$.  We call $\lambda \in \Z^n$ {\it dominant} if
$\lambda_1 \geq \lambda_2 \geq \ldots \geq \lambda_n$ and {\it strictly dominant} if
$\lambda_1 > \lambda_2 > \ldots > \lambda_n$.  Let $\rho = (n-1, n-2, \ldots, 0) \in \Z^n$.
We call $\lambda \in \Z^n$ {\it regular} if there exists $\alpha \in S_n$ such that
$\alpha \cdot \lambda := \alpha(\lambda + \rho) - \rho$ is dominant. Such an element $\alpha$ is uniquely
determined by $\lambda$ and we denote it by $\alpha_\lambda$.
The celebrated Borel-Weil-Bott theorem calculates the cohomology groups $H^q(X, \mathcal{L}_\lambda)$.
In particular, it states that
$H^q(X, \mathcal{L}_\lambda)$ is zero unless $\lambda$ is regular and $q$ equals the length of $\alpha_\lambda$.
In this case,
$H^q(X, \mathcal{L}_\lambda) \cong V(\alpha_\lambda \cdot \lambda)^*$, where for any dominant weight $\mu$,
$V(\mu)$ denotes the irreducible $G$-module with highest weight $\mu$.
In \cite{DR} two of us studied the following question: For what pairs $\lambda, \mu \in \Z^n$ is the cup product map
$$
H^{q_1}(X, \mathcal{L}_\lambda) \otimes H^{q_2}(X, \mathcal{L}_\mu) \stackrel{\cup}{\longrightarrow} H^{q_1+q_2}(X, \mathcal{L}_{\lambda + \mu})
$$
nonzero provided that the all cohomology groups above are nonzero? Theorem~I in \cite{DR} states that
the cup-product map above is non-zero if and only if
$\Phi(\alpha_{\lambda+ \mu}) = \Phi({\alpha_\lambda}) \sqcup \Phi({\alpha_\mu})$.
This turns out (see Lemma~\ref{complement lemma}) to be equivalent to the condition that
$\Delta_n^{+} = \Phi({\alpha_\lambda}) \sqcup \Phi({\alpha_\mu}) \sqcup \Phi(\wo[n]\alpha_{\lambda+ \mu})$.

Both of these motivating problems have versions involving an arbitrary number of factors, (i.e., the
sum of $r$ matrices, or the cup product of $r$ cohomology groups), and their solutions are similarly
expressed as decompositions of $\Delta_{n}$ with $r+1$ factors.   We were thus led to consider
Problem~\ref{main problem}.

\mysubsection{} \label{basic defs section}
Before we state the main results of the paper, we introduce some concepts and state background results.
\begin{definition}
An {\it interval} (of size $t$) is a set of consecutive integers $\{i,i+1,i+2,\dots,i+t-1\}$.
For  a permutation $\alpha \in S_n$, a  {\it block} (of size $t$) of
$\alpha$ is an interval $\{i,i+1,i+2,\dots,i+t-1\}$ of size $t$ such that the set
$\{\alpha(i),\alpha(i+1),\dots,\alpha(i+t-1)\}$ is also an interval (of size $t$).
Every permutation in $S_n$ has $n$ blocks of size 1 and a block of size $n$.
If $\alpha \in S_n$ has no blocks of size $t$ for all $1 < t < n$ then we say that $\alpha$ is
{\it simple}\footnote{We warn the reader that some authors use the terminology {\it connected} rather than simple.}.
\end{definition}

\begin{example}
The permutation $(9,7,1,5,3,4,6,8,2)\in S_9$ has a block of size 8 corresponding to the interval $\{2,3,4,5,6,7,8,9\}$ and a block
size 4 corresponding to the interval $\{4,5,6,7\}$.  The permutation $(5,2,6,1,4,7,3)\in S_7$ has no non-trivial blocks and so is
simple.
\end{example}

To state our results we need to introduce an inflation procedure to describe permutations inductively.
We describe this procedure heuristically as follows.   We consider a permutation on $n$ letters as a shuffling of a
deck of $n$ cards.
To shuffle, we first cut the deck into $m$ piles of sizes $z_1,z_2,\dots,z_m$ respectively.
Shuffle each of these piles according to a permutation $\beta_i \in S_{z_i}$.
Finally reassemble the piles in an order determined by a permutation $\sigma \in S_m$.
The resulting permutation in $S_n$ is denoted by $\sigma[\beta_1,\beta_2,\dots,\beta_m]$ and
is called an {\em inflation of $\sigma$ by $\beta_1, \beta_2, \ldots, \beta_m$}.
For a formal characterization of inflation see \S\ref{section inflation};
for a history of the inflation procedure and a discussion of a number of applications we refer the reader
to the survey article of Brignall \cite{B}.

Note that a permutation $\alpha\in S_n$ is simple if and only if $\alpha$ cannot be expressed as
an inflation $\alpha=\sigma[\beta_1,\beta_2,\dots,\beta_m]$ with $\sigma\in S_m$ and $2 \leq m \leq n-1$.

\begin{definition}
A permutation $\alpha \in S_n$ is called {\it plus-decomposable} if $\alpha$ may be written in the form
$\alpha = \Id_2[\beta_1,\beta_2]$.  Otherwise $\alpha$ is {\it plus-indecomposable}.
Similarly, $\alpha\in S_n$ is called {\it minus-decomposable} if $\alpha$ may be written in the form
$\alpha= \wo[2][\beta_1,\beta_2]$.  Otherwise $\alpha$ is {\it minus-indecomposable}.
\end{definition}

 We follow \cite{AAK} in using the terms ``plus-indecomposable'' and ``minus-indecomposable''.  These are called
  ``sum indecomposable'' and ``skew indecomposable'', respectively, in \cite{B}. It is not difficult to verify that
  $\alpha \in S_n$ cannot be both plus-decomposable and minus-decomposable. On the other hand, there are
  permutations which are both plus-indecomposable and minus-indecomposable, e.g. every simple $\alpha \in S_n$
with $n>2$.

The following theorem of Albert, Atkinson and Klazar illustrates the importance of simple permutations and the
inflation procedure.

\begin{theorem}[{\cite[Theorem~1]{AAK}}]
Let $n \geq 2$.  For every permutation $\alpha \in S_n$ there exists a simple permutation $\sigma \in S_m$ and
permutations $\beta_1,\beta_2,\dots,\beta_m$ such that $\alpha = \sigma[\beta_1,\beta_2,\dots,\beta_m]$.
Moreover if $\sigma \ne \Id_2$ and $\sigma\neq \wo[2]$ then $\beta_1,\beta_2,\dots,\beta_m$ and $\sigma$ are unique.
If $\sigma = \Id_2$ then $\beta_1,\beta_2$ and $\sigma$ are unique if we add the additional condition that
$\beta_1$ is plus-indecomposable.
Similarly, if $\sigma = \wo[2]$ then $\beta_1,\beta_2$ and $\sigma$ are unique if we add the additional condition that
$\beta_1$ is minus-indecomposable.    \qed
\end{theorem}

For our purposes, we modify the statement of the above theorem as follows.
\begin{theorem}\label{simple form theorem}
Let $n \geq 2$.  For every permutation $\alpha \in S_n$ there exists a permutation $\sigma \in S_m$ and
permutations $\beta_1,\beta_2,\dots,\beta_m$ such that $\alpha = \sigma[\beta_1,\beta_2,\dots,\beta_m]$
where either $\sigma$ is simple and $m \geq 4$ or $\sigma = \Id_m$ or $\sigma=\wo[m]$.  Furthermore this
expression for $\alpha$ is unique if we require that $m$ be maximal when $\sigma=\Id_m$ or $\sigma=\wo[m]$,
i.e., that each $\beta_b$ is plus-indecomposable when $\sigma=\Id$ and
each $\beta_b$ is minus-indecomposable when $\sigma=\wo$.   \qed
\end{theorem}

\begin{corollary} \label{new corollary}
In the notation of Theorem~\ref{simple form theorem} above $\sigma = I_m$, $\sigma =J_m$ or $\sigma$ is simple and $m \geq 4$
if and only if $\alpha$ is plus-decomposable, $\alpha$ is minus-decomposable or $\alpha$ is both plus-indecomposable and minus-indecomposable
respectively.
\end{corollary}

\begin{definition}
We say that $\alpha$ is expressed in {\it simple form} when we write $\alpha = \sigma[\beta_1,\beta_2,\dots,\beta_m]$
in the form guaranteed by Theorem~\ref{simple form theorem}, i.e, when $\sigma$ is simple with $m \geq 4$
or $\sigma=\wo[m]$ or $\sigma = \Id_m$ with $m$ maximal.
\end{definition}

\mysubsection{} \label{subsection with main result}
We are now ready to state our main result. The inflation procedure described above allows us to  inductively construct
decompositions of $\Delta_n^+$ into inversion sets. Assume that the set $\{1, 2, \ldots, n\}$ is
partitioned into $m$ intervals of lengths $z_1, z_2, \ldots, z_m$ and let $\sigma_a \in S_m$, and
$\beta_{ab} \in S_{z_b}$ ($1 \leq a \leq r$, $ 1 \leq b \leq m$)
be such that
$$\begin{array}{rcl}
\Delta_m^+ &=& \Phi(\sigma_1) \sqcup \Phi(\sigma_2) \sqcup \ldots \sqcup \Phi(\sigma_r),  \\
&&\\
\Delta^+_{z_1} & = & \Inv{\beta_{11}} \sqcup \Inv{\beta_{21}} \sqcup \dots \sqcup \Inv{\beta_{r1}}, \\
\Delta^+_{z_2} & = & \Inv{\beta_{12}} \sqcup \Inv{\beta_{22}} \sqcup \dots \sqcup \Inv{\beta_{r2}},\\
&\vdots&\\
\Delta^+_{z_m} & = & \Inv{\beta_{1m}} \sqcup \Inv{\beta_{2m}} \sqcup \dots \sqcup \Inv{\beta_{rm}}.
\end{array}
$$
Define $\alpha_1, \alpha_2, \ldots \alpha_r \in S_n$ by
\begin{eqnarray*}
\alpha_1  &= & \sigma_1[\beta_{11},\beta_{12},\dots,\beta_{1m}],\\
\alpha_2  &= & \sigma_2[\beta_{21},\beta_{22},\dots,\beta_{2m}],\\
& \vdots &\\
\alpha_r  &= & \sigma_r[\beta_{r1},\beta_{r2},\dots,\beta_{rm}].
\end{eqnarray*}
It then follows that
$$
\Delta_n^+ = \Inv{\alpha_1} \sqcup \Inv{\alpha_2} \sqcup \dots \sqcup \Inv{\alpha_r}.
$$

Our main result below is that every decomposition of $\Delta_n^+$ into inversion sets can be constructed in such a way.
Moreover, we    identify a canonical way of  representing $\alpha_1, \alpha_2, \ldots, \alpha_r$ as inflations
which identifies the decomposition
$\Delta_n^+ = \Inv{\alpha_1} \sqcup \Inv{\alpha_2} \sqcup \dots \sqcup \Inv{\alpha_r}$ uniquely.

\begin{theorem}\label{main theorem}
Suppose $\alpha_1, \alpha_2,\dots,\alpha_r \in S_n$ and
$$
\Delta_n^+ = \Inv{\alpha_1} \sqcup \Inv{\alpha_2} \sqcup \dots \sqcup \Inv{\alpha_r}
$$
with all $\Inv{\alpha_a} \neq \emptyset$.
  Without loss of generality assume that the highest root $(1,n) \in \Inv{\alpha_1}$.
Let $\alpha_1 = \sigma_1[\beta_{11},\beta_{12},\dots,\beta_{1m}]$ be the simple form expression of $\alpha_1$
with a corresponding partition of the set $\{1,2, \ldots, n\}$ into $m$ intervals of lengths $z_1, z_2, \ldots, z_m$.
Then, up to reordering of $\alpha_2, \alpha_3, \ldots, \alpha_r$,
there exists  a unique set of  elements $\sigma_a\in S_m$ and $\beta_{ab}\in S_{z_b}$
such that $\alpha_a = \sigma_a[\beta_{a1},\beta_{a2},\dots,\beta_{am}]$, for $a=2,\dots,r$, $b=1,2,\dots,m$ and
\begin{itemize}
\item[(i)] $$\begin{array}{rcl}
\Delta_m^+ &=& \Phi(\sigma_1) \sqcup \Phi(\sigma_2) \sqcup \ldots \sqcup \Phi(\sigma_r),  \\
&&\\
\Delta^+_{z_1} & = & \Inv{\beta_{11}} \sqcup \Inv{\beta_{21}} \sqcup \dots \sqcup \Inv{\beta_{r1}}, \\
\Delta^+_{z_2} & = & \Inv{\beta_{12}} \sqcup \Inv{\beta_{22}} \sqcup \dots \sqcup \Inv{\beta_{r2}},\\
&\vdots&\\
\Delta^+_{z_m} & = & \Inv{\beta_{1m}} \sqcup \Inv{\beta_{2m}} \sqcup \dots \sqcup \Inv{\beta_{rm}};
\end{array}
$$
\item[(ii)]   if $\alpha_1$ is minus-decomposable then $\sigma_1 = \wo[m]$
and $\sigma_2=\sigma_3=\dots=\sigma_r=\Id_m$;
\item[(iii)]  if $\alpha_1$ is minus-indecomposable then $\sigma_1$ is simple and $\sigma_2 = \wo \sigma_1$,
and $\sigma_3=\sigma_4=\dots=\sigma_r=\Id_m$.
\end{itemize}
In particular, $\sigma_1$ and at most one other of the $\sigma_a$ are not the identity.\\
Let $q$ denote the number of $\sigma_a$ which are not $\Id_m$, i.e.,
$q := \begin{cases} 1, &\text{if }\alpha_1 \text{ is minus-decomposable};\\
2, &\text{if }\alpha_1 \text{ is minus-indecomposable}.
\end{cases}$

Then, the above decomposition of $\Delta_n^+$ is irreducible if and only if the following four conditions hold
\begin{itemize}
\item[(i)] each of the decompositions
$\Delta^+_{z_b}  =  \Inv{\beta_{1b}} \sqcup \Inv{\beta_{2b}} \sqcup \dots \sqcup \Inv{\beta_{rb}}$
is irreducible;
\item[(ii)] exactly one of of $\beta_{a1},\beta_{a2},\dots,\beta_{am}$ is not equal to the identity for $a=q+1,\dots, r$;
\item[(iii)]  $\beta_{ab} = \Id_{z_b}$ for $a=1,\dots, q$ and $b=1,\dots, m$;
\item[(iv)]  $m = 2$ if $\alpha_1$ is minus-decomposable.
\end{itemize}
\end{theorem}

\begin{example}\label{example}
Let $n=8$ and let $\alpha_1 = (4,5,6,1,7,8,3,2)$, $\alpha_2 = (5,3,4,8,1,2,6,7)$,  $\alpha_3 = (1,3,2,4,6,5,7,8)$.
Here $q=2$ since $\alpha_1$ is minus-indecomposable.
Then  $\Delta_8^+ = \Inv{\alpha_1} \sqcup \Inv{\alpha_2} \sqcup \Inv{\alpha_3}$,
$m=4$, and
$$
\begin{array}{rcl}
\alpha_1 &=& (3,1,4,2)[(1,2,3),(1),(1,2),(2,1)]\\
\alpha_2 &=& (2,4,1,3)[(3,1,2),(1),(1,2),(1,2)]\\
\alpha_3 &=& (1,2,3,4)[(1,3,2),(1),(2,1),(1,2)].
\end{array}
$$
Note that $\beta_{11} = I_3$, $\beta_{13} = \beta_{23} = I_2$ and $\beta_{24} = \beta_{34} = I_2$. Consequently,
$\Delta_3^+ = \Inv{\beta_{11}} \sqcup \Inv{\beta_{21}} \sqcup \Inv{\beta_{31}} =
\Inv{\beta_{21}} \sqcup \Inv{\beta_{31}}$,
$\Delta_2^+ =  \Inv{\beta_{13}} \sqcup \Inv{\beta_{23}} \sqcup \Inv{\beta_{33}} = \Inv{\beta_{33}}$
and
$\Delta_2^+ =  \Inv{\beta_{14}} \sqcup \Inv{\beta_{24}} \sqcup \Inv{\beta_{34}} = \Inv{\beta_{14}}$.
This decomposition is not irreducible: it fails condition (ii) for irreducibility since $\beta_{31} \neq I_3$ and $\beta_{33} \neq I_2$;
it also fails condition (iii) since $\beta_{14} \neq I_2$ and $\beta_{21} \neq I_3$.
\end{example}

The recursive form of this theorem allows us to inductively solve many problems concerning decompositions.
For example, in \S\ref{enumerative section} we exploit this recursiveness to obtain a number of results
enumerating various solutions to the main problem.
In \S\ref{Motivation} we use the form to prove a result about the decompositions which yields an algorithm
producing all generating rays on a given regular codimension $n$ face of the Littlewood-Richardson cone.

\mysubsection{}
The problem discussed above has a Lie theoretic background and a natural generalization.
We recommend the book by Fulton and Harris,~\cite{FH} as a general Lie Theory reference.
Let $\Delta$ be a root
system with corresponding Weyl group $\W$. Fix a splitting $\Delta = \Delta^+ \sqcup \Delta^-$ of $\Delta$ into positive and negative roots.
For $\alpha \in \W$, the inversion set of $\alpha$, $\Inv{\alpha}$ is defined by
$\Inv{\alpha} := \{ v \in \Delta^+ \mid \alpha\cdot v \in \Delta^-\}$.  We are concerned with ways to express the positive roots
as a disjoint union of inversion sets: $\Delta^+ =   \Inv{\alpha_1} \sqcup \Inv{\alpha_2} \sqcup \dots \sqcup \Inv{\alpha_r}$
where $\alpha_1,\alpha_2,\dots,\alpha_r \in \W$.

Problem~\ref{main problem}, which is solved by Theorem~\ref{main theorem}, is
the $A_{n-1}$-case of the more general problem for arbitrary root systems.
(Both of the motivating problems also have versions for arbitrary root systems, and their solution is again
in terms of decompositions of the positive roots into inversion sets, i.e., the more general problem.)
It is natural to attempt to solve the general problem for all root systems. Section \S\ref{section BCD}
is devoted to studying the root systems of types $B, C$ and $D$.
We provide a solution for root systems
of types  $B$ and $C$. Root systems of types $B$ and $C$ have isomorphic Weyl groups and so yield
identical answers to our questions; nevertheless we consider them separately because this
gives us two different ways of looking at the same problem.
Root systems of type $D$ are more complicated and we only provide a brief discussion of the difficulties we encountered when
attempting to deal with them.
 The exceptional root systems are also interesting but our methods are unlikely to yield any results.
The solution of Problem~\ref{main problem} for root systems of type $G_2$ is elementary:  all nontrivial
decompositions are of the form $\Inv{\alpha}\sqcup\Inv{J\alpha}$, where $J\in \W(G_2)$ is the longest element
of $\W(G_2)$.  The root system $F_4$ is probably easily treated by direct computations
(possibly aided by a computer).  Root systems of type $E$, especially $E_8$, may be too complicated to
treat even by computer computations.

In \S\ref{single section} we use Theorem~\ref{main theorem} to give a solution, in the form of an algorithm, to
the following natural variation of Problem~\ref{main problem}:
Given $\alpha \in S_n$, describe all decompositions
$$
\Inv{\alpha}  = \Inv{\alpha_1} \sqcup \Inv{\alpha_2}.
$$

\mysubsection{} {\bf Organization of the paper.} In \S\ref{Preliminaries} we establish some basic results on
inversion sets. In \S\ref{section inflation} we define restriction maps and use them to establish further
results on simple and irreducible permutations and to prove the main theorem.  Section \S\ref{section symmetric}
discusses symmetric permutations. The results about symmetric permutations are then used in \S\ref{section BCD} to extend the
main theorem to root systems of  types $B$ and $C$.
We then turn to applications of the main theorem.
In \S\ref{enumerative section} we give some enumerative results
deduced from the theorems.
In \S\ref{single section} we give an algorithm to decompose a single inversion set.
In \S\ref{Motivation} we use the main theorem to parameterize regular codimension
$n$ faces of the Littlewood-Richardson cone.  Finally in \S A (an appendix) we describe {\em sign diagrams},
a visual method of displaying inversion sets which has proved useful to us in thinking about these problems.

%%%%%%%%%%%%%%%%%%%%%%%%%%%%%%%%%%%%%%%%%%%%%%%%%
\section{Preliminaries on inversion sets}\label{Preliminaries}
It is easy to see that an inversion set $\Phi$ must satisfy the following two conditions:
\begin{itemize}
\item [(i)] If $(i,j), (j,k) \in \Phi$ then $(i,k) \in \Phi$. ({\it closed condition})
\item [(ii)] If $(i,j), (j,k) \notin \Phi$ then $(i,k) \notin \Phi$.  ({\it co-closed condition})
\end{itemize}
Kostant \cite{K} proved the following statement characterizing inversion sets.

\begin{proposition}[{\cite[Proposition~5.10]{K}}] \label{proposition 11}
A set $\Phi \subset \Delta_n^{+}$ is an inversion set if and only if $\Phi$ is both closed and co-closed, i.e.,
both $\Phi$ and its complement $\Delta^{+}_n\setminus \Phi$ satisfy the closed condition.
\end{proposition}

The following simple result is often useful.
%   An illustration of part (2) of the lemma
%can be found in \S\ref{young diagram subsection}.
\begin{lemma}%[{\bf Structure of inversion sets with one simple root}]
\label{simple root lemma}
%\begin{enumerate}
%\item
Every non-empty inversion set $\Inv{\alpha}$ contains at least one simple root.
%\item If $\alpha$ only inverts one simple root, then contained in a rectangle and looks like a Young diagram.
%\end{enumerate}
\end{lemma}

\begin{proof}
%The first assertion
This
 follows from the fact that if $\alpha(i) < \alpha(i+1)$ for all $i=1,2,\dots,n-1$ then $\alpha=\Id_n$.
\end{proof}

\begin{definition}
The graph of a permutation $\alpha\in S_n$ is the set of $n$ lattice points
$\{(i,\alpha(i))\mid i=1,2,\dots,n\}$ considered as a subset of $[1,n]\times [1,n] \subset \R^2$.
\end{definition}

We have already noted that  $\Inv{\wo} = \Delta^+$.
The following two lemmas give further indication of the importance of $\wo$.

\begin{lemma}\label{complements are simple}
Let $\alpha \in S_n$.  The permutation $\alpha$ is simple if and only if $\wo \alpha$ is simple.
\end{lemma}
\begin{proof}
A block of size $t+1$ for the permutation $\alpha$ corresponds to a $t \times t$ closed square in $[1,n]\times[1,n]$
which contains $t+1$ points of the  graph of $\alpha$.  Hence $\alpha$ is simple if there does not
exist a $t\times t$ closed square in $[1,n]\times[1,n]$ containing $t+1$ points of the graph of $\alpha$ with $2\leq t \leq n-1$.
If the graph of $\alpha$ satisfies this condition then so does the graph of $J\alpha$, which
is obtained from that of $\alpha$ by reflecting in the horizontal line $y=n/2$.  Thus $\alpha$ is simple
if and only if $J\alpha$ is simple.
\end{proof}

To each element $(i,j)\in \Delta_n^{+}$ we associate the line segment joining the points $(i,\alpha(i))$
and $(j,\alpha(j))$ on the graph of $\alpha$.  We note that $(i,j)\in\Inv\alpha$ if and only if the
corresponding line segment has negative slope.

\begin{lemma}\label{complement lemma}
Let $\alpha\in S_n$.  Then $\Delta_n^+ = \Inv{\alpha} \sqcup \Inv{\wo \alpha}$, or equivalently,
$\Inv{\wo\alpha}=\Delta_n^{+}\setminus\Inv\alpha$.
\end{lemma}
\begin{proof}
The graph of $\wo\alpha$ is obtained from the graph of $\alpha$ by reflecting in the horizontal line
$y=n/2$.  Using the characterization of $\Inv\alpha$ as those positive roots whose corresponding
line segment has negative slope completes the proof of the lemma.
\end{proof}

\begin{corollary}\label{when is Jm reducible}
The element $\wo[m]$ is reducible for $m\geq 3$ and irreducible for $m=2$.
\end{corollary}

\begin{proof}
By Lemma~\ref{complement lemma} any
$\tau \in S_m \setminus \{\wo[m],\Id_m\}$ gives a non-trivial decomposition
$\Inv{\wo[m]}=\Delta^+_m =  \Inv\tau \sqcup \Inv{\wo[m]\tau}$, and the set
$S_m \setminus \{\wo[m],\Id_m\}$ is non-empty if $m\geq 3$.  Conversely,
$\Inv{\wo[2]} = \{(1,2)\}$ is clearly irreducible.
\end{proof}

%-----------------------------------------------------------
Next we discuss some basic properties of decompositions.

%  In order to reduce the study of arbitrary decompositions of $\Delta_n^{+}$ to the problem of
%classifying irreducible decompositions we need to handle the problem of reversing this process: given an
%irreducible decomposition how do we know which sets in the decomposition can be merged to form inversion
%sets of elements of $S_n$?
%Fortunately, the solution to this problem is simple.

As the following proposition shows, the union of an arbitrary choice of sets appearing in a decomposition
is always the inversion set of a permutation.

\begin{proposition}\label{coalescence}
Suppose $\Delta_n^+ = \Inv{\alpha_1} \sqcup \Inv{\alpha_2} \sqcup \dots \sqcup \Inv{\alpha_r}$
is a decomposition and let $A$ be any subset of $\{1,2,\dots,r\}$.
Then there exists $\alpha \in S_n$ such that $\Inv\alpha = \sqcup_{a\in A} \Inv{\alpha_a}$.
\end{proposition}

\begin{proof}
Clearly it suffices to prove the assertion for doubleton sets $A=\{p,q\}$.
Thus it suffices to show that  $\Inv{\alpha_p} \sqcup \Inv{\alpha_q}$ is both closed and co-closed.
For ease of notation, we will assume $A = \{1,2\}$.
First we show that $\Inv{\alpha_1} \sqcup \Inv{\alpha_2}$ is co-closed.  Suppose that
$(i,j),(j,k) \notin \Inv{\alpha_1} \sqcup \Inv{\alpha_2}$.  Then for $b=1,2$ we have
$(i,k) \notin \Inv{\alpha_b}$ since $\Inv{\alpha_b}$ is co-closed.
Thus $(i,k) \notin \Inv{\alpha_1} \sqcup \Inv{\alpha_2}$ which shows $\Inv{\alpha_1} \sqcup \Inv{\alpha_2}$
is co-closed.

To see that $\Inv{\alpha_1} \sqcup \Inv{\alpha_2}$ is closed, suppose that
$(i,j), (j,k) \in \Inv{\alpha_1} \sqcup \Inv{\alpha_2}$.  Then  for $a=3,4,\dots,r$ we have
$(i,j),(j,k) \notin \Inv{\alpha_a}$ and thus $(i,k) \notin \Inv{\alpha_a}$ since $\Inv{\alpha_a}$ is co-closed.
Hence $(i,k) \notin \sqcup_{a=3}^r \Inv{\alpha_a}$ which implies that $(i,k) \in \Inv{\alpha_1} \sqcup \Inv{\alpha_2}$.
This shows that that $\Inv{\alpha_1} \sqcup \Inv{\alpha_2}$ is closed and completes the proof of proposition.
\end{proof}

Note that some hypothesis of the type
$\Delta_n^+ = \Inv{\alpha_1} \sqcup \Inv{\alpha_2} \sqcup \dots \sqcup \Inv{\alpha_r}$
is necessary in the above proposition; arbitrary unions of inversion sets need not be inversion sets.
For example, consider $n=3$, $\alpha_1 = (2,1,3)$ and $\alpha_2 = (1,3,2)$.  Then $\Inv{\alpha_1} = \{(1,2)\}$,
$\Inv{\alpha_2} = \{(2,3)\}$ and $\Inv{\alpha_1} \sqcup \Inv{\alpha_2}$ is not closed and so is not an inversion set.

\begin{corollary}\label{coalesce alpha decomposition}
If $\Inv\alpha = \Inv{\alpha_1} \sqcup \Inv{\alpha_2} \sqcup \dots \sqcup \Inv{\alpha_r}$ and $A$ is any subset of
$\{1,2,\dots,r\}$ then there exists $\alpha_A\in S_n$ with $\Inv{\alpha_A} = \sqcup_{a\in A} \Inv{\alpha_a}$.
\end{corollary}

\begin{proof}
Set $\alpha_{r+1}=\wo\alpha$.
By Lemma~\ref{complement lemma} we then have the decomposition
$\Delta_n^{+}=\Inv{\alpha_1}\sqcup\cdots\sqcup \Inv{\alpha_{r+1}}$.
The corollary then follows from Proposition~\ref{coalescence} applied to this decomposition.
\end{proof}

Recall that an element $\alpha\in S_n$ is called reducible if there exists a non-trivial decomposition
$\Inv\alpha = \Inv{\alpha_1}\sqcup\cdots\sqcup\Inv{\alpha_r}$.   By Corollary~\ref{coalesce alpha decomposition}
if $\alpha$ is reducible there exists such a non-trivial decomposition with $r=2$.

Given a decomposition $\Delta_{n}^{+}=\Inv{\alpha_1}\sqcup\cdots\sqcup\Inv{\alpha_r}$ we may further decompose
each $\Inv{\alpha_i}$ until we arrive at an {\em irreducible decomposition}, i.e., a decomposition
$\Delta^+_n = \Inv{\gamma_1} \sqcup \Inv{\gamma_2} \sqcup \dots \sqcup \Inv{\gamma_s}$ where each $\Inv{\gamma_i}$
is irreducible.
This provides a finer decomposition of $\Delta_n^+$ than the one we began with.
    Conversely we may get a coarser decomposition from the original decomposition
    by choosing one or more disjoint subsets
  $A_i \subset \{1,2,\dots,r\}$ and replacing $\sqcup_{a \in A_i} \Inv{\alpha_a}$ by the single inversion set
  $\Inv{\alpha_{A_i}}$ where $\alpha_{A_i}$ is the element whose existance is guaranteed by
  Corollary~\ref{coalesce alpha decomposition}.  Clearly every decomposition of
  $\Delta_n^+$ may be obtained in this manner from some irreducible decomposition.
  For this reason, studying and classifying the irreducible decompositions is of particular interest.
  Accordingly, we characterize the irreducible decompositions in our main theroem, Theorem~\ref{main theorem},
  as well as in the analogous theorems for root systems of type B and C.

\section{Restriction maps and proof of the main theorem}\label{section inflation}

\medskip
Given a subset $\F\subseteq\{1,2,\ldots, n\}$  and an element $\alpha\in S_n$ we obtain a permutation
in $S_m$, with $m=|\F|$, by noting how $\alpha$ changes the relative order of elements of $\F$.  This procedure
gives rise to a map of sets $\theta_\F\colon S_n\longrightarrow S_m$ called a {\em restriction map}.
Although not homomorphisms, the maps $\theta_{\F}$ are useful in making inductive arguments on inversion sets.
In this section we use restriction maps to establish several results on simple and irreducible permutations,
culminating in a proof of the main theorem (Theorem~\ref{main theorem}).

We start by giving formal descriptions of the restriction maps and the inflation procedure.

\begin{definitions}

\renewcommand{\theenumi}{{\em \alph{enumi}}}
\begin{enumerate}
\item
Two sequences $x_1,x_2,\dots,x_m$ and $y_1,y_2,\dots,y_m$ each comprised of $m$ distinct real numbers
are {\em order isomorphic} if $x_i > x_j$ if and only if $y_i > y_j$.

\smallskip
\item
Suppose ${\F}$ is a subset of  $\{1,2,\dots,n\}$ of size $m = |\F|$, and write
$\F=\{i_1,i_2,\dots,i_m\}$ where $i_1 < i_2 < \dots < i_m$.  For any $\alpha\in S_n$
restricting $\alpha$ to $\F$
yields a sequence $\alpha(i_1),\alpha(i_2),\dots,\alpha(i_m)$ which is order isomorphic to the sequence
$\mu(1),\mu(2),\dots,\mu(m)$ corresponding to a unique element $\mu \in S_m$.  We denote this element
$\mu$  by $\mu = \theta_\F(\alpha)$ and use $\theta_{\F}\colon S_n\longrightarrow S_m$ for the corresponding
map of sets.

\smallskip
\item
For $\F \subseteq \{1,2,\dots,n\}$ we write  $\Delta^+_\F$ to denote the
set $\Delta^+_\F := \{(i,j) \in \Delta^+_n \mid i,j \in \F\}$.

\smallskip
\item A decomposition of $\{1,\ldots, n\}$ into an {\em ordered disjoint union of intervals} is a decomposition
$\{1,2,\dots,n\} = U_1 \sqcup U_2 \sqcup \dots \sqcup U_m$ where each $U_i$ is an interval and where for each
$1\leq i < j\leq m$, we have $a<b$ for each $a\in U_i$ and $b\in U_j$.

\smallskip
\item
Given a decomposition of
$\{1,\ldots, n\}$ into an  ordered disjoint union of intervals as above, a subset
$\F \subset \{1,2,\dots,n\}$ is {\em admissible} if $|\F \cap U_i|=1$ for all $i=1,2,\dots,m$.

\end{enumerate}
\end{definitions}

Note that the condition of being admissible in ({\em e}) depends on the choice of decomposition into
ordered disjoint intervals.  In every case we use this term will be careful to make the choice of decomposition
explicit.

Suppose now that we are given a decomposition of $\{1,\ldots, n\}$ into ordered disjoint intervals
$U_1$,\ldots, $U_m$.   Choose $\sigma\in S_m$ and $\beta_i\in S_{|U_i|}$ for $i=1$,\ldots, $m$.
In addition to the description by shuffling cards given in \S\ref{basic defs section},
the inflation $\alpha := \sigma[\beta_1,\beta_2,\dots,\beta_m] \in S_n$, $n=\sum |U_i|$,
is characterized by the following two conditions:
\begin{enumerate}
\item   $\theta_{U_i}(\alpha) = \beta_i$ for all $i=1,2,\dots,m$.
\item   $\theta_\F(\alpha)=\sigma$ for any admissible $\F$.
\end{enumerate}

\noindent
We will frequently use the second fact, which allows us to recover $\sigma$ using any admissible subset $\F$.

%{\bf [The material from here through to Corollary~\ref{easy corollary} fits thematically in \S2.  The only reason
%it appears here is because Lemma~\ref{easy lemma} refers to inflation, and the second description of inflation
%appears above.  However we don't prove this lemma, and so nothing about this description is used.   Should we
%move this material?  ]}

The following lemma, computing the inversion set of $\alpha=\sigma[\beta_1,\ldots, \beta_m]$ from those
of its components, follows from either of the descriptions of the inflation procedure.
The reader may find an illustration of this lemma and the arguments for its proof in \S\ref{inflation appendix}.

\begin{lemma}\label{easy lemma}
Suppose $\alpha=\sigma[\beta_1,\beta_2,\dots,\beta_m]\in S_n$ where $U_1$,\ldots, $U_m$ is a decomposition of
$\{1,\ldots, n\}$ into ordered disjoint intervals, $\sigma\in S_m$, and $\beta_i\in S_{|U_i|}$ for $i=1,\ldots, m$.
Let $\Psi_i$ denote the order preserving bijection $\Psi_i : U_i \to \{1,2,\dots,|U_i|\}$.
Then
$$\Inv\alpha = \{(a,b) \mid a \in U_{i}, b \in U_{j}, (i,j) \in \Inv{\sigma}\}  \sqcup
\left(\sqcup_{i=1}^m\Psi_i^{-1}(\Inv{\beta_i})\right)\ .$$
\end{lemma}

If a permutation $\sigma$ is reducible it is clear that any inflation $\sigma[\Id_{z_1},\ldots, \Id_{z_m}]$
(for any positive integers $z_1$,\ldots, $z_m$) is also reducible: one simply takes a decomposition of $\Inv\sigma$
and inflates the permutations which appear.  However, it is not immediately clear that an inflation of an
irreducible element remains irreducible; apriori it seems that there could be decompositions of the inflation
which do not respect the inflation structure, and therefore do not come from decompositions of the original $\sigma$.
That this can never happen is a consequence of the following more precise statement.

\begin{lemma}\label{inflate inversion set lemma}
For any $\sigma\in S_m$, and any positive integers $z_1$,\ldots, $z_m$, inflation of the
decompositions of $\Inv\sigma$ gives a one-to-one correspondence between the
decompositions of $\Inv\sigma$ and the decompositions of $\Inv{\sigma[\Id_{z_1},\ldots, \Id_{z_m}]}$.
\end{lemma}

\begin{proof}
Set $\alpha=\sigma[\Id_{z_1},\ldots, \Id_{z_m}]$ and $n=\sum z_i$.
We must show that for any decomposition $\Inv{\alpha}=\Inv{\alpha_1}\sqcup\cdots \sqcup \Inv{\alpha_r}$ there are
unique $\sigma_1$,\ldots, $\sigma_r\in S_m$ such that
$\alpha_k=\sigma_k[\Id_{z_1},\ldots, \Id_{z_m}]$ for $k=1$,\ldots, $r$. Lemma~\ref{easy lemma} then implies that
$\Inv{\sigma}=\Inv{\sigma_1}\sqcup \cdots \sqcup \Inv{\sigma_r}$.

Let $\{1,2,\dots,n\} = U_1 \sqcup U_2 \sqcup \dots \sqcup U_m$ be the decomposition into ordered disjoint
intervals corresponding to the inflation $\sigma[\Id_{z_1},\Id_{z_2},\dots,\Id_{z_m}]$
By Lemma~\ref{easy lemma},
we have $\Inv{\alpha} = \{(a,b) \in \Delta_n^+ \mid a \in U_{i}, b \in U_{j}, (i,j) \in \Inv{\sigma}\}$.

Choose any root $(i,j)\in \Inv\sigma$. The fact that no root $(a,a')$ with $a,a'\in U_i$ is in $\Inv\alpha$ means
that no such root is in $\Inv{\alpha_1}$,\ldots,  $\Inv{\alpha_r}$, and so each of $\alpha_1$,\ldots,  $\alpha_r$
preserves the relative order of the elements in $U_i$.  Similarly each of $\alpha_1$, \ldots,  $\alpha_r$
preserves the relative order of elements in $U_j$.

Let $a_0$ be the smallest element in $U_i$ and $b_1$ the largest element in $U_j$.  The root $(a_0,b_1)$ is
in $\Inv\alpha$ and so must be contained in one of $\Inv{\alpha_1}$,\ldots, $\Inv{\alpha_r}$.
Suppose that $(a_0,b_1)\in \Inv{\alpha_k}$, i.e., that $\alpha_k(b_1) < \alpha_k(a_0)$.
Then the fact that $\alpha_k$ preserves the relative order of the elements in $U_i$ and $U_j$ now implies
that $U_i\times U_j:=\{(a,b)\mid a\in U_i, b\in U_j\} \subseteq\Inv{\alpha_k}$.
Since the decomposition of $\Inv{\alpha}$ is into disjoint subsets, we therefore have
$U_i\times U_j\cap \Inv{\alpha_{\ell}}=\emptyset$ if $\ell\neq k$.

For each $k=1$,\ldots, $r$ set
$T_k=\{(i,j)\in \Inv\sigma \mid U_i\times U_j\subseteq \Inv{\alpha_k}\}$.   We have just shown that
for any $(i,j)\in \Inv{\sigma}$ there is a unique $k$ such that
$U_i\times U_j\cap \Inv{\alpha_k}\neq \emptyset$, and for that $k$ we have
$U_i\times U_j\subseteq \Inv{\alpha_k}$.  From this we conclude first that $\Inv{\sigma}=T_1\sqcup \cdots \sqcup T_k$,
and second, since $\Inv{\alpha}=\cup_{(i,j)\in \Inv{\sigma}} U_i\times U_j$, that
$\Inv{\alpha_k}=\cup_{(i,j)\in T_k} U_i\times U_j$ for each $k$.

The fact that $\Inv{\alpha_k}$ is
both closed and co-closed implies that the same holds for $T_k$, and thus there is a unique permuation
$\sigma_k\in S_m$ such that $T_k=\Inv{\sigma_k}$.  Lemma~\ref{easy lemma} then says that
$\Inv{\sigma_k[\Id_{z_1},\ldots, \Id_{z_k}]}= \cup_{(i,j)\in T_k} U_i\times U_j$.  Since the inversion
set uniquely determines the permutation, we therefore have $\alpha_k=
\sigma_k[\Id_{z_1},\ldots, \Id_{z_k}]$ for each $k=1$,\ldots, $r$.
\end{proof}

\begin{corollary}\label{irreducible under inflation}
Let $\sigma \in S_m$, and let $z_1,z_2,\dots,z_m$ be positive integers.
Then the permutation $\alpha:=\sigma[\Id_{z_1},\Id_{z_2},\dots,\Id_{z_m}] \in S_n$ is irreducible
if and only if $\sigma$ is irreducible.
\end{corollary}

\begin{corollary} \label{easy corollary}
Let $\alpha=\sigma[\beta_1,\beta_2,\dots,\beta_m]  \neq \Id$
where $\beta_i \in S_{|U_i|}$ for $i=1,2,\dots,m$.
Then $\alpha$ is irreducible if and only if exactly one of the permutations
$\sigma,\beta_1,\beta_2,\dots,\beta_m$ is a non-identity permutation and that non-identity permutation is
itself irreducible.
In particular, if $\alpha$ is irreducible with $\sigma \neq \Id$ then $\alpha=\sigma[\Id,\Id,\dots,\Id]$ where $\sigma$
is irreducible.
\end{corollary}

\begin{proof}
If $\alpha=\sigma[\beta_1,\ldots, \beta_m]$ then it follows immediately from Lemma~\ref{easy lemma} that
$$\Inv\alpha = \Inv{\sigma[\Id_{z_1},\ldots, \Id_{z_m}]} \sqcup \textstyle \bigsqcup_i \displaystyle
\Inv{\Id_m[\Id_{z_1},\ldots, \Id_{z_{i-1}}, \beta_i,\Id_{z_{i+1}},\ldots, \Id_{z_m}]},$$
where $z_i=|U_i|$ for $i=1$,\ldots, $m$.
If $\alpha$ is irreducible all then all but one of the inversion sets in the decomposition on the right are empty,
and hence all but one of the corresponding elements are the identity.
Conversely, if more than one of the inversion sets in this decomposition of $\Inv{\alpha}$ is non-empty then
 we have a non-trivial decomposition of $\alpha$.  Furthermore, if $\sigma$ is reducible then Lemma~\ref{inflate inversion set lemma} shows that $\alpha$ is reducible too.
Similarly if some $\beta_i$ is reducible then the order preserving bijection $\Phi_i$ from Lemma~\ref{easy lemma}
induces a decomposition of $\Inv{\Id_m[\Id_{z_1},\ldots, \Id_{z_{i-1}}, \beta_i,\Id_{z_{i+1}},\ldots, \Id_{z_m}]}$, showing again that $\alpha$ is reducible.
\end{proof}

\begin{definition}
Let $\alpha \in S_n$, and $\F,\F' \subset \{1,2,\dots,n\}$. We say that $\F$ and $\F'$ are $\alpha$-{\em connected} if
\begin{itemize}
\item[(i)] $\theta_\F(\alpha)$ and $\theta_{\F'}(\alpha)$ are irreducible;
\item[(ii)] $\Inv\alpha \cap \Delta^+_\F \cap \Delta^+_{\F'} \neq \emptyset$.
\end{itemize}
\end{definition}

The following two results will be used several times.

\begin{lemma}\label{technical}
Let $\alpha \in S_n$ and $\F,\F' \subset \{1,2,\dots,n\}$.
Assume that $\mu = \theta_\F(\alpha) \neq \Id$ and $\mu' = \theta_{\F'}(\alpha) \neq \Id$.
Suppose that $\Inv\alpha = \Inv{\alpha_1} \sqcup \Inv{\alpha_2} \sqcup \dots \sqcup \Inv{\alpha_r}$.
\begin{enumerate}
\item If $\mu$ is irreducible then there exists a unique index $\delta(\F)$ with $1 \leq \delta(\F) \leq r$ such that
$\Inv\alpha \cap \Delta^+_\F \subseteq \Inv{\alpha_{\delta(\F)}}$, and hence
$\Inv\alpha \cap \Delta^+_\F \cap \Inv{\alpha_{i}}=\emptyset$ for all $i\neq \delta(\F)$. \label{one}
\item If $\F$ and $\F'$ are $\alpha$-connected
then $\delta(\F)=\delta(\F')$. \label{two}
\end{enumerate}
\end{lemma}
\begin{proof}
Let $m=|\F|$ and suppose that $\mu$ is irreducible.     Put $\mu_a = \theta_{\F}(\alpha_a)$ for $a=1,2,\dots,r$.
There exists an order preserving bijection $\Psi : \F \to \{1,2,\dots,m\}$.
It is easy to see that $(i,j) \in \Inv\alpha \cap \Delta^+_\F$ if and only if $(\Psi(i),\Psi(j)) \in \Inv{\theta_\F(\alpha)}$.
Thus $\Psi$ identifies $\Inv\alpha \cap \Delta^+_\F$ with $\Inv{\theta_\F(\alpha)}$.
Intersecting $\Inv\alpha= \Inv{\alpha_1} \sqcup \Inv{\alpha_2} \sqcup \dots \sqcup \Inv{\alpha_r}$ with $\Delta^+_\F$ and
using this identification we get
$\Inv\mu = \Inv{\mu_1} \sqcup \Inv{\mu_2} \sqcup \dots \sqcup \Inv{\mu_r}$.
Since $\mu$ is irreducible, there exists a unique $\delta(\F)$ such that $\Inv\mu = \Inv{\mu_{\delta(\F)}}$,
furthermore $\Inv{\mu_i}=\emptyset$ for all $i\neq \delta(\F)$.
Therefore $\Inv{\alpha} \cap \Delta^+_{\F} \subseteq \Inv{\alpha_{\delta(\F)}}$ and
$\Inv\alpha \cap \Delta^+_\F \cap\Inv{\alpha_{i}}=\emptyset$ for all $i\neq \delta(\F)$.

For the second assertion, recall that $\mu$ and $\mu'$ are irreducible by definition.
By the above, $\Inv{\alpha} \cap \Delta^+_{\F} \subseteq \Inv{\alpha_{\delta(\F)}}$ and
$\Inv{\alpha} \cap \Delta^+_{\F'} \subseteq \Inv{\alpha_{\delta(\F')}}$.
Since $\F$ and $\F'$ are $\alpha$-connected there exists $(i,j) \in \Inv\alpha \cap \Delta^+_\F \cap \Delta^+_{\F'}$.
Thus $(i,j) \in  \Inv{\alpha_{\delta(\F)}} \cap  \Inv{\alpha_{\delta(\F')}}$.
Hence $\delta(\F)=\delta(\F')$.
\end{proof}

\begin{corollary}\label{technical cor}
Suppose $\F_1,\F_2,\dots,\F_s \subset \{1,2,\dots,n\}$ where
$\F_i$ and $\F_{i+1}$ are $\alpha$-connected for all $1\leq i\leq s-1$.
(In particular, $\theta_{\F_i}(\alpha)$ is irreducible for all $i=1,2,\dots,s$.) Assume further that
$\Inv{\alpha} \subseteq \cup_{i=1}^s \Delta_{\F_i}^+$.
Then $\alpha$ is irreducible.
\end{corollary}

\begin{proof}
Suppose that $\Inv{\alpha} = \Inv{\alpha_1} \sqcup \Inv{\alpha_2} \sqcup \dots \sqcup \Inv{\alpha_r}$.
By Lemma~\ref{technical}, we have $j = \delta(\F_1) = \delta(\F_2) = \dots = \delta(\F_r)$
with $\Inv{\alpha} \cap \Delta_{\F_i}^+ \subseteq \Inv{\alpha_j}$.
Therefore $\Inv{\alpha} = \Inv{\alpha} \cap (\cup_{i=1}^s  \Delta_{\F_i}^+)
= \cup_{i=1}^s (\Inv{\alpha} \cap \Delta_{\F_i}^+) \subset \Inv{\alpha_j}$ and thus the
decomposition $\Inv{\alpha}= \Inv{\alpha_1} \sqcup \Inv{\alpha_2} \sqcup \dots \sqcup \Inv{\alpha_r}$ is trivial.
\end{proof}

%\begin{remark}\label{technical remark}
%Suppose $\sigma=\sigma_0[\sigma_1,\sigma_2,\dots,\sigma_m]$ where $\sigma_0$ is irreducible.
%Further suppose that the sets $\F_1,\F_2,\dots,\F_t$ are all admissible (with respect to $\sigma$).
%If the $\F_i$ are all $\sigma$-path connected and $\Inv{\sigma} \subset \cup_{i=1}^s \Delta_{\F_i}^+$ then the above corollary
%applies and shows that $\sigma$ is irreducible.
%\end{remark}

The basic objects for describing a permutation by inflation are the simple permutations, while
in describing decompositions the basic permutations are the irreducible ones.
In our recursive method of describing decompositions by inflations
it is therefore natural to choose the basic object to be those permutations which are both simple and irreducible.

It turns out that simple permutations are automatically irreducible (and thus our basic building blocks are again
the simple permutations).  To prove this and a related useful fact
we need an additional definition.
Recall that a permuation $\alpha\in S_n$ is simple if it has no blocks of length $t$ for $2\leq t\leq n-1$.

\begin{definition}
A permutation $\alpha \in S_n$ is {\it two-block simple} if $\alpha(1) \neq 1$,
$\alpha(i+1) \neq \alpha(i) + 1$ for $1 \leq i \leq n-1$, and $\alpha(n) \neq n$.
\end{definition}

The name is somewhat inaccurate: the condition that $\alpha$ has no blocks of length two is that $\alpha(i+1) \neq \alpha(i) \pm 1$ for all $i$,
whereas we are only asking that $\alpha(i+1) \neq \alpha(i)+1$ for all $i$, and imposing the additional conditions
that $\alpha(1)\neq 1$ and $\alpha(n)\neq n$.
%The name is somewhat inaccurate: the condition that $\alpha$ has no blocks of length two is that
%$\alpha(i+1)\neq\alpha(i)\pm1$ for all $i$, and we have also imposed the additional conditions that
%$\alpha(1)\neq 1$ and $\alpha(n)\neq n$.
Nonetheless, we continue to use this name since it gives an indication of the defining conditions.
In Proposition~\ref{simple implies irreducible} below we will show that the property of being
simple is equivalent to the property of being both irreducible and two-block simple.
From this equivalence and
Lemma~\ref{complements are simple} we will deduce that if $\alpha$ is simple then $J\alpha$ is irreducible.

Our method of proving the equivalence is inductive.  The base cases of the induction are a particular family
of permutations previously appearing in the literature.

\begin{definition} (\cite[Definition 4]{AA})
Let $n=2m$ be even with $m \geq 2$.
A permutation $\alpha\in S_n$ is {\it exceptional} if it is one of the following permutations
\begin{enumerate}
\item $\alpha = (2,4,6,\dots,2m-2,2m,1,3,5,\dots,2m-3,2m-1)$,
\item $\alpha=(m+1,1,   m+2,   2,   m+3,  3,\dots,  2m-1, m-1, 2m,  m)$,
\item $\alpha= (2m-1,2m-3,2m-5,\dots,3,1,2m,2m-2,2m-4,\dots,4,2)$,
\item $\alpha =(m,  2m, m-1, 2m-1,m-2,2m-2,\dots,2,     m+2,   1,  m+1)$.
\end{enumerate}
\end{definition}

\begin{lemma}\label{exceptional}
Let $\alpha \in S_n$ be exceptional.  Then $\alpha$ is irreducible and two-block simple.
\end{lemma}
\begin{proof}
It is easily seen that all of these permutations are two-block simple.

(1) Suppose $\alpha = (2,4,6,\dots,2m-2,2m,1,3,5,\dots,2m-3,2m-1)$.
Then $\Inv\alpha$ has only one simple root, $(m,m+1)$, and so is irreducible.

For the remaining cases, we proceed by induction using Corollary~\ref{technical cor} repeatedly.

(2) Suppose $\alpha=(m+1,1,   m+2,   2,   m+3,  3,\dots,  2m-1, m-1, 2m,  m)$.  For $m=2$ we check directly that $\alpha = (3,1,4,2)$
is irreducible. Let $m \geq 3$ and set
$$
\begin{array} {rclcl}
\F_1& := & \{1,2,\dots,2m-2\} & = &  \{1,2,\dots,2m\} \backslash \{2m-1, 2m\}\\
\F_2& := &\{3,4,\dots,2m\} & = &  \{1,2,\dots,2m\} \backslash \{1,2\}\\
\F_3 & : = & \{1,2,2m-1,2m\}.&&
\end{array}
$$
Then $\theta_{\F_1}(\alpha) = \theta_{\F_2}(\alpha) =  (m,1,   m+1,   2,   m+2,  3,\dots,  2m-3,  m-2, 2m-2, m-1)$
is irreducible by the induction assumption and $\theta_{\F_3}(\alpha) = (3,1,4,2)$ is irreducible by the base case.
Furthermore, $(3,4) \in \Inv{\alpha} \cap \Delta_{\F_1}^+ \cap \Delta_{\F_2}^+$ and
$(2m-1,2m) \in \Inv{\alpha} \cap \Delta_{\F_2}^+ \cap \Delta_{\F_3}^+$ together with  the
observation that $\Inv{\alpha} \subset \Delta^+ = \Delta_{\F_1}^+ \cup \Delta_{\F_2}^+ \cup \Delta_{\F_3}^+$ imply that Corollary~\ref{technical cor}
applies and hence     $\alpha$ is irreducible.

(3) Suppose $\alpha= (2m-1,2m-3,2m-5,\dots,3,1,2m,2m-2,2m-4,\dots,4,2)$.  For $m = 2$, $\alpha = (3,1,4,2)$,
as discussed above, is irreducible. It is not too difficult to check directly that, for $m = 3$,
$\alpha = (5,3,1,6,4,2)$ is irreducible as well. Let $m \geq 4$ and set
$$
\begin{array} {rclcl}
\F_1& := & \{1,2,\dots, m-1, m+1,\ldots, 2m-1\} & = &  \{1,2,\dots,2m\} \backslash \{m, 2m\}\\
\F_2& := &\{2,3,\dots,m, m+2, \ldots,2m\} & = &  \{1,2,\dots,2m\} \backslash \{1,m+1\}\\
\F_3 & : = & \{2,m, m+1,2m\}&&\\
\F_4 & := & \{1, m, m+1, 2m\}.&&
\end{array}
$$
Then $\theta_{\F_1}(\alpha) = \theta_{\F_2}(\alpha) =  (2m-3, 2m-5, \ldots, 3,1, 2m-2, 2m-4, \ldots, 4, 2)$
is irreducible by the induction assumption and $\theta_{\F_3}(\alpha) = \theta_{\F_4}(\alpha) = (3,1,4,2)$ is irreducible by the base case.
Furthermore, $(2, m+3) \in \Inv{\alpha} \cap \Delta_{\F_1}^+ \cap \Delta_{\F_2}^+$,
$(2, 2m) \in \Inv{\alpha} \cap \Delta_{\F_2}^+ \cap \Delta_{\F_3}^+$, and
$(m+1,2m) \in \Inv{\alpha} \cap \Delta_{\F_3}^+ \cap \Delta_{\F_4}^+$ together with the
observation that $\Inv{\alpha} \subset \Delta^+ = \Delta_{\F_1}^+ \cup \Delta_{\F_2}^+ \cup \Delta_{\F_3}^+  \cup \Delta_{\F_4}^+$ imply that
Corollary~\ref{technical cor}
applies and hence     $\alpha$ is irreducible.

(4) Finally suppose that  $\alpha = (m,  2m, m-1, 2m-1,m-2,2m-2,\dots,2,m+2,1, m+1)$.   For $m = 2$ we check directly
that  $\alpha = (2,4,1,3)$ is irreducible. It is not too difficult to check directly that, for $m = 3$,
$\alpha = (3,6,2,5,1,4)$ is irreducible as well. Let $m \geq 4$ and set
$$
\begin{array} {rclcl}
\F_1& := & \{1,2,\dots, 2m-3, 2m-2\} & = &  \{1,2,\dots,2m\} \backslash \{2m-1, 2m\}\\
\F_2& := &\{1,2, \ldots,2m-3, 2m\} & = &  \{1,2,\dots,2m\} \backslash \{2m-2,2m-1\}\\
\F_2& := &\{3,4, \ldots,2m-1, 2m\} & = &  \{1,2,\dots,2m\} \backslash \{1,2\}.
\end{array}
$$
Then $\theta_{\F_1}(\alpha) = \theta_{\F_2}(\alpha) = \theta_{\F_3}(\alpha)=  (m-1, 2m-2, m-2, 2m-3, \ldots, 2, m+1, 1, m)$
is irreducible by the induction assumption.
Furthermore, $(2, 3) \in \Inv{\alpha} \cap \Delta_{\F_1}^+ \cap \Delta_{\F_2}^+$ and
$(3, 2m-3) \in \Inv{\alpha} \cap \Delta_{\F_2}^+ \cap \Delta_{\F_3}^+$,  together with the
observation that $\Inv{\alpha} \subset \Delta^+ = \Delta_{\F_1}^+ \cup \Delta_{\F_2}^+ \cup \Delta_{\F_3}^+$ imply that Corollary~\ref{technical cor}
applies and hence     $\alpha$ is irreducible.
\end{proof}

We now turn to the reduction step and then the inductive proof of Proposition~\ref{simple implies irreducible}.

\begin{definition}
Let $\alpha \in S_n$. Choose $k$ with $1 \leq k \leq n$ and put $\F = \{1,2,\dots,n\} \setminus \{k\}$.
The permutation $\alpha^\circ=\theta_\F(\alpha) \in S_{n-1}$ is called a {\it one point deletion} of $\alpha$.
We say that $\alpha^\circ$ is obtained from $\alpha$ by deleting $(k,\alpha(k))$.
\end{definition}

The following theorem, expressed in the language of posets, was first proved by Schmerl and Trotter \cite{ST}.
The version below in terms of permuatations appears as \cite[Theorem~5]{AA}.

\begin{theorem}\label{Schmerl-Trotter theorem}
Let $n \geq 2$ and suppose $\alpha \in S_n$ is simple but not exceptional.
Then $\alpha$ has a one point deletion $\alpha^\circ$ which is simple.
\end{theorem}

\begin{lemma}\label{one-point-deletion-irreducible-means-not-simple}
Suppose that $\alpha\in S_n$ is reducible and has a one-point deletion which is irreducible.
Then $\alpha$ is not simple.
\end{lemma}

\begin{proof}
Let $\Inv{\alpha}= \Inv{\alpha_1} \sqcup \Inv{\alpha_2}$ be a non-trivial decomposition, and
let $\alpha^\circ$ be an irreducible one-point deletion obtained from $\alpha$ by deleting $(k,\ell)$ with
$\ell=\alpha(k)$.  Since $\alpha^\circ$ is irreducible, applying Lemma~\ref{technical}(\ref{one})
with $\F=\{1,\ldots,  n\}\setminus\{k\}$ gives that either
$\Inv{\alpha}\cap \Delta_{\F}^{+}\subset \Inv{\alpha_1}$ or $\Inv{\alpha}\cap \Delta_{\F}^{+}\subset \Inv{\alpha_2}$ .
By relabeling we may assume that $\Inv{\alpha}\cap \Delta_{\F}^{+}\subset \Inv{\alpha_1}$.
Concretely this means that all roots of the form $(i,j)\in \Inv{\alpha}$ with either $i\neq k$ or $j\neq k$
are in $\Inv{\alpha_1}$. The remaining roots in $\Inv{\alpha}$, those of the form $(i,k)$ or $(k,j)$ may appear
in either $\Inv{\alpha_1}$ or $\Inv{\alpha_2}$, and $\Inv{\alpha_2}$ only has roots of this form.
Furthermore, since the decomposition is assumed non-trivial, there is at least one root in $\Inv{\alpha_2}$.

Write $|\Inv{\alpha_2}| = p+q$ where $p$ of the elements of $\Inv{\alpha_2}$ are of the form
$(i,k)$ with $1 \leq i < k$ and $q$ of the elements of $\Inv{\alpha_2}$ are of the form
$(k,j)$ with $k < j  \leq n$.
Suppose both $p$ and $q$ are non-zero.  Then there exist $i < k$ and $j > k$ with $(i,k), (k,j) \in \Inv{\alpha_2}$.
Since $\Inv{\alpha_2}$ is closed this means that the root $(i,j)\in \Inv{\alpha_2}$ contrary
to the description above.  Thus only one of $p$ and $q$ is non-zero.

Assume first that $p\neq 0$ and $q=0$.  By the form of the roots in $\Inv{\alpha_2}$ the only simple root
in $\Inv{\alpha_2}$ is $(k-1,k)$.
This implies that $\alpha_2$ preserves the relative order of the elements in $\{1,2,\ldots, k-1\}$ and the relative
order of the elements in $\{k,k+1,\ldots, n\}$.  Along with the fact that the only roots in $\Inv{\alpha_2}$ are
of the form $(i,k)$, this implies that $\Inv{\alpha_2}= \{ (k-p,k), (k-p+1,k), \dots, (k-1,k) \}$.

Set $s= \max\{\alpha(k-i) \mid 1\leq i \leq p\} - \ell$, and let
$R$ be the $p\times s$ rectangle $R := [k-p,k] \times [\ell,\ell+s]$.
There are $p+1$ vertical lattice lines and $s+1$ horizontal lattice lines through $R$.
Exactly $p+1$ points of the graph of $\alpha$ lie inside the rectangle $R$:
$(k-p,\alpha(k-p)), (k-p+1,\alpha(k-p+1)), \dots, (k-1,\alpha(k-1))$ and $(k,\ell)$.  Since $\alpha$ is a permutation
(and hence injective), no two points of its graph may lie on the same horizontal line, and thus $s\geq p$.
We will now show that $s\leq p$ and hence $s=p$.

%\newpage
%\phantom{a}

% Delete the above two commands once the diagram and paragraph appear on the same page by themselves without
% coersion.

\vspace{3\baselineskip}

\newgray{vlgray}{0.90}

\hfill
\psset{unit=0.375cm}
\begin{pspicture}(1,1)(12,12)
\pspolygon[linestyle=solid,fillstyle=solid,fillcolor=vlgray](5,3)(5,9)(9,9)(9,3)
\multido{\n=1+1.0}{12}{%
\psline[linecolor=gray](\n,1)(\n,12)
\psline[linecolor=gray](1,\n)(12,\n)
}
\rput(8.5,8.5){\tiny $R$}
\psset{boxsep=false,framesep=-0.1pt,linecolor=white,fillcolor=white,fillstyle=solid}
\rput(10.1,2.5){\psframebox{\tiny $(k,\ell)$}}
\rput(7,9.5){\psframebox{\tiny $(z,\ell+s)$}}
\rput(3,5.5){\psframebox{\tiny $(x,\alpha(x))$}}
\rput(10.5,7.5){\psframebox{\tiny $(y,\alpha(y))$}}
\psset{linecolor=black}
\GraphPoint{1}{11}
\GraphPoint{2}{10}
\GraphPoint{3}{6}
\GraphPoint{4}{1}
\GraphPoint{5}{4}
\GraphPoint{6}{7}
\GraphPoint{7}{9}
\GraphPoint{8}{5}
\GraphPoint{9}{3}
\GraphPoint{10}{12}
\GraphPoint{11}{8}
\GraphPoint{12}{2}
\end{pspicture}

\vspace{-5.3cm}

\parshape 13 0cm \textwidth 0cm \textwidth 0cm 10.6cm 0cm 10.6cm 0cm 10.6cm 0cm 10.6cm 0cm 10.6cm 0cm 10.6cm 0cm 10.6cm 0cm 10.6cm 0cm 10.6cm 0cm 10.6cm 0cm \textwidth
We first claim that there are no points of the graph of $\alpha$ strictly to the left of $R$, i.e.,
a point $(x,\alpha(x))$ with $x < k-p$ and $\ell < \alpha(x) < \ell+s$.  Assume to the contrary that $(x, \alpha(x))$ is such a point.
A potential graph of such an $\alpha$ is shown below right (although, as part of the proof we will show,
certain features of the graph are incorrect).
Let $\F = \{x,z,k\}$ with $z := \alpha^{-1}(\ell+s)$.  By this choice of $x$, the slope between $(x,\alpha(x))$
and $(k,\ell)$ is negative, and so $(x,k)$ is a root of $\Inv{\alpha}$.  This root is not contained in
$\Inv{\alpha_2}$ since if $(i,k)\in \Inv{\alpha_2}$ then $(i,\alpha(i))$ is in $R$, and we have chosen
$(x,\alpha(x))$ outside of $R$.  Thus $(x,k)\in \Inv{\alpha_1}$, and hence $\Inv{\theta_{\F}(\alpha_1)}\neq \emptyset$.
On the other hand, the slope between $(z,\ell+s)$ and $(k,\ell)$ is also negative, and thus
$(z,k)\in \Inv{\alpha}$. Since $(z,\ell)$ is in $R$, this root is in $\Inv{\alpha_2}$, and hence
$\Inv{\theta_{\F}(\alpha_2)}\neq \emptyset$.  Thus applying $\theta_\F$ to the decomposition of $\Inv{\alpha}$
produces a non-trivial decomposition of $\theta_\F(\alpha)=(2,3,1)$. However the permutation $(2,3,1)$ is
irreducible, and this contradiction establishes the claim.

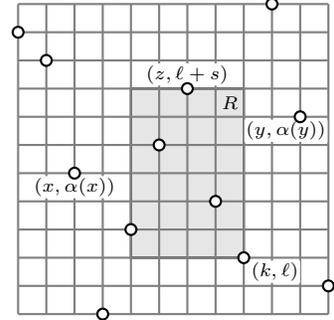
\parshape 1 0cm \textwidth

We similarly claim that there are no points of the graph of $\alpha$ strictly to the right of $R$.
Assume such a point $(y,\alpha(y))$ exists with $k < y$ and $\ell < \alpha(y) < \ell+s$.
Set $\F = \{z,k,y\}$ with $z = \alpha^{-1}(\ell+s)$ as above.
As before, we deduce that $(z,y)\in \Inv{\alpha_1}$ and so $\Inv{\theta_{\F}(\alpha_1})\neq \emptyset$,
that $(z,k)\in \Inv{\alpha_2}$ and so $\Inv{\theta_{\F}(\alpha_2})\neq \emptyset$, and hence that the
decomposition of $\Inv{\alpha}$ induces a nontrivial decomposition of the irreducible element
$\theta_\F(\alpha)=(3,1,2)$.  Thus there are no points on the graph of $\alpha$ to the right of $R$ either.

Since $\alpha$ is a permutation (and hence surjective) there is a point of its graph on each of the $s+1$ horizontal
lattice lines through $R$.  We have just shown that none of these points lie outside of $R$, and hence all are in
$R$. Each of these points lies on a different vertical lattice line of which there are $p+1$, and so $s\leq p$.
Thus $p=s$, $R$ is a square, and the graph of $\alpha$ contains $p+1$ points in $R$.  If $p+1 < n$ then $\alpha$ is
not simple because it has the block of size $p+1$ corresponding to $R$. If $p+1 = n$ then $(k,\ell)=(n,1)$
 and $\alpha$ is not simple because it has the block $\{1, 2, \ldots, n-1\}$ of size $n-1$.

A similar argument, with a rectangle of the form $R=[k,k+q]\times [\ell,\ell-s]$,
handles the case $p=0$ and $q\neq 0$.
\end{proof}

We can now prove our characterization of simple permutations.

\begin{proposition}\label{simple implies irreducible}
Let $n \geq 2$ and let $\alpha \in S_n$ with $\alpha \neq \Id_2$.
Then $\alpha\in S_n$ is simple if and only if it is irreducible and two-block simple.
%
%If $\alpha$ is irreducible and two-block simple then $\alpha$ is simple.
%Conversely, if $\alpha$ is simple and $n \geq 4$ then $\alpha$ is irreducible
%and two-block simple.
%
%A permutation $\alpha\in S_n$ is simple if and only if it is irreducible
%and two-block simple.
\end{proposition}
\begin{proof}
First suppose $\alpha$ is irreducible and two-block simple, and
express $\alpha$ in simple form:  $\alpha=\sigma[\beta_1,\beta_2,\dots,\beta_m]$
with $\beta_b \in S_{z_b}$ for $b=1,2,\dots,m$.
By Corollary~\ref{easy corollary}, exactly one of the permutations $\sigma,\beta_1,\beta_2,\dots,\beta_m$
is a non-identity permutation.
Since $\alpha$ is two-block simple we have $\alpha(1)\neq 1$ and $\alpha(n)\neq n$.
If $\sigma=\Id_m$ then this would imply that $\beta_1 \neq \Id$ and $\beta_m \neq \Id$,
which is a contradiction.
Thus $\sigma \neq \Id$ and $\alpha=\sigma[\Id_{z_1},\Id_{z_2},\dots,\Id_{z_m}]$.
Since $\alpha$ is two-block simple, $z_b=1$ for all $b$ and hence
 $\alpha=\sigma$.   Now either $\sigma$ is simple or $\sigma = \wo[m]$.  In the former case, $\alpha=\sigma$
is simple.  Otherwise, if $\alpha=\wo[m]$, then we must have $m=2$ by Corollary~\ref{when is Jm reducible} and thus
$\alpha=\wo[2]$  is simple.

Next we suppose that $\alpha$ is simple.   We proceed by induction on $n$.
%We prove that simple implies irreducible by induction.
For $n=2$, we must have $\alpha=\wo[2]$ which is simple, two-block simple and irreducible.
For $n=3$, no elements are simple. % and the only element which is two-block simple is $\wo[3]$ which is not irreducible.
%The result is easily verified for $n=1,2$, and $3$.
%{\bf [Say explicitly what happens for $n=1,2,3$.]}
%\comment{is this Proposition true for $\alpha=\Id_2$?}
For $n \geq 4$,  it follows immediately from the definitions that $\alpha$ is two-block simple.
It remains for us to prove, by induction, that simple implies irreducible.

For $n=4$ the only two simple permutations are $(2,4,1,3)$ and $(3,1,4,2)$, both of which are exceptional (they
appear as (1) and (2) on the list with $m=2$) and hence irreducible by Lemma~\ref{exceptional}.

Suppose $n \geq 5$ and that $\alpha \in S_n$ is simple.
If $\alpha$ is exceptional then the result follows from Lemma~\ref{exceptional}.
If not, then by Theorem~\ref{Schmerl-Trotter theorem}, $\alpha$ has a one-point deletion $\alpha^{\circ}$ which
is also simple, and hence irreducible by the inductive hypothesis.  But then $\alpha$ must also be irreducible.
If not, then Lemma~\ref{one-point-deletion-irreducible-means-not-simple} would apply to show that $\alpha$ is
not simple, contrary to assumption.
\end{proof}

\begin{remark}
The permutation $\alpha=(2,4,5,1,3) \in S_5$ is irreducible since its inversion set contains only one simple root.
However $\alpha$ is not two-block simple (and hence also not simple).  I.e., although the condition of being
simple implies that of being irreducible, the reverse implication does not hold.
\end{remark}

\begin{corollary}\label{main cor}
Suppose that $\alpha \in S_n$ is simple.  Then $\wo\alpha$ is irreducible.
\end{corollary}

\begin{proof}
By Lemma~\ref{complements are simple} the fact that $\alpha$ is simple implies that $\wo\alpha$ is simple.
But then $\wo\alpha$ is irreducible (and also two-block simple) by Proposition~\ref{simple implies irreducible}.
\end{proof}

This corollary is required for the proof of the main theorem and was one of the motivations for proving
Proposition~\ref{simple implies irreducible}.

The following lemma is also required for the proof of the main theorerm.
\begin{lemma}\label{vee lemma}
  Let $\alpha \in S_n$ be simple with $n\geq 4$.  Fix $k$ with $1 \leq k \leq n$.  Then there exists $(i,j) \in \Inv{\alpha}$
  with $i \neq k$ and $j \neq k$.
\end{lemma}
\begin{proof}
  Define $a := \alpha^{-1}(1)$ and $b := \alpha^{-1}(n)$.  Then $a \neq 1$, $a \neq n$, $b\neq 1$ and $b \neq n$
   since $\alpha$ is simple.
  Clearly $(1,a), (b,n) \in \Inv{\alpha}$.
  The inequalities $1 \neq n$, $a \neq b$, $b \neq 1$ and $a \neq n$ imply that $\{1,a\} \cap \{b,n\} = \emptyset$.
   Thus either $k \notin\{1,a\}$ or $k \notin \{b,n\}$.  Hence one of $(1,a)$ or $(b,n)$ may be used as the required root
   $(i,j)$.
\end{proof}

We now prove the main theorem.

\medskip
\noindent
{\em Proof of Theorem~\ref{main theorem}.\ }
Suppose that $\Delta_n^+ = \Inv{\alpha_1} \sqcup \Inv{\alpha_2} \sqcup \dots \sqcup \Inv{\alpha_r}$
is a decomposition with $\Inv{\alpha_s} \neq \emptyset$ for all $s$,
and express $\alpha_1 = \sigma_1[\beta_{11},\beta_{12},\dots,\beta_{1m}]$ in simple form.
Recall that by assumption the highest root $(1,n)$ is an element
of $\Inv{\alpha_1}$. Let $\{1,2,\dots,n\} = U_1 \sqcup U_2 \sqcup \dots \sqcup U_m$
be the intervals corresponding to the simple
form $\alpha_1  =  \sigma_1[\beta_{11},\beta_{12},\dots,\beta_{1m}]$ with $z_b=|U_b|$.
Throughout this proof ``admissible'' refers to the intervals $U_1, \ldots, U_m$.

%We consider two cases according to whether $\sigma_1=\wo[m]$ or not.

%First, if $\sigma_1={\wo[m]}$ then $\alpha_1 = \wo[m][\beta_{11},\beta_{12},\dots,\beta_{1m}]$
%and $\alpha_1$ is minus-decomposable.
%Conversely if $\alpha_1$ is minus-decomposable, then by definition $\alpha_1$ can be written in the form
%$\alpha_1=\wo[m][\beta_{11},\beta_{12},\dots,\beta_{1m}]$ for some $\beta_{11}$, $\beta_{12},\dots,\beta_{1m}$.
%In other words $\sigma_1=\wo$ if and only if $\alpha_1$ is minus-decomposable.

The assumption that $(1,n) \in \Inv{\alpha_1}$ means that $\sigma_1 \neq \Id_m$ and Corollary~\ref{new corollary}
implies that $\sigma_1={\wo[m]}$
and $\alpha_1$ is minus-decomposable or $\sigma_1\in S_m$ is simple with $m\geq 4$ and $\alpha_1$ is minus-indecomposable.
We consider these two cases separately.

First, suppose  that $\sigma_1\in S_m$ is simple with $m\geq 4$.
We then show that there is an $s\geq 2$ such that
$\alpha_s$ is of the form $\alpha_s=(\wo[m]\sigma_1)[\beta_{s1},\beta_{s2},\dots,\beta_{sm}]$.
Let $\F$ be an admissible set.  Then $\theta_\F(\alpha_1)=\sigma_1$ and
$\Inv{\sigma_1} \sqcup \Inv{\theta_\F(\alpha_2)} \sqcup \dots \sqcup \Inv{\theta_\F(\alpha_r)} = \Delta_\F^+$.
Thus
$\Inv{\wo[m]\sigma_1} = \Inv{\theta_\F(\alpha_2)} \sqcup \Inv{\theta_\F(\alpha_3)}\sqcup \dots \sqcup \Inv{\theta_\F(\alpha_r)}$.
The element $\wo \sigma_1$ is irreducible by Corollary~\ref{main cor} and thus there exists $\delta(\F)\geq 2$ such
that $\Inv{\wo \sigma_1} = \Inv{\theta_\F(\alpha_{\delta(\F)})}$, and
$\Inv{\theta_\F(\alpha_s)}=\emptyset$ for all $s\neq \delta(\F), s\geq 2$.
We claim that the number $\delta(\F)$ is independent of the choice of admissible set $\F$.

Recall that an admissible set is the choice of a single element from each of the sets $U_1$, $U_2$, \ldots,
$U_m$, and thus the admissible sets are in one to one correspondence with the points of $U_1\times U_2\times
\cdots \times U_m$.  Given any two admissible subsets $\F$ and $\F'$ we may find a sequence of
admissible subsets $\F_0=\F$, $\F_1$, \ldots, $\F_{l-1}$, $\F_l=\F'$ such that each $\F_i$ and $\F_{i+1}$ differ by
only a single element (i.e, under the correspondence with elements of $U_1\times\cdots\times U_m$, differ
in only a single coordinate).  To prove that $\delta(\F)$ is independent of the choice of admissible set
we may thus reduce to the case that $\F$ and $\F'$ differ by a single element.

Suppose that $\F'$ is obtained from $\F$ by replacing $u_k\in U_k$ with $u'_k\in U_k$ for some $1\leq k\leq m$.
There is a root $(i,j)$ in $\Inv{\wo\sigma_1}$ with $i\neq k$ and $j\neq k$ by Lemma~\ref{vee lemma}.
Since $\F$ and $\F'$ differ only in the element in $U_k$, the elements they choose from $U_i$ (respectively $U_j$)
are the same.  Set
$a=\Delta_{\F}\cap U_i = \Delta_{\F'}\cap U_i$ and $b=\Delta_{\F}\cap U_j = \Delta_{\F'}\cap U_j$.  Then
$(i,j)\in \Inv{\theta_{\F}(\alpha_s)}$ if and only if $(a,b)\in \Inv{\alpha_s}$, and similarly
$(i,j) \in \Inv{\theta_{F'}(\alpha_s)}$ if and only if $(a,b) \in \Inv{\alpha_s}$.  We have seen above that
$\Inv{\theta_\F(\alpha_s)}=\emptyset$ (respectively $\Inv{\theta_{\F'}(\alpha_s)}=\emptyset$)
for all $s\neq \delta(\F)$ (respectively, $s\neq \delta(\F')$), $s\geq 2$.  Thus $\delta(\F)=\delta(\F')$, and so
$\delta(\F)$ is constant for all admissible sets $\F$.  By relabeling the elements $\alpha_2, \ldots, \alpha_r$, we may assume that this
value is $2$.

The statement we have just proved, that $\Inv{\theta_{\F}(\alpha_2)} = \Inv{\wo\sigma_1}$ for all admissible sets
$\F$, is equivalent to the statement that for any $(i,j)\in \Delta_m$, and any $a\in U_i$, $b\in U_j$,
$(a,b)\in \Inv{\alpha_2}$ if and only if $(i,j)\in \Inv{\wo\sigma_1}$.  This implies that $\alpha_2$ permutes the
intervals $U_1$,\ldots, $U_m$, in the manner specified by $\wo\sigma_1$ and thus can be written as an inflation
$\alpha_2=(\wo\sigma_1)[\beta_{21},\ldots, \beta_{2,m}]$.  Specifically, $\beta_{2t} = \theta_{U_t}(\alpha_2)$ for
$t=1$,\ldots, $m$.

Summarizing, so far we have shown that
if $\alpha_1$ is minus-decomposable then $\alpha_1$ has simple form
$\alpha_1 = \wo[m][\beta_{11},\beta_{12},\dots,\beta_{1m}]$.
Conversely if $\alpha_1$ is minus-indecomposable then $\alpha_1$ has simple form
$\alpha_1 = \sigma_1[\beta_{11},\beta_{12},\dots,\beta_{1m}]$ and $\alpha_2$ has simple form
$\alpha_2=(\wo[m]\sigma_1)[\beta_{21},\beta_{22},\dots,\beta_{2m}]$ where $\sigma_1\in S_m$ is simple and
$m \geq 4$.

For the remaining $\alpha_s$ ($s=2,\ldots, r$ in the case that $\alpha_1$ is minus-decomposable, and $s=3,\ldots, r$
in the case that $\alpha_1$ is minus-indecomposable) we have $\alpha_s(a)<\alpha_s(b)$ for all $a\in U_i$, $b\in U_j$
and $1\leq i<j\leq m$ since these roots $(a,b)$ are all contained in
$\Inv{\alpha_1}$ (respectively $\Inv{\alpha_1}\sqcup \Inv{\alpha_2}$).
This implies that for each such $\alpha_s$ we have $\alpha_s(U_i)=U_i$ for each $i$ and
hence that $\alpha_s=\Id_m[\beta_{s1},\ldots, \beta_{sm}]$ with $\beta_{si}=\theta_{U_i}(\alpha_s)$ for
$i=1$,\ldots, $m$.   This proves the theorem in the case of a general (possibly reducible) decomposition.

Next we consider irreducible decompositions.
Corollary~\ref{easy corollary},
%By Lemma~\ref{easy lemma}, we have
%$$\Inv{\sigma_a[\beta_{a1},\beta_{a2},\dots,\beta_{am}]}
%           = \Inv{\sigma_a[\Id_{z_1},\ldots, \Id_{z_m}]} \sqcup \textstyle \bigsqcup_b \displaystyle
%\Inv{\Id_m[\Id_{z_1},\ldots, \Id_{z_{b-1}}, \beta_{ab},\Id_{z_{b+1}},\ldots, \Id_{z_m}]}.$$
%This equation
shows the necessity of conditions (i), (ii) and (iii) of the final assertion.  The element
$J_m$ is irreducible if and only if $m=2$ and this shows the necessity of condition (iv).
Thus these four conditions hold for irreducible decompositions.
Finally if these four conditions hold it is clear that each of the inversion sets
$\Inv{\alpha_a}$ is irreducible by Corollary~\ref{easy corollary}.
\hfill \qed

%{\bf [Here is a problem with the current version of the proof.
%(1) The indices used ($i$, $j$, and $t$) in the general case don't seem to conform to our conventions
%for the indices anywhere else.
%}

\section{Symmetric permutations} \label{section symmetric}
The aim of this section is to extend the results obtained so far to a special class of permutations.
The results will then be used in the next section to study
root systems of types $B$, $C$ and $D$.

A permutation $\alpha \in S_N$  is {\it symmetric} if $\alpha = \wo_{N} \alpha \wo_{N}$. Equivalently,
$\alpha \in S_N$ is symmetric if the graph of $\alpha$ is symmetric under rotation by $\pi$ radians about the
point $(\frac{N+1}{2}, \frac{N+1}{2})$.

\begin{proposition} \label{proposition 43}
Let $\alpha \in S_N$  and write $\alpha$ in simple form: $\alpha=\sigma[\beta_1,\beta_2,\dots,\beta_s]$.
If $\alpha$ is symmetric then $\sigma$ is necessarily symmetric
and $\beta_{s+1-b} = \wo_{z_b} \beta_b \wo_{z_b}$ for all $b=1,2,\dots,s$.
Consequently, if $N$ is odd then $s = 2m+1$ is odd, $z_{m+1}$ is odd and $\beta_{m+1} \in S_{z_{m+1}}$ is symmetric.
If $N$ is even then then $s$ may be even or odd; if, in addition, $s = 2m+1$ is odd then $z_{m+1}$ is even and
$\beta_{m+1} \in S_{z_{m+1}}$ is symmetric.
\end{proposition}

\begin{proof} The proof follows from the facts that, if $\alpha \in S_N$ is written in simple form as
$\alpha=\sigma[\beta_1,\beta_2,\dots,\beta_s]$, then the simple form of $\wo_N \alpha \wo_N$
is
$$
\wo_N \alpha \wo_N = (\wo_s \sigma \wo_s) [ \wo_{z_s} \beta_s \wo_{z_s}, \wo_{z_{s-1}} \beta_{s-1} \wo_{z_{s-1}}, \dots, \wo_{z_1} \beta_1 \wo_{z_1}]
$$
and the uniqueness of the simple form.
\end{proof}

Next we define an inflation operation which produces symmetric permutations.
Let $0 \leq p \leq n$ and let $\{1,2,\dots,n-p\} = U_1 \sqcup U_2 \sqcup \dots \sqcup U_m$ be a decomposition into intervals.
Put $z_b = |U_b|$ for $b=1,2,\dots,m$.
Suppose that $\beta_b \in S_{z_b}$ for $b=1,2,\dots,m$ and $\beta_{m+1} \in S_{2p+1}$ or $\beta_{m+1} \in S_{2p}$.
(For uniformity of notation we allow considering $S_{2p}$ for $p=0$; we will use $\phantomelement$ to denote the ``phantom''
element of $S_0$.)
Let $\sigma \in \left\{ \begin{array}{lcl}
S_{2m+1} & {\text { if }} & \beta_{m+1} \neq \phantomelement \\ S_{2m} & {\text { if }} & \beta_{m+1} = \phantomelement  \end{array} \right.$,

We form the inflation
$$
\alpha = \sigma[\beta_1,\beta_2,\dots,\beta_m,\beta_{m+1}, \beta_{m+2},\dots,\beta_{2m+1}]
$$
where $\beta_{2m+2-b} = \wo_{z_b} \beta_b \wo_{z_b}$ for $t=1,2,\dots,m$.  Clearly
$\alpha \in \left\{ \begin{array}{lcl}
S_{2n+1} & {\text { if }} & \beta_{m+1} \in S_{2p+1} \\ S_{2n} & {\text { if }} & \beta_{m+1} \in S_{2p}  \end{array} \right.$ is a
symmetric permutation. We call this operation {\it symmetric inflation} and denote it by
$$
\alpha =\sigma [[\beta_1, \ldots, \beta_m; \beta_{m+1}]].
$$

Proposition~\ref{proposition 43} implies that the natural notion of a ``simple symmetric permutation'' is equivalent with the requirement that a
symmetric permutation is simple. More precisely, we have the following corollary.

\begin{corollary} \label{corollary 431} Let $\alpha \in S_{2n+1}$ (respectively, $\alpha \in S_{2n}$) be a symmetric element.
Then $\alpha$ is simple in $S_{2n+1}$ (respectively, in $S_{2n}$) if and only if
$$
\alpha = \sigma[[\beta_1,\beta_2,\dots,\beta_m;\beta_{m+1}]]
$$
implies $m= 0$ or $m = n$. \qed
\end{corollary}

Finally, Theorem~\ref{simple form theorem} implies the existence of a {\it simple symmetric form expression} of a symmetric element $\alpha \in S_N$.

\begin{proposition} \label{simple symmetric form}
Let $\alpha \in S_N$ be symmetric. Then $\alpha$ can be written as
$$
\alpha = \sigma[[\beta_1, \ldots, \beta_m; \beta_{m+1}]],
$$
where $\sigma \in S_M$ is simple with $M \geq 4$ or $\sigma = I_M$ or $\sigma = \wo_M$. Furthermore, this expression is unique
if we require that $M$ be maximal when $\sigma = I_M$ or $\sigma = \wo_M$.  \qed
\end{proposition}

It also natural to discuss decomposing $\Delta_N^+$ into symmetric inversion sets. Theorem~\ref{main theorem} and Proposition~\ref{proposition 43}
imply in a straightforward manner the following theorem.

\begin{theorem}\label{symmetric A main theorem} Let $N = 2n+1$ or $N  = 2n$.
Suppose $\alpha_1, \alpha_2,\dots,\alpha_r \in S_N$ are symmetric elements and
$$
\Delta_N^+ = \Inv{\alpha_1} \sqcup \Inv{\alpha_2} \sqcup \dots \sqcup \Inv{\alpha_r}
$$
with all $\Inv{\alpha_a} \neq \emptyset$.  Without loss of generality assume that the root $(1, N) \in \Inv{\alpha_1}$.
Let $\alpha_1 =  \sigma_1[[\beta_{11},\beta_{12},\dots,\beta_{1m};\beta_{1(m+1)}]]$ be the simple symmetric form expression for $\alpha_1$
with $\sigma_1 \in S_M$ and a corresponding partition of the set $\{1,2,\dots,n\}$ into $m+1$ intervals of
 lengths $z_1,z_2,\dots,z_m,z_{m+1}$.
 Then, up to reordering of $\alpha_2,\alpha_3,\dots,\alpha_r$ there exist unique elements
$\sigma_a \in S_M$, $\beta_{ab} \in S_{z_b}$
and $\beta_{a(m+1)} \in S_P$, with $P = N - 2(z_1 + \ldots + z_m)$  such that
$\alpha_a = \sigma_a[[\beta_{a1},\beta_{a2},\dots,\beta_{am};\beta_{a(m+1)}]]$ for $a=2,\dots,r$
and
\begin{itemize}
\item[(i)] $$\begin{array}{rcl}
\Delta^+_M &=& \Inv{\sigma_1} \sqcup \Inv{\sigma_2} \sqcup \dots \sqcup \Inv{\sigma_r},  \\
&&\\
\Delta^+_{z_1} & = & \Inv{\beta_{11}} \sqcup \Inv{\beta_{21}} \sqcup \dots \sqcup \Inv{\beta_{r1}}, \\
\Delta^+_{z_2} & = & \Inv{\beta_{12}} \sqcup \Inv{\beta_{22}} \sqcup \dots \sqcup \Inv{\beta_{r2}},\\
&\vdots&\\
\Delta^+_{z_m} & = & \Inv{\beta_{1m}} \sqcup \Inv{\beta_{2m}} \sqcup \dots \sqcup \Inv{\beta_{rm}},\\
&&\\
\Delta^+_P & = & \Inv{\beta_{1(m+1)}} \sqcup \Inv{\beta_{2(m+1)}} \sqcup \dots \sqcup \Inv{\beta_{r(m+1)}};
\end{array}
$$
\item[(ii)]   if $\alpha_1$ is minus-decomposable then $\sigma_1 = \wo$
and $\sigma_2=\sigma_3=\dots=\sigma_r=\Id$;
\item[(iii)]  if $\alpha_1$ is minus-indecomposable then $\sigma_1$ is  simple
 and, after relabelling $\alpha_2, \ldots, \alpha_r$, we have $\sigma_2 = \wo \sigma_1$,
and $\sigma_3=\sigma_4=\dots=\sigma_r=\Id$.
\end{itemize}
In particular, $\sigma_1$ and at most one other of the $\sigma_a$ are not equal to the identity.\\

Let $q$ denote the number of $\sigma_a$ which are not $\Id$, i.e.,
$q := \begin{cases} 1, &\text{if }\alpha_1 \text{ is minus-decomposable};\\
2, &\text{if }\alpha_1 \text{ is minus-indecomposable}.
\end{cases}$
Again, after relabelling $\alpha_2, \ldots, \alpha_r$, we assume that $\sigma_{q+1} = \ldots = \sigma_r = I$.

The above decomposition of $\Delta_N^+$ is irreducible
 if and only if the following four conditions hold
\begin{itemize}
\item[(i)] each of the decompositions listed in (i) above
is irreducible;
\item[(ii)] exactly one of of $\beta_{a1},\beta_{a2},\dots,\beta_{am}$ is not equal to the identity for $a=q+1,\dots, r$;
\item[(iii)]  $\beta_{ab} = \Id$ for $a=1,\dots, q$ and $b=1,\dots, m+1$;
\item[(iv)]  $m = 1$ if $\alpha_1$ is minus-decomposable.  \qed
\end{itemize}
\end{theorem}

\section{Decompositions of Types B, C and D} \label{section BCD}
We now turn to root systems of types $B$, $C$ and $D$. We introduce some notation related to these root systems;
our exposition is limited only to the minimum that we need. For a reference on root systems, see \cite{BB}.
We will compare the root systems of types $B$, $C$ and $D$ with the root systems of type $A$.
We take $\{e_1,e_2,\dots,e_{n+1}\}$ as a standard basis for $\R^{n+1}$
 and consider  the Weyl group
$\W(A_n) \cong S_{n+1}$ as the group of all permutations of this basis.
With this notation, the positive roots are
$$
\Delta_{A_n}^+ = \{(i,j)=e_i-e_j \mid 1 \leq i < j \leq n+1\}.
$$

To describe the root systems $B_n, C_n$ and $D_n$, fix a standard basis
$\{\varepsilon_1,\varepsilon_2,\dots,\varepsilon_n\}$ of $\R^n$. The corresponding positive roots are
$$
\begin{array}{lll}
B_n: & \quad & \Delta_{B_n}^+ =  \{\varepsilon_i \pm \varepsilon_j \mid 1 \leq i < j \leq n\} \sqcup \{\varepsilon_i \mid 1 \leq i \leq n\};\\
C_n: & \quad & \Delta_{C_n}^+ = \{\varepsilon_i \pm \varepsilon_j \mid 1 \leq i < j \leq n\} \sqcup \{2\varepsilon_i \mid 1 \leq i \leq n\};\\
D_n: & \quad & \Delta_{D_n}^+ = \{\varepsilon_i \pm \varepsilon_j \mid 1 \leq i < j \leq n\}.\\
\end{array}
$$

The Weyl group $\W(B_n)$ is the set of signed permutations of the set
$$
\{\varepsilon_1,\varepsilon_2,\dots,\varepsilon_n,0,-\varepsilon_n,\dots,-\varepsilon_2,-\varepsilon_1 \}.
$$
These are the permutations $\alpha$ of this set such that $\alpha(0)=0$ and $\alpha(-\varepsilon_i)=-\alpha(\varepsilon_i)$
for all $1 \leq i \leq n$.
Abstractly, $\W(B_n) \cong S_n \rtimes (\Z/2\Z)^{n}$.

The Weyl group $\W(C_n)$ is the set of signed permutations of the set
$$
\{\varepsilon_1,\varepsilon_2,\dots,\varepsilon_n,-\varepsilon_n,\dots,-\varepsilon_2,-\varepsilon_1 \}.
$$
Abstractly, $\W(C_n) \cong S_n \rtimes (\Z/2\Z)^{n}$.

The Weyl group $\W(D_n)$ is the set of signed permutations of the set
$$
\{\varepsilon_1,\varepsilon_2,\dots,\varepsilon_n,-\varepsilon_n,\dots,-\varepsilon_2,-\varepsilon_1 \}
$$
involving an even number of sign changes, i.e., permutations $\alpha$ for which $\alpha(\varepsilon_i) = - \varepsilon_j$ for
an even number of indices $1 \leq i \leq n$.
Abstractly, $\W(D_n) \cong S_n \rtimes (\Z/2\Z)^{n-1}$.

In order to treat the roots systems of types $B, C$ and $D$ and their Weyl groups uniformly, we introduce some notation.
Instead of discussing separately the root systems $B_n, C_n$ or $D_n$ we will sometimes discuss
the root system $X_n$ understanding that $X$ stands for one of $B, C$ or $D$.
For uniformity of notation below, when considering $X_n$, we allow all values of $n \geq 0$: for instance, $\Delta_{X_0}^+ = \emptyset$.

Let
$$
\tilde{n} := \left\{ \begin{array}{lcl} 2n & {\text{ if }}&  X_n = B_n \\ 2n-1 & {\text{ if }} & X_n = C_n {\text{ or }} D_n. \end{array} \right.
$$
Extend the set
$\{\ep_1, \ldots, \ep_n\}$ to the set
$\X_{X_n} = \{ \ep_1, \ldots, \ldots, \ep_{\tilde{n}+1} \},$
where $\ep_i = - \ep_{\tilde{n}+2 - i}$. Note that for $X = B$ this forces $\ep_{n+1} = 0$.
Put $\Y_{X_n} = \{ e_1, \ldots, e_{\tilde{n}+1} \}$.
In order to discuss the relationship between the positive roots of $X_n$ and $A_{\tilde n}$ we set
$$
\hat{\Delta}_{X_n}^+ := \left\{
\begin{array}{lcl} \Delta_{X_n}^+ & {\text{ if }}&  X_n = C_n \\ \Delta_{X_n}^+ \sqcup \{2 \ep_i \mid 1 \leq i \leq n\}
 & {\text{ if }} & X_n = B_n {\text{ or }} D_n. \end{array} \right.
$$
Identifying $\X_{X_n}$ and $\Y_{X_n}$ by the map $\ep_i \leftrightarrow e_i$ yields an embedding
$$
\iota: \W(X_n) \hookrightarrow \W(A_{\tilde n})
$$
and a surjection
$$
\rho: \Delta_{A_{\tilde{n}}}^+ \to \hat{\Delta}_{X_n}^+.
$$

For $\alpha \in \W(X_n)$ we define $\Inv{\alpha} := \{v \in \Delta_{X_n}^+ \mid \alpha(v) \not \in \Delta_{X_n}^+\}$.
It is clear that this definition of inversion set agrees with our previous definition when $X =A$. Let $\wo_{X_n}$ denote the element of $\W(X_n)$
such that $\Inv{\wo_{X_n}} = \Delta_{X_n}^+$. For $X = B$ or $C$
we have $\iota(\wo_{X_n}) = \wo_{\tilde{n}+1}$, while for $X = D$ this is true if and only if $n$ is even.

Most of the contents of \S\ref{Preliminaries} transfer to the cases when $X = B, C$ or $D$. For instance,
call a set $\Phi \subset \Delta_{X_n}^+$ {\it closed} if $\alpha_1 + \alpha_2 \in \Phi$ whenever
$\alpha_1, \alpha_2 \in \Phi$ and $\alpha_1 + \alpha_2 \in \Delta_{X_n}^+$. Proposition~\ref{proposition 11}
still holds:  $\Phi \subset \Delta_{X_n}^+$ is an inversion set if and only if both $\Phi$ and $\Delta_{X_n}^+ \backslash \Phi$
are closed. Similarly, the obvious analogs of Lemmas~\ref{simple root lemma} and \ref{complement lemma},
Proposition~\ref{coalescence} and Corollary~\ref{coalesce alpha decomposition} hold in general.

The following proposition establishes the behaviour of inversion sets under the maps $\iota$ and $\rho$ above. Its proof
is straightforward and is left to the reader.

\begin{proposition} \label{proposition 41}
Let $X = B, C$ or $D$.
\begin{itemize}
\item[(i)] For any $\alpha \in \W(X_n)$, $\iota(\alpha) \in \W(A_{\tilde{n}})$ is symmetric. If $X = B$ or $C$, the image of $\iota$ consists of
all symmetric permutations in $\W(A_{\tilde{n}})$; if $X = D$, the image of $\iota$ consists of all symmetric permutations $\beta \in \W(A_{\tilde{n}})$ such that
an even number of the elements $\beta(1), \ldots, \beta(n)$ are greater than $n$.
\item[(ii)] The map $\rho$ is surjective. More precisely, each element of $\hat{\Delta}_{X_n}^+$ of the form $2\ep_i$ has a unique preimage in
$\Delta_{A_{\tilde{n}}}^+$ and each of the remaining elements of  $\hat{\Delta}_{X_n}^+$ has exactly two preimages in
$\Delta_{A_{\tilde{n}}}^+$.
\item[(iii)] If $\beta \in \W(A_{\tilde{n}})$ is symmetric then $\rho(\Inv{\beta}) \cap \Delta_{X_n}^+ \subset \Delta_{X_n}^+$ is an inversion set.
\item[(iv)] If $\alpha \in \W(X_n)$ then $\Inv{\iota(\alpha)} = \rho^{-1}(\Inv{\alpha})$.
\item[(v)] Let $\alpha \in \W(X_n)$.  If $X = B$ or $C$ then $\beta = \iota(\alpha)$ is the
unique element of $\W(A_{\tilde{n}})$ such that $\Inv{\alpha} = \rho(\Inv{\beta}) \cap \Delta_{X_n}^+$. If $X = D$ there are exactly two elements
$\beta = \iota(\alpha)$ and $\beta' \not \in \iota(\W(X_n))$ such that $\Inv{\alpha} = \rho(\Inv{\beta}) \cap \Delta_{X_n}^+ =
\rho(\Inv{\beta'}) \cap \Delta_{X_n}^+$. \qed
\end{itemize}
\end{proposition}

Proposition~\ref{proposition 41} implies that the map $\iota$ interacts well with decompositions into inversion sets. More precisely,  the following
statements follow immediately from Proposition~\ref{proposition 41}.

\begin{corollary} \label{corollary 42} $\phantom{x}$
\begin{itemize}
\item[(i)] Assume that $\alpha_1, \alpha_2 \in \W(X_n)$ satisfy $\Inv{\iota(\alpha_1)} \cap \Inv{\iota(\alpha_2)} = \emptyset$.
Then
$$
\rho(\Inv{\iota(\alpha_1)} \sqcup \Inv{\iota(\alpha_2)}) = \rho(\Inv{\iota(\alpha_1)}) \sqcup \rho(\Inv{\iota(\alpha_2)}).
$$
\item[(ii)] Let $\alpha_1,\alpha_2,\dots,\alpha_r \in \W(X_n)$.  Then
$$
\Delta_{X_n}^+ = \sqcup_{i=1}^r \Inv{\alpha_r} \quad {\text{ if and only if }} \quad
\Delta_{A_{\tilde{n}}}^+ = \sqcup_{i=1}^r \Inv{\iota(\alpha_r)}.
$$
\item[(iii)] An element
$\alpha \in \W(X_n)$ is irreducible if and only if $\iota(\alpha) \in \W(A_{\tilde{n}})$ is irreducible.  \qed
\end{itemize}
\end{corollary}

Corollary~\ref{corollary 42} suggests that one can approach studying decompositions of $\Delta_{X_n}^+$ inversion sets by studying
decompositions of  $\Delta_{A_{\tilde{n}}}^+$ into (symmetric) inversion sets.
Indeed, this approach can be carried out successfully in the cases when $X = B$ and $C$. Unfortunately,
ambiguity in Proposition~\ref{proposition 41} (i), (iv) prevented us from obtaining results for $X=D$. The first step is to
define (or attempt to define) an inflation operation for the Weyl groups of types $B$, $C$ and $D$.
Proposition~\ref{proposition 41} (i) allows us to transfer the inflation operation for symmetric permutations to an
inflation operation for the Weyl groups of types $B$ and $C$ but not $D$.

For the sake of completeness, below we provide the
description of an inflation operations for the Weyl groups of types $B$ and $C$. Let $X =B$ or $C$.
Let $0 \leq p \leq n$ and let $\{1,2,\dots,n-p\} = U_1 \sqcup U_2 \sqcup \dots \sqcup U_m$ be a decomposition into intervals.
Put $z_b = |U_b|$ for $b=1,2,\dots,m$.
Suppose that $\sigma \in \left\{ \begin{array}{lcl}
\W(B_m) & {\text { if }} & X_p \neq C_0\\ \W(C_m) & {\text { if }} & X_p = C_0 \end{array} \right.$, $\beta_{m+1} \in \W(X_p)$ and $\beta_b \in S_{z_b}$ for $b=1,2,\dots,m$.
We form the inflation
$$
\tilde{\alpha} = \iota(\sigma)[\beta_1,\beta_2,\dots,\beta_m,\iota(\beta_{m+1}),\beta_{m+2},\dots,\beta_{2m+1}]
$$
where $\beta_{2m+2-b} = \wo_{z_b} \beta_b \wo_{z_b}$ for $t=1,2,\dots,m$.
Note that in the case when $X_p = C_0$ the elements $\beta_{m+1}$ and $\iota(\beta_{m+1})$ are actually empty and hence the expression above is well-defined.
Then $\tilde{\alpha} \in \W(A_{\tilde{n}})$ is symmetric and so
$\tilde{\alpha} = \iota(\alpha)$ for a unique element $\alpha \in \W(X_n)$.  We say that $\alpha$ is an inflation in $\W(X_n)$ and we write
$$
\alpha = \sigma[[\beta_1,\beta_2,\dots,\beta_m;\beta_{m+1}]]
$$
to denote the fact that
$$
\iota(\alpha) = \iota(\sigma)[\beta_1,\beta_2,\dots,\beta_m,\iota(\beta_{m+1}),
\wo_{z_m}\beta_m\wo_{z_m},\dots,\wo_{z_2} \beta_2 \wo_{z_2}, \wo_{z_1} \beta_1\wo_{z_1}]
$$
where $\alpha \in \W(X_n)$, $\sigma \in \left\{ \begin{array}{lcl}
\W(B_m) & {\text { if }} & X_p \neq C_0\\ \W(C_m) & {\text { if }} & X_p = C_0 \end{array} \right.$, $\beta_b \in S_{z_b}$ for $b=1,2,\dots,m$ and $\beta_{m+1} \in \W(X_p)$ with
$z_1 + z_2 + \dots + z_m = n-p$.

An element $\alpha \in \W(X_n)$ which cannot be realized as such an inflation in $\W(X_n)$ except with $m=0$ or $m=n$ is
said to be {\it simple in $\W(X_n)$}. Propositions~\ref{proposition 43} and  \ref{proposition 41}(i) imply immediately
the following statement.

\begin{proposition} \label{proposition 44} Let $X = B$ or $C$ and let
$\alpha \in \W(X_n)$.  Then $\alpha$ is simple in $\W(X_n)$ if and only if $\iota(\alpha)$ is simple in $\W(A_{\tilde{n}})$. \qed
\end{proposition}

We call the expression
$\alpha = \sigma[[\beta_1,\beta_2,\dots,\beta_m;\beta_{m+1}]]$ the {\it simple form expression for
$\alpha \in \W(X_n)$}
if
$$
\iota(\alpha) = \iota(\sigma)[\beta_1,\beta_2,\dots,\beta_m,\iota(\beta_{m+1}),\wo_{z_m}\beta_m\wo_{z_m},\dots,\wo_{z_2} \beta_2 \wo_{z_2}, \wo_{z_1} \beta_1\wo_{z_1}]
$$
is the simple form expression for $\iota(\alpha)$ in $\W(A_{\tilde{n}})$.

Note that the definition of inflation operation above does not apply for type $D$. On one hand, the element $\tilde{\alpha}$ defined above may not belong
to the image of $\iota$ and, on the other hand, for $\alpha \in \W(D_n)$ the element $\iota(\alpha)$ may be an inflation
$$
\tilde{\sigma}[\beta_1,\beta_2,\dots,\beta_m,\tilde{\beta}_{m+1},
\wo_{z_m}\beta_m\wo_{z_m},\dots,\wo_{z_2} \beta_2 \wo_{z_2}, \wo_{z_1} \beta_1\wo_{z_1}]
$$
where $\tilde{\sigma}$ and $\tilde{\beta}_{m+1}$ are symmetric but not necessarily in the image of $\iota$.

\begin{theorem}\label{symmetric main theorem} Let $X = B$ or $C$.
Suppose $\alpha_1, \alpha_2,\dots,\alpha_r \in \W(X_n)$ and
$$
\Delta_{X_n}^+ = \Inv{\alpha_1} \sqcup \Inv{\alpha_2} \sqcup \dots \sqcup \Inv{\alpha_r}
$$
with all $\Inv{\alpha_a} \neq \emptyset$.  Without loss of generality assume that the root $e_1 - e_{\tilde{n}+1} \in \Inv{\iota(\alpha_1)}$.
Let $\alpha_1 =  \sigma_1[[\beta_{11},\beta_{12},\dots,\beta_{1m};\beta_{1(m+1)}]]$ be the simple form expression for $\alpha_1$
with a corresponding partition of the set $\{1,2,\dots,n\}$ into $m+1$ intervals of
 lengths $z_1,z_2,\dots,z_m,z_{m+1}$ where $z_{m+1}=p$.
 Then, up to reordering of $\alpha_2,\alpha_3,\dots,\alpha_r$ there exist unique elements
$\sigma_a \in \left\{ \begin{array}{lcl}
\W(B_m) & {\text { if }} & X_p \neq C_0\\ \W(C_m) & {\text { if }} & X_p = C_0 \end{array} \right.$, $\beta_{ab} \in S_{z_b}$
and $\beta_{a(m+1)} \in \W(X_p)$  such that
$\alpha_a = \sigma_a[[\beta_{a1},\beta_{a2},\dots,\beta_{am};\beta_{a(m+1)}]]$ for $a=2,\dots,r$
and
\begin{itemize}
\item[(i)] $$\begin{array}{rcl}
\Delta^+_{X_m} &=& \Inv{\sigma_1} \sqcup \Inv{\sigma_2} \sqcup \dots \sqcup \Inv{\sigma_r},  \\
&&\\
\Delta^+_{A_{z_1-1}} & = & \Inv{\beta_{11}} \sqcup \Inv{\beta_{21}} \sqcup \dots \sqcup \Inv{\beta_{r1}}, \\
\Delta^+_{A_{z_2-1}} & = & \Inv{\beta_{12}} \sqcup \Inv{\beta_{22}} \sqcup \dots \sqcup \Inv{\beta_{r2}},\\
&\vdots&\\
\Delta^+_{A_{z_m-1}} & = & \Inv{\beta_{1m}} \sqcup \Inv{\beta_{2m}} \sqcup \dots \sqcup \Inv{\beta_{rm}},\\
&&\\
\Delta^+_{X_p} & = & \Inv{\beta_{1(m+1)}} \sqcup \Inv{\beta_{2(m+1)}} \sqcup \dots \sqcup \Inv{\beta_{r(m+1)}};
\end{array}
$$
\item[(ii)]   if $\alpha_1$ is minus-decomposable then $\sigma_1 = \wo$
and $\sigma_2=\sigma_3=\dots=\sigma_r=\Id$;
\item[(iii)]  if $\alpha_1$ is minus-indecomposable then $\sigma_1$ is  simple (in $\W(B_m)$ or in $\W(C_m)$)
 and, after relabelling $\alpha_2, \ldots, \alpha_r$, we have $\sigma_2 = \wo \sigma_1$,
and $\sigma_3=\sigma_4=\dots=\sigma_r=\Id$.
\end{itemize}
In particular, $\sigma_1$ and at most one other of the $\sigma_a$ are not equal to the identity.\\

Let $q$ denote the number of $\sigma_a$ which are not $\Id$, i.e.,
$q := \begin{cases} 1, &\text{if }\alpha_1 \text{ is minus-decomposable};\\
2, &\text{if }\alpha_1 \text{ is minus-indecomposable}.
\end{cases}$
Again, after relabelling $\alpha_2, \ldots, \alpha_r$, we assume that $\sigma_{q+1} = \ldots = \sigma_r = I$.

The above decomposition of $\Delta_{X_n}^+$ is irreducible
 if and only if the following four conditions hold
\begin{itemize}
\item[(i)] each of the decompositions listed in (i) above
is irreducible;
\item[(ii)] exactly one of of $\beta_{a1},\beta_{a2},\dots,\beta_{am}$ is not equal to the identity for $a=q+1,\dots, r$;
\item[(iii)]  $\beta_{ab} = \Id$ for $a=1,\dots, q$ and $b=1,\dots, m+1$;
\item[(iv)]  $m = 1$ if $\alpha_1$ is minus-decomposable.
\end{itemize}
\end{theorem}

\begin{proof}
  This result follows directly from Theorem~\ref{main theorem} and the results of this section.
   Only two additional observations are needed.  The first is that $\wo_{X_m}$ is irreducible if and only if $m=1$.
   The second is that the assumption $e_1 - e_{\tilde{n}+1} \in \Inv{\iota(\alpha_1)}$
    implies that $\iota(\sigma_1)$ is not the identity.
\end{proof}

We conclude this section with a few remarks about decomposing $\Delta_{D_n}^+$ into inversion sets.
As we already mentioned, it is not clear how to define the inflation operation for type $D$. Another possible approach
to decomposing $\Delta_{D_n}^+$ may be to use the fact that $\Delta_{D_n}^+$ embeds naturally into $\Delta_{C_n}^+$.
Indeed, one can show that every decomposition of $\Delta_{C_n}^+$ into inversion sets produces a unique decompsition
of $\Delta_{D_n}^+$ into inversion sets. We do not know, however, whether the converse is true.

\section{Enumerative Results}\label{enumerative section}
%       * Simples (already known, and basis for further computations)
%       * Irreducible decompositions
%          - what the normal form for irreducible decomposition is
%          - resulting recursion.
%       * Triples
%          - same procedure as above
%       * An Catalan numbers.

\newcommand{\np}{\medskip\noindent}

The inductive description for a decomposition provided by
Theorems~\ref{main theorem} and
\ref{symmetric main theorem} allows us to use generating series or recursion to
enumerate many different types of decompositions.   We give a few examples.

Let $s_n$ be the number of simple pairs in $S_n$, i.e., the number of subsets $\{\alpha, \wo \alpha\}$
with $\alpha\in S_{n}$ and both $\alpha$ and $\wo \alpha$ simple
(note that by Lemma~\ref{complements are simple}
$\alpha$ is simple if and only if $\wo \alpha$ is simple).
Let $S_A(z) = \sum_{n\geq 0} s_n z^n = z^2 + z^4 + 3z^5+\cdots $
be the corresponding generating function.
%The following description of $S(z)$ is essentially \cite[equation (2)]{AAK}.
By \cite[page~5]{AAK} we have the following description of $S(z)$.
Let $F(z) = \sum_{n\geq 1} n! z^{n}$ and $G(z) = \sum_{n\geq 1} g_n z^n$ its functional inverse, i.e.,
the function defined by the relation $G(F(z))=z$.  Then $s_1=0$, $s_2=1$, and $s_n = -g_n/2 - (-1)^n$ for $n\geq 3$.

\np
{\bf Number of decompositions into irreducibles.}  Let $a_n$ be the number of decompositions
$\Delta_n^{+} = \Inv{\alpha_1} \sqcup \Inv{\alpha_2} \sqcup \dots \sqcup \Inv{\alpha_r}$ into non-empty inversion sets,
where each $\alpha_k\in S_n$ is irreducible, and where we ignore the order in the decomposition. Let
$A(z) = \sum_{n\geq 1} a_n z^n$ be the generating series. Theorem~\ref{main theorem}
leads to the relation
$A(z) = S_A(A(z)) + z$, which recursively determines the coefficients $a_n$.  Here are the low order terms of $A(z)$:
$$A(z) = z + z^2 + 2z^3 + 6z^4 + 23 z^5 + 114 z^6 + 717 z^7 + 5510 z^8 + 49570 z^9 + 504706 z^{10} + \cdots.$$

\noindent
{\bf Decompositions of maximal length.}
If $\alpha\neq \Id$ then the inversion set $\Inv{\alpha}$ must contain at least one simple root.  Since there
are only $n-1$ simple roots, any decomposition
$\Delta_n^{+} = \Inv{\alpha_1} \sqcup \Inv{\alpha_2} \sqcup \dots \sqcup \Inv{\alpha_r}$,
with no $\alpha_s=\Id$ must satisfy $r\leq n-1$.
Let $\CatA{n-1}$ denote the number of decompositions of $\Delta_n^+$ into exactly $n-1$ non-empty inversion sets.
(Thus each inversion set appearing in the decomposition must contain exactly one simple root.)

\begin{lemma}
$\CatA{n}= \frac{1}{n+1}\binomial{2n}{n}$, the $n^\text{th}$ Catalan number.
\end{lemma}

\begin{proof}
We consider decompositions of the form  $\Delta_n^{+} = \Inv{\alpha_1} \sqcup \Inv{\alpha_2} \sqcup \dots \sqcup \Inv{\alpha_{n-1}}$ and compute $\CatA{n-1}$.
Without loss of generality, the highest root $e_1-e_n \in \Inv{\alpha_1}$.  Suppose that $e_k-e_{k+1}$ is the simple root in $\Inv{\alpha_1}$.
Then $\alpha_1(k+1) < \alpha_1(k+2) < \dots < \alpha_1(n) < \alpha_1(1) < \alpha_1(2) \dots < \alpha_1(k)$ and therefore
$\alpha_1 = (n-k+1, n-k+2, \dots, n, 1 ,2 \dots, n-k) = (1,2)[\Id_k,\Id_{n-k}]$.  Let $U_1 := \{1,2,\dots,k\}$ and $U_2 := \{k+1,k+2,\dots,n\}$.
Then  $\Inv{\alpha_1} = \{ (e_i-e_j \in \Delta_n^+ \mid i \in U_1, j \in U_2\} = \{ e_i-e_j \in \Delta_n^+ \mid i \leq k, j \geq k+1\}$.
Therefore $\Delta_{U_1}^+ \sqcup \Delta_{U_2}^+ = \Inv{\alpha_2} \sqcup \Inv{\alpha_3} \sqcup \dots \sqcup \Inv{\alpha_{n-1}}$.
Without loss of generality,
$\Delta_{U_1}^+ = \Inv{\alpha_2} \sqcup \Inv{\alpha_3} \sqcup \dots \sqcup \Inv{\alpha_{k-1}}$ and
$\Delta_{U_2}^+ = \Inv{\alpha_{k+1}} \sqcup \Inv{\alpha_{k+2}} \sqcup \dots \sqcup \Inv{\alpha_{n-1}}$.
This yields the recursion relation
$\CatA{n-1} = \sum_{t=1}^{n-1} \CatA{t-1}\CatA{n-t-1} = \sum_{t=0}^{n-2} \CatA{t}\CatA{n-t-2}$.
Thus $\CatA{n} = \sum_{t=1}^{n} \CatA{t-1}\CatA{n-t}$.
Since $\CatA{1}=1$ and $\CatA{2}=2$ we see that $\CatA{n}$ satisfies the usual recursion relation for the Catalan numbers.
\end{proof}

\noindent
This incarnation of the Catalan numbers does not currently seem to appear on Richard Stanley's  list
\cite{Cat Addendum} of 207 combinatorial interpretations of the Catalan numbers.

\np
{\bf Type B/C results.}
Theorem~\ref{symmetric main theorem} leads to similar recursions in types $B/C$.
%The version of Theorem~1.9 for Type $B/C$ leads to similar recursions in those types.
Let $S_{B}(z)$ be the generating series for the number of simple pairs in type $B_n/C_n$.  Equivalently the coefficient
of $z^n$ in $S_{B}(z)$ is the number of pairs of simple elements in $S_{2n+1}$ each of which are symmetric.
The isomorphism $\W(B_n) \cong \W(C_n)$ implies that this is also the number of pairs of simple symmetric elements
in $S_{2n+1}$.
One deduces the functional equation
$$S_{B}(F(z)) = 1 - \frac{1}{1+F(2z)} -  \frac{2F(z)}{1+F(z)},$$
(where $F(z) = \sum_{n\geq 1} n!z^n$ as above) which determines $S_{B}(z)$.
Here are some low order terms:
$$S_{B}(z) = 2 z^2 + 10 z^3 + 90 z^4 + 966 z^5 + 12338 z^6 + 181470 z^7 + 3018082 z^8 + 55995486 z^9 + \cdots.$$

\np
{\bf Decompositions into irreducibles.} Let $b_n$ be the number of decompositions of the positive roots in
types $B_n/C_n$ into disjoint unions of irreducible inversion sets, and let $B(z) = \sum_{n\geq 1} b_n z^n$ to be the
generating function.  %One may deduce the following formula for $B(z)$:
Theorem~\ref{symmetric main theorem} leads to the relation

$$B(z) = \frac{S_{B}(A(z))}{1-S_{B}(A(z))},$$
which completely determines $B(z)$.  Here are the low order terms of $B(z)$:
$$B(z) = z + 3 z^2 + 14 z^3 + 100 z^4 + 973 z^5 + 11804 z^6 + 168809 z^7 + 2757930 z^8 + 50522912 z^9 + \cdots.$$

\np
{\bf $\mathbf{B}_n/\mathbf{C}_n$ Catalan numbers.}
Let $\CatB{n}$ be the number of decompositions of the positive roots of $B_n/C_n$
into disjoint unions of inversion sets, where each inversion set contains a single simple root.
The isomorphism  $\W(B_n) \cong \W(C_n)$ implies that the number of such decompositions is the same
for types $B_n$ and $C_n$.
As in type $A$, these are the decompositions of maximal length (subject to the restriction that
each inversion set is non-empty) and thus are irreducible decompositions.

%A similar argument to the type $A$ case leads to the recursion
%$\CatB{n} = \CatB{n-1} + 2 \sum_{k=0}^{n-2} \CatA{n-k-1}\CatB{k}$, which easily implies that the
%generating function for $\CatB{n}$ is

\begin{proposition} \label{B-C Catalan recursion}
The numbers $\CatB{n}$ satisfy the recursion
$\CatB{n} = \CatB{n-1} + 2 \sum_{k=0}^{n-2} \CatA{n-k-1}\CatB{k}$, and thus
$$\sum_{n\geq 1} \CatB{n} z^n =
\frac{1}{(1-4z)^{\frac{1}{2}} + z}.
$$
\end{proposition}

%Let $\CatB{n}$ be the number of decompositions of the positive roots of $B_n/C_n$
%into disjoint unions of inversion sets, where each inversion set contains a single simple root.
%As in type $A$, these are the decompositions of maximal length (subject to the restriction that
%each inversion set is non-empty) and thus these are irreducible decompositions.

\begin{proof}
We consider the $B_n$ case.
Suppose then that $\Delta_{B_n}^+ = \Inv{\alpha_1} \sqcup \Inv{\alpha_2} \sqcup \dots \sqcup \Inv{\alpha_n}$
where each $\alpha_i \in \W(B_n)$ and each $\Inv{\alpha_i}$ contains a single simple root of $\Delta_{B_n}^+$.
Without loss of generality $\Inv{\iota(\alpha_1)}$ contains $e_1 - e_{2n+1}$.
By Theorem~\ref{symmetric main theorem}, we have
$ \alpha_1  =  \sigma_1[[\beta_{11},\beta_{12},\dots,\beta_{1s};\beta_{1(s+1)}]]$
where $\sigma_1 \in \W(B_s)$
and either $\sigma_1 = \wo_{B_1}$ or $\sigma_1$ is simple and $\Inv{\sigma_1}$ contains a single simple root.
Thus if $\sigma_1 \neq \wo_{B_1}$ then $\iota(\sigma_1)$ is simple, symmetric and $\Inv{\iota(\sigma_1)}$ contains a pair of $A_{2n}$ simple roots of the form
$e_i-e_{i+1}, e_{i'-1}-e_{i'}$, where $i' = 2n+2 - i$.
It is not hard to see that this forces $\iota(\sigma_1) = \wo[3]$, $\iota(\sigma_1)  = (41352)$ or
$\iota(\sigma_1) = (25314)$.  The last possibility is excluded by the fact that $\Inv{\iota(\sigma_1)}$ contains the highest root $e_1-e_5$.

First suppose that $\iota(\sigma_1) = \wo[3]$  and let $\{1,2,\dots,2n+1\}=U_1 \sqcup U_2 \sqcup U_3$ be the corresponding
decomposition into intervals with $|U_1|=|U_3|=n-k$ and $|U_2|=2k+1$ where $0\leq k \leq n-1$.
Then $\iota(\alpha_j) = \Id_3[\beta_{j1},\beta_{j2},\beta_{j3}]$ for $j=2,3,\dots,n$.
Furthermore, without loss of generality,
$\Delta_{U_1}^+ = \Inv{\beta_{21}} \sqcup \Inv{\beta_{31}} \sqcup \dots \sqcup \Inv{\beta_{(n-k)1}}$ is a maximal length
decomposition of a root system of type $A_{n-k-1}$.  There are $\CatA{n-k-1}$ such decompositions.
(We also have
$\Delta_{U_3}^+ = \Inv{\wo\beta_{23}\wo} \sqcup \Inv{\wo\beta_{33}\wo} \sqcup \dots \sqcup \Inv{\wo\beta_{(n-k)3}\wo}$.)
Finally
$\Delta_{U_2}^+ = \Inv{{\beta}_{(n-k+1)2}} \sqcup \Inv{{\beta}_{(n-k+2)2}} \sqcup \dots \sqcup \Inv{{\beta}_{n2}}$
is a maximal symmetric decomposition.  There are $\CatB{k}$ such decompositions.
Thus there are $\sum_{k=0}^{n-1} \CatA{n-k-1}\CatB{k}$ maximal decompositions of $\Delta_{B_n}^+$ with
$\iota(\alpha_1) =\wo[3]$.

Next suppose that $\iota(\sigma_1)  = (41352)$ and let $\{1,2,\dots,2n+1\}=U_1 \sqcup U_2 \sqcup \dots \sqcup U_5$ be the
corresponding decomposition.  Then, as above,
$\iota(\alpha_2)  \sqcup \iota(\alpha_3)  \sqcup \dots \sqcup \iota(\alpha_n) $
comprises maximal $A$ type decompositions of $\Delta_{U_1}^+$ and  $\Delta_{U_2}^+$
and a maximal symmetric decomposition of $\Delta_{U_3}^+$.
Thus there are

\begin{eqnarray*}
\sum_{z_1=1}^{n-1} \sum_{z_2=1}^{n-z_1} \CatA{z_1}\CatA{z_2}\CatB{n-z_1-z_2}
&  = & \sum_{k=0}^{n-2}\sum_{z_1+z_2=n-k} \CatB{k}\CatA{z_1}\CatA{z_2}  \\
& = &  \sum_{k=0}^{n-2} \CatB{k}\CatA{n-k-1}
\end{eqnarray*}
maximal decompositions of $\Delta_{B_n}^+$ with $\iota(\alpha_1) = (41352)$.
% (here we have used the recursion for the type $A$ Catalan numbers in simplifying the formula).

Adding the contributions of the two cases gives
$$\CatB{n} = \CatB{n-1} + 2 \sum_{k=0}^{n-2} \CatA{n-k-1}\CatB{k}$$ as claimed.   This easily implies the stated
form of the generating function.
\end{proof}

\np
{\em Remark. }
We have chosen to call these numbers the ``type $B/C$ Catalan numbers'', since they come from an enumerative
problem about Coxeter groups which yields the usual Catalan numbers in the type $A$ case.
There is at least one other use of the
term ``Catalan numbers for other types'' in the literature, again stemming from an enumerative problem
(generalizing non-crossing partitions) valid for all Coxeter groups.  In this second problem, the type $B_n/C_n$
numbers are $\binom{2n}{n}$  (see \cite[pg.\ 39]{Arm}) --
different from the numbers given by the recursion and generating function above.

\np
{\bf Number of decompositions into triples.} The most important case -- in any type -- of the problems
motivating
these questions about decompositions is the case of
decompositions into a disjoint union of three inversion sets.
As described in \S\ref{Motivation} this corresponds to the
the case of the eigenvalues of three Hermitian matrices summing to zero
(respectively the cup product of two cohomology groups into a third, after a similar symmetrization).
The corresponding enumerative/classification problem is to write down all triples
$\alpha_1$, $\alpha_2$, $\alpha_3\in S_n$
(again disregarding order) with $\Delta_n^{+}= \Inv{\alpha_1} \sqcup \Inv{\alpha_2} \sqcup \Inv{\alpha_3}$.
We make the further restriction that no $\alpha_j=\Id$ (all such triples are of the form $(w,\wo w, \Id)$ %with $w\in S_n$
and hence elementary to understand).
Theorems~\ref{main theorem} and \ref{symmetric main theorem}
provide a recursive way to generate and enumerate all such triples.
%(In the recursion it is useful to seperate the triples where $\alpha_1=w_0[m]$ for some $m$).
%Here is a small table of values of such triples, where $T_n$ is the number of triples in $S_n$.
Briefly, the method is a parallel recursion keeping track of not only the triples of the kind above,
but also the subset of those triples where $\alpha_1=\wo[m]$ for some $m$.
At each step, the new triples of each kind depend on the triples of both kinds for smaller $n$.
(We omit the exact description of the recursion since, although elementary, it is slightly messy.)
Here is a small table of the number of such triples, and both the $A_n$ and $B_n/C_n$ cases.

\hfill
\begin{longtable}{|c|r|r|}
\hline
$n$ & $A_n$ triples & $B_n/C_n$ triples \\
\hline
1 &  & 1 \\
2 & 1 & 4 \\
3 & 3 &  33 \\
4 & 17 &  351 \\
5 & 129 & 4210 \\
6 & 1116 & 55495 \\
7 & 10474 & 800476 \\
8 & 104604 & 12654164 \\
9 & 1101012 & 219870187 \\
10 & 12153179 & 4206375350 \\
11 & 140397525 & 88539459103 \\
12 & 1697555983 & 2043502238365 \\
13 & 21516940295 & 51440876843396 \\
14 & 286680892462 &  1403608329020473 \\
15 & 4028129552836 & 41257592671098146 \\
16 & 59885247963954 & 1299045890821350162 \\
17 & 944511887685826 & 43596718839825553381 \\
18 & 15828354015222453 & 1552871403021630700936 \\
19 & 281880601827533671 & 58488502832975791077421 \\
20 & 5327985147037232973 & 2322044948865982864468235 \\
\hline
\end{longtable} \hfill \rule{0cm}{0.1cm}\\

\section{Decomposing a single inversion set} \label{single section}

In this section we provide a recursive algorithm for listing all decompositions of the inversion set $\Inv{\alpha}$
of a given element $\alpha \in S_n$ as
$$
\Inv{\alpha} = \Inv{\alpha_1} \sqcup \Inv{\alpha_2}
$$
and provide a formula for the number of such decompositions\footnote{ We thank Lukas Katth\"an for asking us this question after a previous version of this paper appeared on ArXiv, see \cite{LK}.}.

Let $\alpha = \sigma[\beta_1, \ldots, \beta_m]$ be the simple form of $\alpha$.
To list all ordered decompositions\footnote{ We choose to list the ordered decompositions of $\Inv{\alpha}$ to
simplify the formula for counting them.} of $\Inv{\alpha}$ we proceed as follows:

\noindent
{\bf Step 1.} Write all decompositions
$$
\Inv{\beta_b} = \Inv{\beta_{1b}} \sqcup \Inv{\beta_{2b}} \quad \quad {\text { for }} \quad \quad 1 \leq b \leq m.
$$

\noindent
{\bf Step 2.} For every decomposition of Step 1 write the decompositions
$$
\begin{array}{rcl}
\alpha_1 &=& \sigma[\beta_{11}, \ldots, \beta_{1m}]\\
\alpha_2 &=& \Id_m[\beta_{21}, \ldots, \beta_{2m}]
\end{array}
$$
and, if $\sigma \neq \Id_m$,
$$
\begin{array}{rcl}
\alpha_1 &=& \Id_m[\beta_{11}, \ldots, \beta_{1m}]\\
\alpha_2 &=& \sigma[\beta_{21}, \ldots, \beta_{2m}].
\end{array}
$$
If $\sigma = \Id_m$ or if $\sigma$ is simple
with $m \geq 4$, these are all decompositions of $\Inv{\alpha}$ and the algorithm stops.
The remaining possibility is $\sigma=\wo[m]$, and in this case we continue to the next step.

\noindent
{\bf Step 3.} Write  all partitions $\mathcal{U}$ of the set
$\{1, 2, \ldots, m\}$ into $l \geq 4$ intervals $U_1, U_2, \ldots , U_l$ of lengths $z_1, z_2, \ldots, z_l$ and
for each such partition construct the following elements:
$$
\begin{array}{rcl}
\gamma_1 &=& \wo_{z_1}[\beta_1, \ldots, \beta_{z_1}];\\
\gamma_2 &=& \wo_{z_2}[\beta_{z_1 + 1}, \ldots, \beta_{z_1+ z_2}];\\
& \vdots&\\
\gamma_l &=& \wo_{z_l}[\beta_{z_1+ \ldots + z_{l-1} + 1}, \ldots, \beta_m].
\end{array}
$$

\noindent
{\bf Step 4.} Write  all decompositions
$$
\Inv{\gamma_c} = \Inv{\gamma_{1c}} \sqcup \Inv{\gamma_{2c}} \quad \quad {\text { for }} \quad \quad 1 \leq c \leq l.
$$

\noindent
{\bf Step 5.} For every decomposition of Step 1 and every simple $\sigma \in S_l$ write the decompositions
$$
\begin{array}{rcrl}
\alpha_1 &=& \sigma[\gamma_{11}, \ldots, \gamma_{1l}]&\\
\alpha_2 &=& (\wo_l \sigma) [\gamma_{21}, \ldots, \gamma_{2l}]&.
\end{array}
$$
These complete the list of all decompositions of $\Inv{\alpha}$.

The algorithm above provides a recursive
formula for the number of ordered decompositions of $\Inv{\alpha}$.
For $\alpha \in S_n$, denote by $\dd(\alpha)$ the number of ordered decompositions of
$\Inv{\alpha}$ into two pieces as above.  With this notation $\dd(\Id) = 1$ and $\dd(\sigma) =2$
if $\sigma$ is simple.
As in \S\ref{enumerative section} let $s_l$ denote the number of simple pairs in $S_l$, so that $2s_l$ is the
number of simple elements.
Then, in the notation of the algorithm, one has the following formula for $\dd(\alpha)$:
$$
\dd(\alpha) = \left\{
\begin{array}{ll}
\dd(\beta_1) \cdots \dd(\beta_m) & {\text{ if }}  \sigma = \Id_m\\
2\dd(\beta_1) \cdots  \dd(\beta_m) & {\text{ if }}  \sigma  {\text { is simple and }} m \geq 4\\
2\dd(\beta_1) \cdots \dd(\beta_m) + 2\sum_{l \geq 4} s_l (\sum_{\mathcal U}  \dd(\gamma_1) \ldots \dd(\gamma_l)) & {\text{ if }} \sigma = \wo_m,
\end{array}
\right.
$$
where the summation $\sum_{\mathcal{U}}$ in the third case is over all partitions $\mathcal{U}$ of
$\{1, 2, \ldots, m\}$ into $l$ intervals.

The problem of decomposing a single inversion set can be solved algorithmically for types $B$ and $C$ as well and, furthermore,
 one can also discuss the decomposition of a given inversion set into the disjoint union of a fixed number of
inversion sets. These descriptions are analogous to the one given above and we omit them here.

\section{Parametrizing regular codimension $n$ faces of the Littlewood-Richardson cone}\label{Motivation}
In this section we explain in detail how our work illuminates the structure of the Littlewood-Richardson cone.
For clarity of exposition
we discuss only the case of type $A$ but everything carries over to the cases of types $B$ and $C$.
%The two motivating problems both have versions involving an arbitrary number of factors (which lead to
%Problem~\ref{main problem}) but for simplicity we only describe the case of two factors which leads to considering
%decompositions into three inversion sets.

\np
{\bf Regular faces of the Littlewood-Richardson cone.}
To describe how our work relates to the Littlewood-Richardson cone we first convert the problem of eigenvalues of
Hermitian matrices to its symmetric version, i.e., instead of Hermitian matrices $A,B, C$ satisfying
$C = A+B$ we will consider Hermitian matrices $A,B, C$ satisfying $A+B+C = 0$. It is clear that the cone $\mathcal{C''}$, analogous to the cone $\mathcal{C'}$ described in \S\ref{Introduction}
is contained in the hyperplane $V$ defined by
$$
\lambda_1 + \ldots + \lambda_n + \mu_1 + \ldots + \mu_n + \nu_1 + \ldots + \nu_n = 0
$$
and contains the two-dimensional subspace $W \subset V$ of $(\R^n)^3$ spanned by \newline
$(1, \dots, 1, 0, \ldots, 0, -1, \ldots, -1)$ and
$(0, \dots, 0, 1, \ldots, 1, -1, \ldots, -1)$.
Denote by $\mathcal{C}$ the image of $\mathcal{C''}$ under the projection $V \to V/W$. We will use again
$(\lambda, \mu, \nu)$ to denote the projection of a point in $V$ to $V/W$. The natural coordinates in $V/W$ are
$\lambda = (a_1, \ldots, a_{n-1})$, $\mu = (b_1, \ldots, b_{n-1})$, and $\nu = (c_1, \ldots, c_{n-1})$, where $a_i = \lambda_i - \lambda_{i+1}$,
$b_i = \mu_i - \mu_{i+1}$, and $c_i = \nu_i - \nu_{i+1}$ for $1 \leq i \leq n-1$.
Clearly $V/W \cong (\R^{n-1})^3$ and $S_n$ acts naturally on each of the components
of $(\R^{n-1})^3$: we fix the natural basis $\{\vep_i - \vep_{i+1} \, | \, 1 \leq i \leq n-1\}$ of $\R^{n-1}$ and the action of $S_n$ is by permuting the
indices of this basis. The cone $\mathcal{C}$ is a pointed polyhedral cone of full dimension.
Each of the coordinate hyperplanes $a_i = 0$, $b_i=0$, and $c_i = 0$ for a fixed $i$ with
$1 \leq i \leq n-1$ is a facet of $\mathcal{C}$. Let $(\R^{n-1})^3_+$ denote the dominant cone defined by $a_i \geq 0, b_i \geq 0, c_i \geq 0$ for all $1 \leq i \leq n-1$.
A face of $\mathcal{C}$ is called {\em regular} if it intersects the interior of $(\R^{n-1})^3_+$.
N. Ressayre proved that the regular faces of $\mathcal{C}$ have codimension at most $n-1$. Furthermore, the faces of codimension $n-1$
are exactly the intersection of $(\R^{n-1})^3_+$ with the codimension $n-1$ subspaces $T_{\alpha_1, \alpha_2, \alpha_3}$ defined by
$$
\alpha_1^{-1} \lambda + \alpha_2^{-1} \mu + \alpha_3^{-1} \nu = 0
$$
for $(\alpha_1, \alpha_2, \alpha_3)$ with the property that $\Delta_n^+ = \Phi(\alpha_1) \sqcup \Phi(\alpha_2) \sqcup \Phi(\alpha_3)$,
see \cite[Theorem~C]{R}.
Let $(\alpha_1, \alpha_2, \alpha_3)$
be such a triple and denote by $\mathcal{C}_{\alpha_1, \alpha_2, \alpha_3}$ the corresponding face of $\mathcal{C}$, i.e.,
$\mathcal{C}_{\alpha_1, \alpha_2, \alpha_3} = T_{\alpha_1, \alpha_2, \alpha_3} \cap (\R^{n-1})^3_+ = T_{\alpha_1, \alpha_2, \alpha_3} \cap \mathcal{C}$.

Note that $\mathcal{C}_{\alpha_1, \alpha_2, \alpha_3}$ is described by its defining hyperplanes: $(n-1)$ from the equation
$\alpha_1^{-1} \lambda + \alpha_2^{-1} \mu + \alpha_3^{-1} \nu = 0$ and $3(n-1)$ from $a_i = 0$, $b_i=0$, and $c_i = 0$. It is difficult to conclude from this
description what its defining rays are. We will now show that Theorem~\ref{main theorem} allows us to conclude that $\mathcal{C}_{\alpha_1, \alpha_2, \alpha_3}$
is a simplicial cone and provides an algorithm for writing down its defining rays. (The fact that $\mathcal{C}_{\alpha_1, \alpha_2, \alpha_3}$ is a simplicial cone
also follows from some results in \cite{DR}.) In this section it will be convenient to identify the elements of $\Delta_n^+$ with the vectors $\vep_i - \vep_j$.
Consider the inner product in $\R^{n-1}$ defined by $(\lambda, \vep_i - \vep_j) := \lambda_i - \lambda_j$. It is immediate that, for $1 \leq i < j \leq n$,
$$
(\lambda, \vep_i - \vep_j) = a_i + \ldots + a_{j-1}.
$$
This inner product is $S_n$-invariant; in particular we have $(\alpha^{-1} \lambda, \vep_i - \vep_j) = (\lambda, \alpha(\vep_i - \vep_j))$ for any $\alpha \in S_n$ and
$\vep_i - \vep_j \in \Delta_n^+$. To obtain a set of defining equations
for $T_{\alpha_1, \alpha_2, \alpha_3}$ it is sufficient to chose a basis $\{v_1, v_2,\ldots, v_{n-1}\}$
of $\R^{n-1}$ consisting of elements of $\Delta_n^+$ and write
$$
( \alpha_1^{-1} \lambda + \alpha_2^{-1} \mu + \alpha_3^{-1} \nu, v_i) = 0
$$
for $1 \leq i \leq n-1$. Consider the form of the equation $( \alpha_1^{-1} \lambda + \alpha_2^{-1} \mu + \alpha_3^{-1} \nu, v) = 0$ for $v \in \Delta_n^+$.
Exactly one of the roots $\alpha_1(v), \alpha_2(v), \alpha_3(v)$ is negative, say $\alpha_1(v) = - (\vep_i - \vep_j)$, $\alpha_2(v) = \vep_k - \vep_l$, and
$\alpha_3(v) = \vep_p - \vep_q$. Then $( \alpha_1^{-1} \lambda + \alpha_2^{-1} \mu + \alpha_3^{-1} \nu, v) = 0$ becomes
$$
a_i + \ldots + a_{j-1} = b_k + \ldots + b_{l-1} + c_p + \ldots + c_{q-1}.
$$
This equation is especially simple when $-w_1(v)$ is simple, i.e., when $j = i+1$. Then it becomes
$$
a_i= b_k + \ldots + b_{l-1} + c_p + \ldots + c_{q-1}.
$$
Borrowing from elementary linear algebra, we call $a_i$ an {\it $v$-pivot variable} and
$b_k, \ldots, b_{l-1}$, $c_p, \ldots, c_{q-1}$ {\it $v$-free variables} in this case.

\begin{proposition} \label{cone} Assume that $\Delta_n^+ = \Phi(\alpha_1) \sqcup \Phi(\alpha_2) \sqcup \Phi(\alpha_3)$. The set
$$
S_{\alpha_1, \alpha_2, \alpha_3} = \{v \in \Delta_n^+ \, | \, -\alpha_1(v) {\text { is simple or }} -\alpha_2(v) {\text { is simple or }} -\alpha_3(v) {\text { is simple}}\}
$$
is a basis of $\R^{n-1}$. Furthermore, this set can be labeled $\{ v_1, v_2, \ldots, v_{n-1}\}$ so that, for $i < j$, the $v_i$-pivot
variable is not an $v_j$-free variable.
\end{proposition}

\begin{proof}
Let $\alpha_i = \sigma_i[\beta_{i1}, \beta_{i2}, \ldots, \beta_{im}]$ and let $\{1,2, \ldots, n\} = U_1 \sqcup U_2 \sqcup \ldots \sqcup U_m$ be
the corresponding decomposition into intervals. Assume $v = \vep_i - \vep_j \in S_{\alpha_1,\alpha_2, \alpha_3}$.
Define the {\it level of $v$} inductively as follows: if $i$ and $j$ belong to different parts of $I$, then the level of $v$ is one;
otherwise, $i, j \in U_k$ and the level of $v$ is one plus the level of $v$ for the decomposition
$\Delta_{z_k}^+ = \Phi(\beta_{1k}) \sqcup \Phi(\beta_{2k}) \sqcup \Phi(\beta_{3k})$. Consider the projection $I \to \{1,2, \ldots, m\}$.
Under this projection the level one elements of $S_{\alpha_1, \alpha_2, \alpha_3}$ are sent to the elements of $S_{\sigma_1, \sigma_2, \sigma_3}$
which form a basis since either $\sigma_1 = \wo$ or $\sigma_2 = \wo \sigma_1$. The elements of level greater than one are sent to zero.
On the other hand, by a simple inductive argument,
the elements of level greater than one form bases in the subspace generated by $\{\vep_i - \vep_j\, | \, i, j {\text { in the same }} U_k\}$.
Combining the above we conclude that $S_{\alpha_1, \alpha_2, \alpha_3}$ is a basis of $\R^{n-1}$.

To prove the second assertion, we order $S_{\alpha_1, \alpha_2, \alpha_3}$ linearly so that elements of lower level come before elements of higher level.
Notice first that if $v_1$ is of level one and $v_2$ is of level greater than one, than no
$v_1$-pivot variable is $v_2$-free. Now assume that both $v_1$ and $v_2$ are of level one. Passing to the projection as above, we
conclude again that no $v_1$-pivot variable is $v_2$-free.
\end{proof}

We call the $v_i$-pivot variables simply {\it pivot variables} of $C_{\alpha_1,\alpha_2,\alpha_3}$ and the rest of $a_i, b_i, c_i$
we call {\it free variables}.

\begin{corollary} \label{corollary cone} $C_{\alpha_1,\alpha_2,\alpha_3}$ is a simplicial cone.
\end{corollary}

\begin{proof} It follows from Proposition~\ref{cone} that there are exactly $n-1$ pivot variables. Furthermore, by ordering them as above
we can start from the bottom and replace any pivot variable appearing in the expression of another pivot variable by its expression. When
we reach the top equation, every pivot variable will have become expressed with non-negative coefficients in terms of
the free variables only.
\end{proof}

\begin{example}
We continue with Example~\ref{example}.
%\np
%{\bf Example Continued.}
% Let $n=8$ and let $\alpha_1 = (5,3,4,8,1,2,6,7)$, $\alpha_2 = (4,5,6,1,7,8,3,2)$, $\alpha_3 = (1,3,2,4,6,5,7,8)$. Then $m=4$,
%$$
%\begin{array}{rcl}
%\alpha_1 &=& (2,4,1,3)[(3,1,2),(1),(1,2),(1,2)]\\
%\alpha_2 &=& (3,1,4,2)[(1,2,3),(1),(1,2),(2,1)]\\
%\alpha_3 &=& (1,2,3,4)[(1,3,2),(1),(2,1),(1,2)].
%\end{array}
%$$
Recall that
$\alpha_1 = (4,5,6,1,7,8,3,2)$, $\alpha_2 = (5,3,4,8,1,2,6,7)$, $\alpha_3 = (1,3,2,4,6,5,7,8)$
and $\Delta_8^+ = \Inv{\alpha_1} \sqcup \Inv{\alpha_2} \sqcup \Inv{\alpha_3}$.
The set $S_{\alpha_1, \alpha_2, \alpha_3}$
together with the corresponding equations by level is:
$$
\begin{array} {lllll}
{\text {Level 1: }} & \vep_2 - \vep_6: & a_2 = b_5 + b_6 + b_7 + &c_3 + c_4\\
& \vep_4 - \vep_8: & a_7 = b_1 + &c_4 + c_5 + c_6 + c_7\\
& \vep_1 - \vep_7: & b_3 = a_5 + &c_1 + c_2 + c_3 + c_4 + c_5 + c_6\\
{\text {Level 2: }} & \vep_1 - \vep_3: & a_4 = b_4 + b_5 + &c_1\\
& \vep_5-\vep_6: & c_5 = a_1 + & b_7\\
&\vep_7 - \vep_8: & b_2 = a_6 + & c_7\\
{\text {Level 3:}} & \vep_2 - \vep_3: & c_2 = a_3 + &b_5.
\end{array}
$$
The pivot variables $c_2$ and $c_5$ appear in the expressions for $a_7$ and $b_3$ and need to be replaced. After the appropriate substitutions we
obtain that the generating rays $r_1$,\ldots, $r_{14}$ of $C_{\alpha_1,\alpha_2,\alpha_3}$ corresponding to the free variables
$a_1, a_3, a_5, a_6, b_1, b_4, b_5, b_6, b_7, c_1, c_3, c_4, c_6, c_7$ respectively are:
$$
\begin{array}{l|ccccccccccccccccccccl}
& a_1 & a_2 & a_3 & a_4 & a_5 & a_6 & a_7 &
b_1 & b_2 & b_3 & b_4 & b_5 & b_6 & b_7 &
c_1 & c_2 & c_3 & c_4 & c_5 & c_6 & c_7  \\
\hline
r_1   &1&0&0&0&0&0&1 &0&0&1&0&0&0&0 &0&0&0&0&1&0&0\\
r_2   &0&0&1&0&0&0&0 &0&0&1&0&0&0&0 &0&1&0&0&0&0&0\\
r_3   &0&0&0&0&1&0&0 &0&0&1&0&0&0&0 &0&0&0&0&0&0&0\\
r_4   &0&0&0&0&0&1&0 &0&1&0&0&0&0&0 &0&0&0&0&0&0&0\\
r_5   &0&0&0&0&0&0&1 &1&0&0&0&0&0&0 &0&0&0&0&0&0&0\\
r_6   &0&0&0&1&0&0&0 &0&0&0&1&0&0&0 &0&0&0&0&0&0&0\\
r_7   &0&1&0&1&0&0&0 &0&0&1&0&1&0&0 &0&1&0&0&0&0&0\\
r_8   &0&1&0&0&0&0&0 &0&0&0&0&0&1&0 &0&0&0&0&0&0&0\\
r_9   &0&1&0&0&0&0&1 &0&0&1&0&0&0&1 &0&0&0&0&1&0&0\\
r_{10}&0&0&0&1&0&0&0 &0&0&3&0&0&0&0 &1&0&0&0&0&0&0\\
r_{11}&0&1&0&0&0&0&0 &0&0&1&0&0&0&0 &0&0&1&0&0&0&0\\
r_{12}&0&1&0&0&0&0&1 &0&0&1&0&0&0&0 &0&0&0&1&0&0&0\\
r_{13}&0&0&0&0&0&0&1 &0&0&1&0&0&0&0 &0&0&0&0&0&1&0\\
r_{14}&0&0&0&0&0&0&1 &0&1&0&0&0&0&0 &0&0&0&0&0&0&1.
\end{array}
$$
\end{example}

\renewcommand{\thesubsection}{A.\arabic{subsection}}
\setcounter{subsection}{0}
\renewcommand{\mysubsection}[1]{\refstepcounter{subsection}\vspace{0.3cm}\noindent A.\arabic{subsection}.\ {#1.}}

% The TeX code for the figures in this appendix was generated by the Maple program "drawinv", written by M.R.,
% with some occasional minor modifications by hand.   The commands for generating the pictures are included near
% the diagrams, in order to be able to reproduce or alter the diagrams in the future.

\resetdiagramdefaults   %  Set the defaults for colors and scale of the diagram

\newpage
\noindent
\rule{0.01cm}{0cm}\hfill{\sc Appendix A: Sign diagrams}\hfill\rule{0.01cm}{0cm}

\pdfbookmark[1]{A. Sign Diagrams}{Sign Diagrams}

This appendix is devoted to {\em sign diagrams}, a method of displaying type $A$ inversion sets which in some sense
extends to complete flag varieties the Young diagrams used when describing Schubert cycles on Grassmanians.
Although the use of sign diagrams is not necessary for the proofs of the theorems, many of our arguments have been
guided by diagrammatic thinking and their point of view makes several statements in the paper transparent.

\parshape 9 0cm \textwidth 0cm \textwidth  0cm 10cm 0cm 10cm 0cm 10.5cm 0cm 11cm 0cm 11.5cm 0cm 12cm 0cm \textwidth
\mysubsection{\bf Basic definition}
In order to display the inversion set of an element $\alpha\in S_n$ we start by listing the numbers $1$,\ldots, $n$
across the page, and draw a triangular grid of squares below them, as illustrated at right in the case $n=6$.
Every square in the grid corresponds to exactly one $(i,j)$ with $1\leq i< j\leq n$; the square corresponding
to $(i,j)$ is the unique square which is directly southeast of $i$ and directly southwest of $j$.
In the picture we have labelled the sample squares (a) (1,6); (b) (2,4); and (c) (4,5).

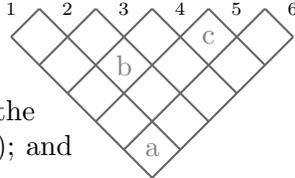
\newgray{minusedge}{0.40}

\vspace{-9\baselineskip}
\hfill
\smash{
\begin{tabular}{c}
\begin{pspicture}(-1,2.8)(5,2)
\SpecialCoor
\putnums{6}
\drawgrid{5}
\putinsquare{1}{6}{\color{gray}{a}}
\putinsquare{2}{4}{\color{gray}{b}}
\putinsquare{4}{5}{\color{gray}{c}}
\end{pspicture}
\end{tabular}
}

\resetdiagramdefaults   %  Reset the defaults for colors and scale of the diagram

\vspace{8.5\baselineskip}

Given $\alpha$ we then mark all the squares corresponding to $(i,j)\in \Inv\alpha$ with a shaded ``$-$''
(to indicate that the positive
root $(i,j)$ is sent to a negative root by $\alpha$), and mark those $(i,j)\not\in\Inv\alpha$
with an unshaded ``$+$'' (to indicate that $(i,j)$ is sent to a positive root by $\alpha$).
In order to reduce clutter in the diagram
we sometimes simply omit the $+/-$ signs or the numbers $1$,\ldots, $n$ at the top,
since these may be deduced from the size and shading of the diagram.
Here is the sign diagram for the inversion set of $\alpha=(1, 6, 3, 5, 2, 4)\in S_6$
displayed using the two different conventions.

% DrawInv([5, 4, 2, 5, 3, 4, 2],5,3) vs DrawInv([5, 4, 2, 5, 3, 4, 2],5,0)

\noindent
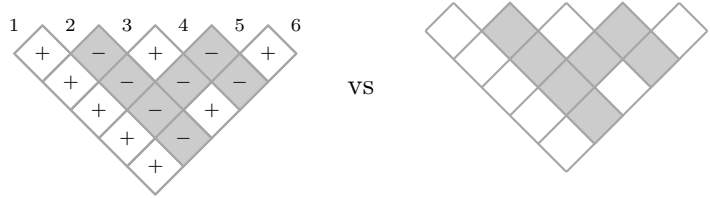
\hfill
\begin{tabular}{c}
\begin{pspicture}(-1,-2)(5,1)
\SpecialCoor
\putnums{6}
\drawgrid{5}
\putplus{1}{2}
\putplus{1}{3}
\putplus{1}{4}
\putplus{1}{5}
\putplus{1}{6}
\putminus{2}{3}
\putminus{2}{4}
\putminus{2}{5}
\putminus{2}{6}
\putplus{3}{4}
\putminus{3}{5}
\putplus{3}{6}
\putminus{4}{5}
\putminus{4}{6}
\putplus{5}{6}
\end{pspicture}
\end{tabular}
vs
\begin{tabular}{c}
\begin{pspicture}(-1,-2.8)(5,1)
\SpecialCoor
\drawgrid{5}
\putminusbox{2}{3}
\putminusbox{2}{4}
\putminusbox{2}{5}
\putminusbox{2}{6}
\putminusbox{3}{5}
\putminusbox{4}{5}
\putminusbox{4}{6}
\end{pspicture}
\end{tabular}
\hfill\rule{0.01cm}{0cm}

The main problem motivating the paper is describing decompositions of $\Delta^{+}_{n}$.
Here are the sign diagrams for such a decomposition with $n=21$, reduced in scale to fit the page.

%DrawTripInv([1, 2, 3, 4, 5, 6, 7, 8, 9, 4], [3, 4, 5, 6, 7, 8, 2, 5, 7, 1, 4, 6, 3, 5, 2, 4, 1, 3],
%    [2, 3, 4, 5, 7, 8, 1, 3, 4, 6, 7, 3, 5, 6, 2, 4, 5],9,0);

% The example below is

% DrawTripInv(
%  [2, 3, 4, 5, 6, 7, 8, 9, 10, 11, 12, 13, 14, 15, 16, 17, 1, 2, 3, 4, 5, 6, 7, 8, 9, 10, 11, 12, 13, 14, 15, 1, 2, 3,
%  5, 6, 8, 9, 10, 11, 12, 13, 14, 5, 7, 9, 10, 11, 12, 13, 6, 8, 10, 11, 12, 5, 7, 9, 11, 4, 6, 8, 10, 5, 7, 9, 4, 6,
%  8, 3, 7, 2, 6, 1, 5], [1, 2, 3, 4, 5, 6, 7, 8, 9, 10, 11, 12, 13, 14, 15, 16, 17, 18, 20, 2, 4, 6, 8, 9, 10, 11,
%  12, 13, 14, 3, 5, 7, 8, 9, 10, 11, 12, 13, 2, 4, 6, 7, 8, 9, 10, 11, 12, 1, 3, 5, 6, 7, 8, 9, 10, 11, 7, 8, 9, 10,
%  6, 7, 8, 9, 5, 6, 7, 8, 4, 5, 6, 3, 4, 5, 2, 3, 4, 3, 2], [2, 3, 4, 5, 6, 7, 8, 9, 10, 11, 12, 13, 14, 15, 16, 17,
%  18, 19, 20, 1, 2, 3, 4, 5, 6, 7, 8, 9, 10, 11, 12, 13, 14, 15, 16, 17, 18, 19, 1, 2, 3, 4, 5, 6, 7, 8, 9, 10, 11,
%  12, 13, 14, 15, 17, 3, 16],20,0);

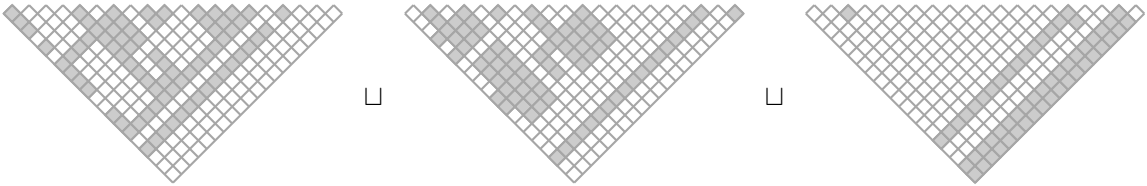
\darkergrays
\psset{unit=0.30}

\hspace{-1.25cm}
\noindent
\begin{tabular}{c@{$\sqcup$}c@{$\sqcup$}c}
\begin{tabular}{c}
\begin{pspicture}(-1,-10.3)(20,1)
\SpecialCoor
\drawgrid{20}
\putminusbox{1}{2}
\putminusbox{1}{3}
\putminusbox{1}{4}
\putminusbox{1}{6}
\putminusbox{1}{8}
\putminusbox{1}{9}
\putminusbox{1}{10}
\putminusbox{1}{11}
\putminusbox{1}{14}
\putminusbox{1}{15}
\putminusbox{1}{16}
\putminusbox{1}{18}
\putminusbox{2}{8}
\putminusbox{2}{16}
\putminusbox{2}{18}
\putminusbox{3}{8}
\putminusbox{3}{16}
\putminusbox{3}{18}
\putminusbox{4}{8}
\putminusbox{4}{16}
\putminusbox{4}{18}
\putminusbox{5}{6}
\putminusbox{5}{7}
\putminusbox{5}{8}
\putminusbox{5}{9}
\putminusbox{5}{10}
\putminusbox{5}{11}
\putminusbox{5}{12}
\putminusbox{5}{13}
\putminusbox{5}{14}
\putminusbox{5}{15}
\putminusbox{5}{16}
\putminusbox{5}{18}
\putminusbox{6}{8}
\putminusbox{6}{9}
\putminusbox{6}{10}
\putminusbox{6}{11}
\putminusbox{6}{15}
\putminusbox{6}{16}
\putminusbox{6}{18}
\putminusbox{7}{8}
\putminusbox{7}{9}
\putminusbox{7}{10}
\putminusbox{7}{11}
\putminusbox{7}{13}
\putminusbox{7}{14}
\putminusbox{7}{15}
\putminusbox{7}{16}
\putminusbox{7}{18}
\putminusbox{8}{16}
\putminusbox{8}{18}
\putminusbox{9}{10}
\putminusbox{9}{11}
\putminusbox{9}{16}
\putminusbox{9}{18}
\putminusbox{10}{11}
\putminusbox{10}{16}
\putminusbox{10}{18}
\putminusbox{11}{16}
\putminusbox{11}{18}
\putminusbox{12}{13}
\putminusbox{12}{14}
\putminusbox{12}{15}
\putminusbox{12}{16}
\putminusbox{12}{18}
\putminusbox{13}{14}
\putminusbox{13}{15}
\putminusbox{13}{16}
\putminusbox{13}{18}
\putminusbox{14}{15}
\putminusbox{14}{16}
\putminusbox{14}{18}
\putminusbox{15}{16}
\putminusbox{15}{18}
\putminusbox{17}{18}
\end{pspicture}
\end{tabular}
&
\begin{tabular}{c}
\begin{pspicture}(-1,-10.3)(20,1)
\SpecialCoor
\drawgrid{20}
\putminusbox{1}{5}
\putminusbox{1}{7}
\putminusbox{1}{12}
\putminusbox{1}{13}
\putminusbox{1}{19}
\putminusbox{2}{3}
\putminusbox{2}{4}
\putminusbox{2}{5}
\putminusbox{2}{6}
\putminusbox{2}{7}
\putminusbox{2}{9}
\putminusbox{2}{10}
\putminusbox{2}{11}
\putminusbox{2}{12}
\putminusbox{2}{13}
\putminusbox{2}{14}
\putminusbox{2}{15}
\putminusbox{2}{19}
\putminusbox{3}{5}
\putminusbox{3}{6}
\putminusbox{3}{7}
\putminusbox{3}{9}
\putminusbox{3}{10}
\putminusbox{3}{11}
\putminusbox{3}{12}
\putminusbox{3}{13}
\putminusbox{3}{14}
\putminusbox{3}{15}
\putminusbox{3}{19}
\putminusbox{4}{5}
\putminusbox{4}{6}
\putminusbox{4}{7}
\putminusbox{4}{9}
\putminusbox{4}{10}
\putminusbox{4}{11}
\putminusbox{4}{12}
\putminusbox{4}{13}
\putminusbox{4}{14}
\putminusbox{4}{15}
\putminusbox{4}{19}
\putminusbox{5}{19}
\putminusbox{6}{7}
\putminusbox{6}{12}
\putminusbox{6}{13}
\putminusbox{6}{14}
\putminusbox{6}{19}
\putminusbox{7}{12}
\putminusbox{7}{19}
\putminusbox{8}{9}
\putminusbox{8}{10}
\putminusbox{8}{11}
\putminusbox{8}{12}
\putminusbox{8}{13}
\putminusbox{8}{14}
\putminusbox{8}{15}
\putminusbox{8}{19}
\putminusbox{9}{12}
\putminusbox{9}{13}
\putminusbox{9}{14}
\putminusbox{9}{15}
\putminusbox{9}{19}
\putminusbox{10}{12}
\putminusbox{10}{13}
\putminusbox{10}{14}
\putminusbox{10}{15}
\putminusbox{10}{19}
\putminusbox{11}{12}
\putminusbox{11}{13}
\putminusbox{11}{14}
\putminusbox{11}{15}
\putminusbox{11}{19}
\putminusbox{12}{19}
\putminusbox{13}{19}
\putminusbox{14}{19}
\putminusbox{15}{19}
\putminusbox{16}{19}
\putminusbox{17}{19}
\putminusbox{18}{19}
\putminusbox{20}{21}
\end{pspicture}
\end{tabular}
&
\begin{tabular}{c}
\begin{pspicture}(-1,-10.3)(20,1)
\SpecialCoor
\drawgrid{20}
\putminusbox{1}{17}
\putminusbox{1}{20}
\putminusbox{1}{21}
\putminusbox{2}{17}
\putminusbox{2}{20}
\putminusbox{2}{21}
\putminusbox{3}{4}
\putminusbox{3}{17}
\putminusbox{3}{20}
\putminusbox{3}{21}
\putminusbox{4}{17}
\putminusbox{4}{20}
\putminusbox{4}{21}
\putminusbox{5}{17}
\putminusbox{5}{20}
\putminusbox{5}{21}
\putminusbox{6}{17}
\putminusbox{6}{20}
\putminusbox{6}{21}
\putminusbox{7}{17}
\putminusbox{7}{20}
\putminusbox{7}{21}
\putminusbox{8}{17}
\putminusbox{8}{20}
\putminusbox{8}{21}
\putminusbox{9}{17}
\putminusbox{9}{20}
\putminusbox{9}{21}
\putminusbox{10}{17}
\putminusbox{10}{20}
\putminusbox{10}{21}
\putminusbox{11}{17}
\putminusbox{11}{20}
\putminusbox{11}{21}
\putminusbox{12}{17}
\putminusbox{12}{20}
\putminusbox{12}{21}
\putminusbox{13}{17}
\putminusbox{13}{20}
\putminusbox{13}{21}
\putminusbox{14}{17}
\putminusbox{14}{20}
\putminusbox{14}{21}
\putminusbox{15}{17}
\putminusbox{15}{20}
\putminusbox{15}{21}
\putminusbox{16}{17}
\putminusbox{16}{18}
\putminusbox{16}{20}
\putminusbox{16}{21}
\putminusbox{17}{20}
\putminusbox{17}{21}
\putminusbox{18}{20}
\putminusbox{18}{21}
\putminusbox{19}{20}
\putminusbox{19}{21}
\end{pspicture}
\end{tabular}
\end{tabular}

\resetdiagramdefaults   %  Reset the defaults for colors and scale of the diagram

\noindent
For large $n$ the inversion sets can become quite intricate, revealing patterns reminicent of cellular automata.

\mysubsection{\bf Connection with Young diagrams}\label{young diagram subsection}
Let $G(r,n)$ denote the Grassmanian of $r$-planes through the origin in $\CC^n$  (with $1\leq r\leq n$).
The cohomology ring of $G(r,n)$ has a $\ZZ$-basis consisting of {\em Schubert cycles}:
cohomology classes Poincar\'e dual to particular Zariski-closed subsets of $G(r,n)$.
Fixing a complete flag in $\CC^n$ (equivalently a Borel subgroup $B$ of $\GL_n(\CC)$), the subsets are the
closures of the points in $G(r,n)$ parameterizing those $r$-planes intersecting the elements of the flag in fixed
dimensions (equivalently the closures of the $B$-orbits).   The combinatorial object parameterizing the data
of how the $r$-planes meet the fixed flag, and therefore parameterizing the cohomology classes, are
the Young diagrams which fit into an $r\times(n-r)$ box.

A similar construction works for the variety $X=\GL_n(\CC)/B$ parameterizing complete flags in $\CC^n$.
Here the subsets are the Zariski closures of the set of points in $X$ where the elements of the flag meet
elements of the fixed flag in prescribed dimensions, or equivalently, the $B$-orbits on $X$.
The combinatorial objects parameterizing the $B$-orbits in this case are the elements of $S_{n}$, the Weyl group
of $\GL_n(\CC)$.

The Grassmanian $G(r,n)$ may be realized as $\GL_n(\CC)/P$, where $P$ is a maximal parabolic subgroup containing $B$
(which maximal subgroup depends on the value of $r$). We therefore have a quotient map
$\pi\colon X\longrightarrow G(r,n)$, and this gives rise to the following procedure.
Start with a Young diagram $\lambda$ fitting in an $r\times(n-r)$ box, take the corresponding Schubert class
$[\Sigma_{\lambda}]$ on $G(r,n)$, pull this back via $\pi$ to a cohomology class $[\Sigma_{\alpha}]$ on $X$
(with $\alpha\in S_{n}$), and finally take the inversion set of $\alpha$, as represented by a sign diagram.
Skipping the cohomology classes and showing only the combinatorial objects (Young diagram, element of $S_{n}$,
and inversion set) here is an example from the cohomology of $G(3,7)$:

\noindent
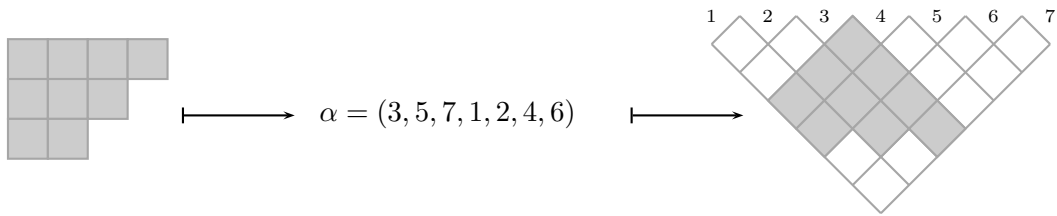
\begin{tabular}{ccccc}
\begin{tabular}{c}
\begin{pspicture}(0,-1.7677)(4,0)
\SpecialCoor
\Youngrow{0}{4}
\Youngrow{-1}{3}
\Youngrow{-2}{2}
\end{pspicture}
\end{tabular}
&
\begin{tabular}{c}
\begin{pspicture}(0.5,-0.1)(0.5,0.1)
\psset{linecolor=black,arrows=|->}
\psline(-1,0)(1,0)
\end{pspicture}
\end{tabular}
&
\begin{tabular}{c}
$\alpha=(3, 5, 7, 1, 2, 4, 6)$
\end{tabular}
&
\begin{tabular}{c}
\begin{pspicture}(-1.0,-0.1)(-1.0,0.1)
\psset{linecolor=black,arrows=|->}
\psline(-1,0)(1,0)
\end{pspicture}
\end{tabular}
&
\begin{tabular}{c}
\begin{pspicture}(-1,-3.3)(6,1)
\SpecialCoor
\drawgrid{6}
\putnums{7}
\putminusbox{1}{4}
\putminusbox{1}{5}
\putminusbox{2}{4}
\putminusbox{2}{5}
\putminusbox{2}{6}
\putminusbox{3}{4}
\putminusbox{3}{5}
\putminusbox{3}{6}
\putminusbox{3}{7}
\end{pspicture}
\end{tabular}

\end{tabular}

\noindent
The conclusion suggested by this example holds in general: the inversion set associated
to a Young diagram $\lambda$ by this procedure {\em is} that same Young diagram, rotated $45^{\circ}$.
For a class on $G(r,n)$ the top corner of the Young diagram appears between the labels $r$ and $r+1$.

% Further discussion?  E.g.:  (1) Why the inversion set fits into an r x (n-r) box, (2) the algorithm
% for determining alpha based on lambda, or (3) the "Lie" reason for the sign diagram, i..e, the bottom
% triangle in the matrix of wB, where w is the Weyl group element thought of as a permutation.

\mysubsection{\bf Inflation}\label{inflation appendix}
The graphical description of inflation follows easily from the ``shuffling cards'' model.
It is again easiest to explain with an example.

% Based on
% DrawInflInv([2,3,1],[3,4,5,3],[[2],[3,1,2,1,3],[4, 2, 3, 1, 4, 2],[1,2]]); + Other commands

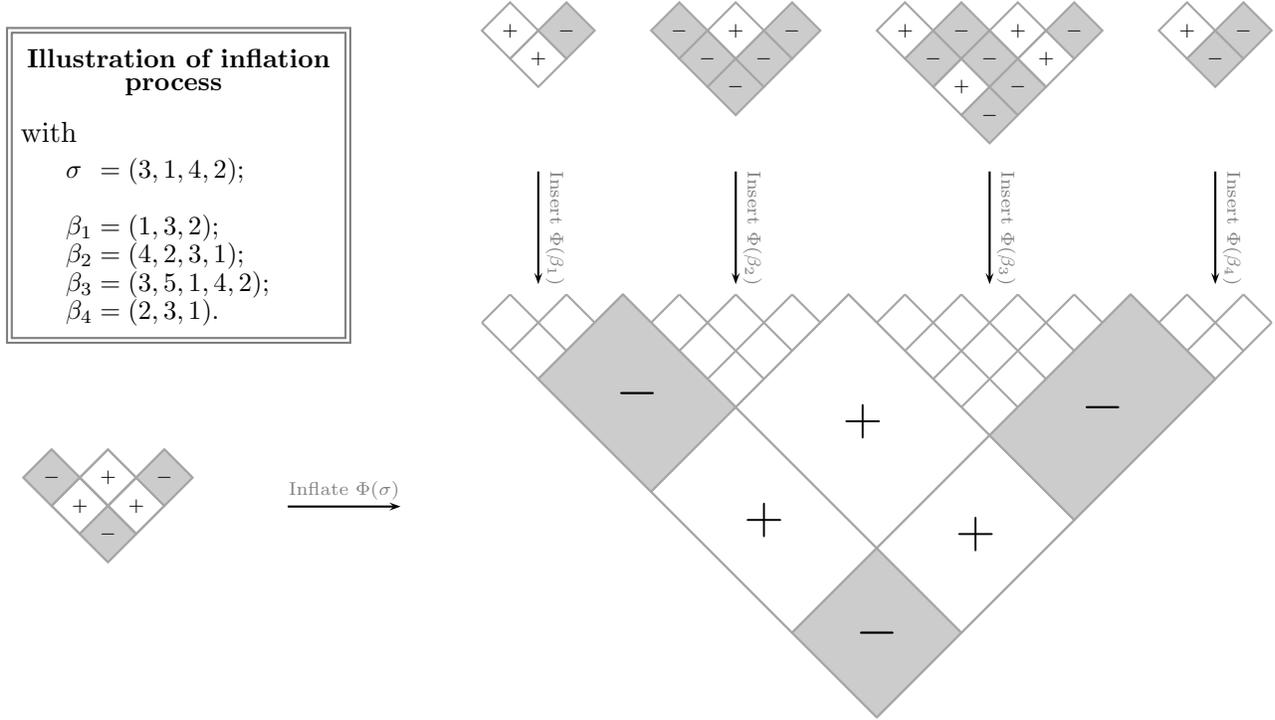
\newgray{vlgray}{0.50}
\hspace{-1.5cm}
\noindent
\begin{tabular}{ccc}
\begin{tabular}{c}
\begin{pspicture}(-2,-3.5)(2,1.5)
\psset{linecolor=vlgray,doubleline=true}
\pspolygon(-2,0.5)(4,0.5)(4,-5)(-2,-5)
\rput(1,0){\small\bf Illustration of inflation}
\rput(1,-0.5){\small\bf process }
\rput[l](-1.8,-1.3){with}
\rput[l](-1,-2.0){\small $\sigma_{\phantom{1}}=(3,1,4,2)$;}
\rput[l](-1,-3.0){\small $\beta_1=(1,3,2)$;}
\rput[l](-1,-3.5){\small $\beta_2=(4,2,3,1)$;}
\rput[l](-1,-4.0){\small $\beta_3=(3,5,1,4,2)$;}
\rput[l](-1,-4.5){\small $\beta_4=(2,3,1)$.}
\end{pspicture}
\end{tabular}
& &
\begin{tabular}{c}
\begin{pspicture}(-1,-4)(14,1)
\SpecialCoor
\putplus{1}{2}
\putplus{1}{3}
\putminus{2}{3}
\putminus{4}{5}
\putminus{4}{6}
\putminus{4}{7}
\putplus{5}{6}
\putminus{5}{7}
\putminus{6}{7}
\putplus{8}{9}
\putminus{8}{10}
\putplus{8}{11}
\putminus{8}{12}
\putminus{9}{10}
\putminus{9}{11}
\putminus{9}{12}
\putplus{10}{11}
\putplus{10}{12}
\putminus{11}{12}
\putplus{13}{14}
\putminus{13}{15}
\putminus{14}{15}
\psset{linecolor=black,arrows=->}
\psline(0.5,-2.5)(0.5,-4.5)
\rput{270}(0.80,-3.5){\color{gray}\tiny Insert $\Inv{\beta_1}$}
\psline(4,-2.5)(4,-4.5)
\rput{270}(4.30,-3.5){\color{gray}\tiny Insert $\Inv{\beta_2}$}
\psline(8.5,-2.5)(8.5,-4.5)
\rput{270}(8.80,-3.5){\color{gray}\tiny Insert $\Inv{\beta_3}$}
\psline(12.5,-2.5)(12.5,-4.5)
\rput{270}(12.80,-3.5){\color{gray}\tiny Insert $\Inv{\beta_4}$}
\end{pspicture}
\end{tabular} \\
\begin{tabular}{c}
\begin{pspicture}(-0.5,-1.8)(3,1)
\SpecialCoor
\drawgrid{3}
\putminus{1}{2}
\putplus{1}{3}
\putminus{1}{4}
\putplus{2}{3}
\putplus{2}{4}
\putminus{3}{4}
\end{pspicture}
\end{tabular}
&
\begin{tabular}{c}
\begin{pspicture}(-1,-0.1)(1,0.1)
\psset{linecolor=black,arrows=->}
\psline(-1,0)(1,0)
\rput(0,0.30){\color{gray}\tiny Inflate $\Inv\sigma$}
\end{pspicture}
\end{tabular}
&
\begin{tabular}{c}
\begin{pspicture}(-1,-7.3)(14,1)
\SpecialCoor
\drawgrid{14}
\putbigminus{1}{4}{3}{7}
\putbigplus{1}{8}{3}{12}
\putbigminus{1}{13}{3}{15}
\putbigplus{4}{8}{7}{12}
\putbigplus{4}{13}{7}{15}
\putbigminus{8}{13}{12}{15}
\end{pspicture}
\end{tabular}
\end{tabular}

\noindent
In this example the fact that $\beta_1$,\ldots, $\beta_4$ are elements of $S_3$, $S_4$, $S_5$, and $S_3$
respectively tells us that the resulting inflation is an element of $S_n$ with $n=3+4+5+3=15$, and that
we should divide $\{1,\ldots, 15\}$ into the consecutive subsets $U_1=\{1,2,3\}$, $U_2=\{4,5,6,7\}$,
$U_3=\{8,9,10,11,12\}$, and $U_4=\{13,14,15\}$ of lengths $3$, $4$, $5$, and $3$ respectively.

The large blocks of $+$ and $-$ signs (indicated by the large blocks with a single $+$ or $-$) result from
permuting the subsets $U_1$,\ldots, $U_4$ as prescribed by $\sigma\in S_4$.
Explicitly, setting  $\alpha=\sigma[\beta_1,\beta_2,\beta_3,\beta_4]$, for every
$(i,j)\in\Inv\sigma$, we have $(a,b)\in\Inv\alpha$ for all $a\in U_i$, $b\in U_j$, and similarly for
$(i,j)\notin\Inv\sigma$.  Each element $(i,j)$ of $\Inv\sigma$ therefore inflates to give an $|U_i|\times |U_j|$
block in $\Inv\alpha$ (length $|U_i|$ in the northeast-southwest direction, $|U_j|$ in the northwest-southeast
direction).   For each $(a,b)\in \Delta^{+}_n$ with $a$ and $b$ in different intervals,
we thus know whether $(a,b)$ is in $\Inv\alpha$ or not.
However, as part of inflation we also permute each $U_i$ using $\beta_i$, and this tells us how to decide
on the status of those $(a,b)$ with $a,b$ in the same interval.
Visually this amounts to simply inserting the sign diagram for $\Inv{\beta_i}$ in the appropriate empty space
left by the inflation process.
This procedure is the graphical translation of Lemma~\ref{easy lemma}.

After inflating, we may leave the large blocks in the diagram to remind us of the inflation, or subdivide them
into the usual smaller squares, depending on the situation.  Thus the inflation above may be represented
(again reduced in scale to fit the page) by

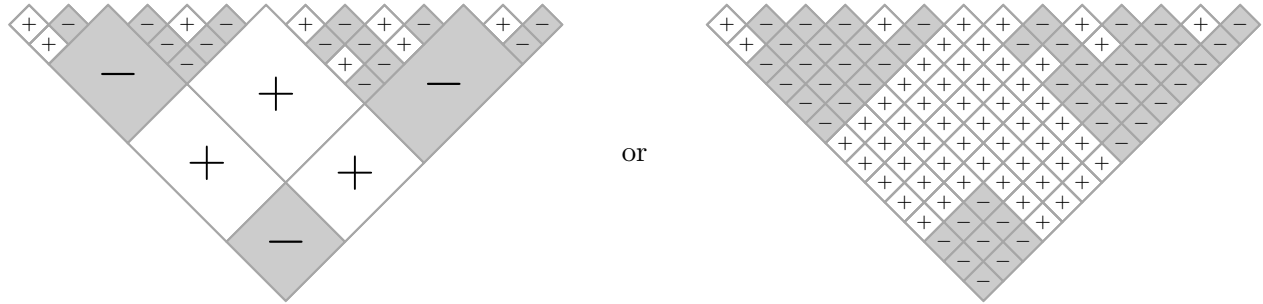
\psset{unit=0.7}
\hspace{-2cm}
\noindent
\begin{tabular}{ccc}

\begin{tabular}{c}
\begin{pspicture}(-1,-7.3)(14,1)
\SpecialCoor
%\putnums{15}
%\connectnums{1}{3}
%\connectnums{4}{7}
%\connectnums{8}{12}
%\connectnums{13}{15}
\drawgrid{14}
\putbigminus{1}{4}{3}{7}
\putbigplus{1}{8}{3}{12}
\putbigminus{1}{13}{3}{15}
\putbigplus{4}{8}{7}{12}
\putbigplus{4}{13}{7}{15}
\putbigminus{8}{13}{12}{15}
\putplus{1}{2}
\putplus{1}{3}
\putminus{2}{3}
\putminus{4}{5}
\putminus{4}{6}
\putminus{4}{7}
\putplus{5}{6}
\putminus{5}{7}
\putminus{6}{7}
\putplus{8}{9}
\putminus{8}{10}
\putplus{8}{11}
\putminus{8}{12}
\putminus{9}{10}
\putminus{9}{11}
\putminus{9}{12}
\putplus{10}{11}
\putplus{10}{12}
\putminus{11}{12}
\putplus{13}{14}
\putminus{13}{15}
\putminus{14}{15}
\end{pspicture}
\end{tabular}
& or &
\begin{tabular}{c}
\begin{pspicture}(-1,-7.3)(14,1)
\SpecialCoor
%\putnums{15}
\drawgrid{14}
\putplus{1}{2}
\putplus{1}{3}
\putminus{1}{4}
\putminus{1}{5}
\putminus{1}{6}
\putminus{1}{7}
\putplus{1}{8}
\putplus{1}{9}
\putplus{1}{10}
\putplus{1}{11}
\putplus{1}{12}
\putminus{1}{13}
\putminus{1}{14}
\putminus{1}{15}
\putminus{2}{3}
\putminus{2}{4}
\putminus{2}{5}
\putminus{2}{6}
\putminus{2}{7}
\putplus{2}{8}
\putplus{2}{9}
\putplus{2}{10}
\putplus{2}{11}
\putplus{2}{12}
\putminus{2}{13}
\putminus{2}{14}
\putminus{2}{15}
\putminus{3}{4}
\putminus{3}{5}
\putminus{3}{6}
\putminus{3}{7}
\putplus{3}{8}
\putplus{3}{9}
\putplus{3}{10}
\putplus{3}{11}
\putplus{3}{12}
\putminus{3}{13}
\putminus{3}{14}
\putminus{3}{15}
\putminus{4}{5}
\putminus{4}{6}
\putminus{4}{7}
\putplus{4}{8}
\putplus{4}{9}
\putplus{4}{10}
\putplus{4}{11}
\putplus{4}{12}
\putplus{4}{13}
\putplus{4}{14}
\putplus{4}{15}
\putplus{5}{6}
\putminus{5}{7}
\putplus{5}{8}
\putplus{5}{9}
\putplus{5}{10}
\putplus{5}{11}
\putplus{5}{12}
\putplus{5}{13}
\putplus{5}{14}
\putplus{5}{15}
\putminus{6}{7}
\putplus{6}{8}
\putplus{6}{9}
\putplus{6}{10}
\putplus{6}{11}
\putplus{6}{12}
\putplus{6}{13}
\putplus{6}{14}
\putplus{6}{15}
\putplus{7}{8}
\putplus{7}{9}
\putplus{7}{10}
\putplus{7}{11}
\putplus{7}{12}
\putplus{7}{13}
\putplus{7}{14}
\putplus{7}{15}
\putplus{8}{9}
\putminus{8}{10}
\putplus{8}{11}
\putminus{8}{12}
\putminus{8}{13}
\putminus{8}{14}
\putminus{8}{15}
\putminus{9}{10}
\putminus{9}{11}
\putminus{9}{12}
\putminus{9}{13}
\putminus{9}{14}
\putminus{9}{15}
\putplus{10}{11}
\putplus{10}{12}
\putminus{10}{13}
\putminus{10}{14}
\putminus{10}{15}
\putminus{11}{12}
\putminus{11}{13}
\putminus{11}{14}
\putminus{11}{15}
\putminus{12}{13}
\putminus{12}{14}
\putminus{12}{15}
\putplus{13}{14}
\putminus{13}{15}
\putminus{14}{15}
\end{pspicture}
\end{tabular}

\end{tabular}

\resetdiagramdefaults   %  Reset the defaults for colors and scale of the diagram

\mysubsection{\bf Relation with ideas from the text}\label{revisit A}
In this subsection we use sign diagrams to illustrate some of the ideas from the main article.

\medskip
\noindent
{\bf If $\sigma$'s and $\beta$'s give a decomposition, so do the inflations.}
As in \S\ref{subsection with main result}
suppose that we divide $\{1,\ldots, n\}$ into $m$ consecutive intervals $U_1$,\ldots, $U_m$,
choose $\sigma_i\in S_m$, $i=1$,\ldots, $r$  such that $\Delta^{+}_m = \sqcup_i \Inv{\sigma_i}$,
and furthermore choose $\beta_{ij}\in S_{|U_j|}$ for $i=1,\ldots, r$, $j=1,\ldots, m$ such that
$\Delta^{+}_{|U_j|}=\sqcup_i \Inv{\beta_{ij}}$ for each $j$.  Then it should be clear from the
visual description of the inflation procedure that this implies the decomposition
$\Delta^{+}_n = \sqcup_i \sigma_i[\beta_{i1},\ldots, \beta_{im}]$.

As an exercise the decomposition from Example~\ref{example}, which is constructed in such a manner,
is pictured below.  The reader is invited to identify the diagrams inflated and inserted in each of the three
pieces of the decomposition and check that they satisfy the hypotheses above.

% DrawInflInv([2,3,1],[3,1,2,2],[[],[],[],[1]]);
% DrawInflInv([1,3,2],[3,1,2,2],[[2,1],[],[],[]]);
% DrawInflInv([],[3,1,2,2],[[2],[],[1],[]]);
% In the current version these were all called with printing option 0, to omit the +'s and -'s.

%\newgray{minusfill}{0.92}
%\newgray{minusedge}{0.87}
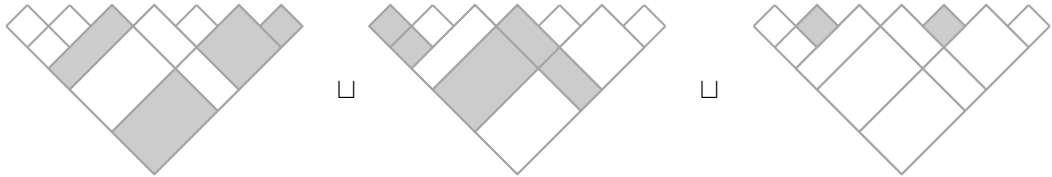
\darkergrays
\psset{unit=0.75}

\noindent
\begin{tabular}{c@{$\sqcup$}c@{$\sqcup$}c}
\begin{tabular}{c}
\begin{pspicture}(-1,-3.8)(7,1)
\SpecialCoor
\drawgrid{7}
\putbigminusbox{1}{4}{3}{4}
\putbigplusbox{1}{5}{3}{6}
\putbigminusbox{1}{7}{3}{8}
\putbigplusbox{4}{5}{4}{6}
\putbigplusbox{4}{7}{4}{8}
\putbigminusbox{5}{7}{6}{8}
\putminusbox{7}{8}
\end{pspicture}
\end{tabular}
&
\begin{tabular}{c}
\begin{pspicture}(-0.5,-3.8)(7,1)
\SpecialCoor
\drawgrid{7}
\putbigplusbox{1}{4}{3}{4}
\putbigminusbox{1}{5}{3}{6}
\putbigplusbox{1}{7}{3}{8}
\putbigminusbox{4}{5}{4}{6}
\putbigminusbox{4}{7}{4}{8}
\putbigplusbox{5}{7}{6}{8}
\putminusbox{1}{2}
\putminusbox{1}{3}
\end{pspicture}
\end{tabular}
&
\begin{tabular}{c}
\begin{pspicture}(-1,-3.8)(7,1)
\SpecialCoor
\drawgrid{7}
\putbigplusbox{1}{4}{3}{4}
\putbigplusbox{1}{5}{3}{6}
\putbigplusbox{1}{7}{3}{8}
\putbigplusbox{4}{5}{4}{6}
\putbigplusbox{4}{7}{4}{8}
\putbigplusbox{5}{7}{6}{8}
\putminusbox{2}{3}
\putminusbox{5}{6}
\end{pspicture}
\end{tabular}
\end{tabular}

\resetdiagramdefaults   %  Reset the defaults for colors and scale of the diagram

\noindent
Theorem~\ref{main theorem}, the central result of the paper, shows conversely is that {\em every} decomposition
of $\Delta^{+}_n$ admits a recursive description by inflations satisfying the above conditions.
The result of the theorem is more precise, identifying a canonical such description satisfying additional properties
well suited to recursive analysis.

\medskip
\noindent
{\bf Rules for indecomposibility.}
If $\alpha=\sigma[\beta_1,\ldots, \beta_m]$ with each $\beta_i\in S_{z_i}$, then it is clear from the
graphical procedure for inflation that we may use this description to decompose $\Inv\alpha$, as
shown in the following example.

% Everything is based on
% DrawInflInv([2, 3, 2, 1],[2, 4, 2, 4],[[], [3, 2, 3], [], [1, 2, 1, 3]],0);

%\newgray{minusfill}{0.92}
%\newgray{minusedge}{0.87}
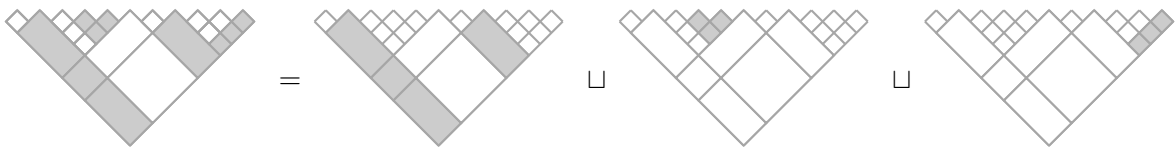
\darkergrays
\psset{unit=0.4}

\hspace{-1.0cm}
\noindent
\begin{tabular}{c@{$=$}c@{$\sqcup$}c@{$\sqcup$}c}
\begin{tabular}{c}
\begin{pspicture}(-0.5,-5.8)(11,1)
\SpecialCoor
\drawgrid{11}
\putbigminusbox{1}{3}{2}{6}
\putbigminusbox{1}{7}{2}{8}
\putbigminusbox{1}{9}{2}{12}
\putbigplusbox{3}{7}{6}{8}
\putbigplusbox{3}{9}{6}{12}
\putbigminusbox{7}{9}{8}{12}
\putplusbox{1}{2}
\putplusbox{3}{4}
\putplusbox{3}{5}
\putplusbox{3}{6}
\putminusbox{4}{5}
\putminusbox{4}{6}
\putminusbox{5}{6}
\putplusbox{7}{8}
\putplusbox{9}{10}
\putplusbox{9}{11}
\putminusbox{9}{12}
\putplusbox{10}{11}
\putminusbox{10}{12}
\putminusbox{11}{12}
\end{pspicture}
\end{tabular}
&
\begin{tabular}{c}
\begin{pspicture}(-0.5,-5.8)(11,1)
\SpecialCoor
\drawgrid{11}
\putbigminusbox{1}{3}{2}{6}
\putbigminusbox{1}{7}{2}{8}
\putbigminusbox{1}{9}{2}{12}
\putbigplusbox{3}{7}{6}{8}
\putbigplusbox{3}{9}{6}{12}
\putbigminusbox{7}{9}{8}{12}
\end{pspicture}
\end{tabular}
&
\begin{tabular}{c}
\begin{pspicture}(-0.5,-5.8)(11,1)
\SpecialCoor
\drawgrid{11}
\putbigplusbox{1}{3}{2}{6}
\putbigplusbox{1}{7}{2}{8}
\putbigplusbox{1}{9}{2}{12}
\putbigplusbox{3}{7}{6}{8}
\putbigplusbox{3}{9}{6}{12}
\putbigplusbox{7}{9}{8}{12}
\putminusbox{4}{5}
\putminusbox{4}{6}
\putminusbox{5}{6}
\end{pspicture}
\end{tabular}
&
\begin{tabular}{c}
\begin{pspicture}(-0.5,-5.8)(11,1)
\SpecialCoor
\drawgrid{11}
\putbigplusbox{1}{3}{2}{6}
\putbigplusbox{1}{7}{2}{8}
\putbigplusbox{1}{9}{2}{12}
\putbigplusbox{3}{7}{6}{8}
\putbigplusbox{3}{9}{6}{12}
\putbigplusbox{7}{9}{8}{12}
\putminusbox{9}{12}
\putminusbox{10}{12}
\putminusbox{11}{12}
\end{pspicture}
\end{tabular}

\end{tabular}

\resetdiagramdefaults   %  Reset the defaults for colors and scale of the diagram

\noindent
In formulas this kind of decomposition is written
$$\Inv\alpha = \Inv{\sigma[\Id_{z_1},\ldots, \Id_{z_m}]} \sqcup \textstyle \bigsqcup_i \displaystyle
\Inv{\Id_m[\Id_{z_1},\ldots, \Id_{z_{i-1}}, \beta_i,\Id_{z_{i+1}},\ldots, \Id_{z_m}]}.$$
%which is somewhat harder to parse than the picture.

For the element $\alpha$ to be irreducible %(i.e., $\Inv\alpha$ has no nontrivial decompositions)
it follows from the inflation decomposition that at most one of $\sigma$, $\beta_1$,\ldots, $\beta_m$
can be different from the identity, and that this nontrivial element must itself be irreducible.
This is the content of Corollary~\ref{easy corollary}.

%Corollary~\ref{irreducible under inflation} shows that if $\sigma$ is irreducible then any inflation of
%$\sigma$ (with the $\beta_j$'s the identity) remains irreducible.
%This is enough to imply that the converse of the above statement holds:
%if $\alpha=\sigma[\beta_1,\ldots, \beta_m]$ is such that at most one of $\sigma$, $\beta_1$,\ldots, $\beta_m$
%is different from the identity, and that element is irreducible, then $\alpha$ is irreducible.    (The
%case that one of the $\beta_j$ is the non-trivial irreducible element is obvious, and the case that $\sigma$
%is that element is taken care of by the corollary.)
%This criterion for irreducibility leads to the second half of Theorem~\ref{main theorem}, characterizing
%irreducible decompositions.

\medskip
\noindent
{\bf Rules for uniqueness in inflations.}
The sign diagram for $\wo[m]$ consists entirely of minus signs.  If $\alpha$ is of the form
$\alpha=\wo[m][\beta_1,\ldots, \beta_m]$ then of course these minus signs are inflated when
making the sign diagram of $\alpha$, and surround the sign diagrams of $\beta_1$,\ldots, $\beta_m$.
If some $\beta_j$ also has this form (i.e., $\beta_j=\wo_{m'}[\tau_1,\ldots, \tau_{m'}]$ for some $m'$)
then some of the minus signs from $\Inv{\beta_j}$ may be merged with the minus signs from the inflation,
as in the following example.

% DrawInflInv([1, 2, 1],[5, 2, 2],[[1, 2, 3, 4, 2, 3, 1, 2], [], []],3)
% DrawInflInv([1, 2, 1],[5, 2, 2],[[1, 2, 3, 4, 2, 3, 1, 2], [], []],11)
% DrawInflInv([1, 2, 3, 4, 1, 2, 3, 1, 2, 1],[2, 2, 1, 2, 2],[[], [], [], [], []],3)

% ( = DrawInflInv(Wo(4),[2,2,1,2,2],[[],[],[],[],[]],3);  )

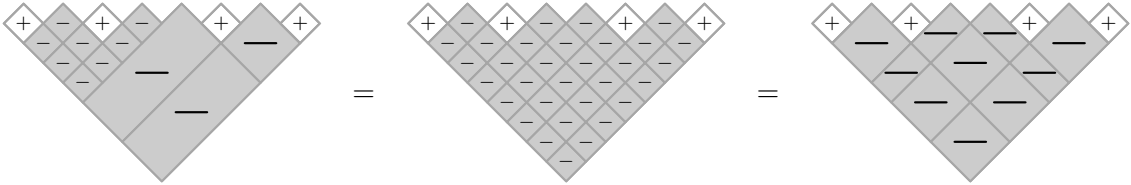
\psset{unit=0.70}

\hspace{-0.75cm}
\noindent
\begin{tabular}{c@{$=$}c@{$=$}c}
\begin{tabular}{c}
\begin{pspicture}(-0.5,-4.3)(8,1)
\SpecialCoor
\drawgrid{8}
\putbigminus{1}{6}{5}{7}
\putbigminus{1}{8}{5}{9}
\putbigminus{6}{8}{7}{9}
\putplus{1}{2}
\putminus{1}{3}
\putminus{1}{4}
\putminus{1}{5}
\putminus{2}{3}
\putminus{2}{4}
\putminus{2}{5}
\putplus{3}{4}
\putminus{3}{5}
\putminus{4}{5}
\putplus{6}{7}
\putplus{8}{9}
\end{pspicture}
\end{tabular}
&
\begin{tabular}{c}
\begin{pspicture}(-1,-4.3)(8,1)
\SpecialCoor
\drawgrid{8}
\putplus{1}{2}
\putminus{1}{3}
\putminus{1}{4}
\putminus{1}{5}
\putminus{1}{6}
\putminus{1}{7}
\putminus{1}{8}
\putminus{1}{9}
\putminus{2}{3}
\putminus{2}{4}
\putminus{2}{5}
\putminus{2}{6}
\putminus{2}{7}
\putminus{2}{8}
\putminus{2}{9}
\putplus{3}{4}
\putminus{3}{5}
\putminus{3}{6}
\putminus{3}{7}
\putminus{3}{8}
\putminus{3}{9}
\putminus{4}{5}
\putminus{4}{6}
\putminus{4}{7}
\putminus{4}{8}
\putminus{4}{9}
\putminus{5}{6}
\putminus{5}{7}
\putminus{5}{8}
\putminus{5}{9}
\putplus{6}{7}
\putminus{6}{8}
\putminus{6}{9}
\putminus{7}{8}
\putminus{7}{9}
\putplus{8}{9}
\end{pspicture}
\end{tabular}
&
\begin{tabular}{c}
\begin{pspicture}(-1,-4.3)(8,1)
\SpecialCoor
\drawgrid{8}
\putbigminus{1}{3}{2}{4}
\putbigminus{1}{5}{2}{5}
\putbigminus{1}{6}{2}{7}
\putbigminus{1}{8}{2}{9}
\putbigminus{3}{5}{4}{5}
\putbigminus{3}{6}{4}{7}
\putbigminus{3}{8}{4}{9}
\putbigminus{5}{6}{5}{7}
\putbigminus{5}{8}{5}{9}
\putbigminus{6}{8}{7}{9}
\putplus{1}{2}
\putplus{3}{4}
\putplus{6}{7}
\putplus{8}{9}
\end{pspicture}
\end{tabular}
\end{tabular}

\resetdiagramdefaults   %  Reset the defaults for colors and scale of the diagram

\noindent
There is an identical problem (with the roles of the + and - signs reversed)
for permutations of the form $\alpha=\Id_m[\beta_1,\ldots, \beta_m]$, where some $\beta_j$ is also of the
form $\beta_j=\Id_{m'}[\tau_1,\ldots, \tau_{m'}]$.
In such cases we obtain uniqueness of the representation as an inflation by requiring that the diagram of $\wo[m]$ or
$\Id_m$ which is inflated
account for as many of the $-$ or $+$ signs in $\Inv{\alpha}$ as possible (i.e., that $m$ be as large as possible).
In the example considered the diagram on the right is the one corresponding to the maximal $\wo[m]$, with $m=5$.

Somewhat the opposite problem occurs for representations of the form
$\alpha=\sigma[\beta_1,\ldots, \beta_m]$ with $\sigma\neq \Id_m,\wo[m]$. In this case it may be that $\sigma$
can itself be represented in a non-trivial way as an inflation, and this description can then be propagated
upwards to give a different representation of $\alpha$ as an inflation (i.e., if
$\sigma=\gamma[\delta_1,\ldots, \delta_r]$ then we may write $\alpha=\gamma[\tau_1,\ldots, \tau_s]$ for some $\tau_i$).
Here is an example where this occurs.

% DrawInflInv([4, 5, 6, 7, 3, 4, 5, 6, 2, 3, 5, 1, 2, 1],[2, 2, 2, 2, 2, 2, 2, 2],
%                      [[], [], [1], [1], [1], [], [1], []],0);
% DrawInflInv([4, 5, 6, 7, 3, 4, 5, 6, 2, 3, 5, 1, 2, 1],[2, 2, 2, 2, 2, 2, 2, 2],
%                      [[], [], [1], [1], [1], [], [1], []],8);
% DrawInflInv([2, 3, 1],[4, 4, 4, 4],[[2, 3, 1, 2], [1, 3], [2, 3, 1, 2, 1], [1]],0);

%\newgray{minusfill}{0.92}
%\newgray{minusedge}{0.87}
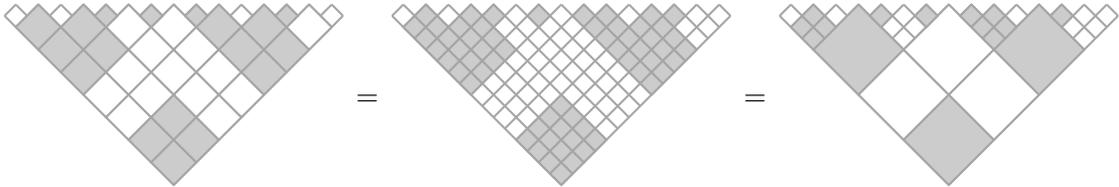
\darkergrays
\psset{unit=0.4}

\hspace{-0.5cm}
\noindent
\begin{tabular}{c@{$=$}c@{$=$}c}
\begin{tabular}{c}
\begin{pspicture}(-0.5,-7.8)(14.5,1)
\SpecialCoor
\drawgrid{15}
\putbigminusbox{1}{3}{2}{4}
\putbigminusbox{1}{5}{2}{6}
\putbigminusbox{1}{7}{2}{8}
\putbigplusbox{1}{9}{2}{10}
\putbigplusbox{1}{11}{2}{12}
\putbigminusbox{1}{13}{2}{14}
\putbigminusbox{1}{15}{2}{16}
\putbigminusbox{3}{5}{4}{6}
\putbigminusbox{3}{7}{4}{8}
\putbigplusbox{3}{9}{4}{10}
\putbigplusbox{3}{11}{4}{12}
\putbigminusbox{3}{13}{4}{14}
\putbigminusbox{3}{15}{4}{16}
\putbigplusbox{5}{7}{6}{8}
\putbigplusbox{5}{9}{6}{10}
\putbigplusbox{5}{11}{6}{12}
\putbigplusbox{5}{13}{6}{14}
\putbigplusbox{5}{15}{6}{16}
\putbigplusbox{7}{9}{8}{10}
\putbigplusbox{7}{11}{8}{12}
\putbigplusbox{7}{13}{8}{14}
\putbigplusbox{7}{15}{8}{16}
\putbigminusbox{9}{11}{10}{12}
\putbigminusbox{9}{13}{10}{14}
\putbigminusbox{9}{15}{10}{16}
\putbigminusbox{11}{13}{12}{14}
\putbigminusbox{11}{15}{12}{16}
\putbigplusbox{13}{15}{14}{16}
\putminusbox{5}{6}
\putminusbox{7}{8}
\putminusbox{9}{10}
\putminusbox{13}{14}
\end{pspicture}
\end{tabular}
&
\begin{tabular}{c}
\begin{pspicture}(-0.5,-7.8)(14.5,1)
\SpecialCoor
\drawgrid{15}
\putminusbox{1}{3}
\putminusbox{1}{4}
\putminusbox{1}{5}
\putminusbox{1}{6}
\putminusbox{1}{7}
\putminusbox{1}{8}
\putminusbox{1}{13}
\putminusbox{1}{14}
\putminusbox{1}{15}
\putminusbox{1}{16}
\putminusbox{2}{3}
\putminusbox{2}{4}
\putminusbox{2}{5}
\putminusbox{2}{6}
\putminusbox{2}{7}
\putminusbox{2}{8}
\putminusbox{2}{13}
\putminusbox{2}{14}
\putminusbox{2}{15}
\putminusbox{2}{16}
\putminusbox{3}{5}
\putminusbox{3}{6}
\putminusbox{3}{7}
\putminusbox{3}{8}
\putminusbox{3}{13}
\putminusbox{3}{14}
\putminusbox{3}{15}
\putminusbox{3}{16}
\putminusbox{4}{5}
\putminusbox{4}{6}
\putminusbox{4}{7}
\putminusbox{4}{8}
\putminusbox{4}{13}
\putminusbox{4}{14}
\putminusbox{4}{15}
\putminusbox{4}{16}
\putminusbox{5}{6}
\putminusbox{7}{8}
\putminusbox{9}{10}
\putminusbox{9}{11}
\putminusbox{9}{12}
\putminusbox{9}{13}
\putminusbox{9}{14}
\putminusbox{9}{15}
\putminusbox{9}{16}
\putminusbox{10}{11}
\putminusbox{10}{12}
\putminusbox{10}{13}
\putminusbox{10}{14}
\putminusbox{10}{15}
\putminusbox{10}{16}
\putminusbox{11}{13}
\putminusbox{11}{14}
\putminusbox{11}{15}
\putminusbox{11}{16}
\putminusbox{12}{13}
\putminusbox{12}{14}
\putminusbox{12}{15}
\putminusbox{12}{16}
\putminusbox{13}{14}
\end{pspicture}
\end{tabular}
&
\begin{tabular}{c}
\begin{pspicture}(-0.5,-7.8)(14.5,1)
\SpecialCoor
\drawgrid{15}
\putbigminusbox{1}{5}{4}{8}
\putbigplusbox{1}{9}{4}{12}
\putbigminusbox{1}{13}{4}{16}
\putbigplusbox{5}{9}{8}{12}
\putbigplusbox{5}{13}{8}{16}
\putbigminusbox{9}{13}{12}{16}
\putminusbox{1}{3}
\putminusbox{1}{4}
\putminusbox{2}{3}
\putminusbox{2}{4}
\putminusbox{5}{6}
\putminusbox{7}{8}
\putminusbox{9}{10}
\putminusbox{9}{11}
\putminusbox{9}{12}
\putminusbox{10}{11}
\putminusbox{10}{12}
\putminusbox{13}{14}
\end{pspicture}
\end{tabular}
\end{tabular}

\resetdiagramdefaults   %  Reset the defaults for colors and scale of the diagram

In these cases we obtain uniqueness of the representation by requiring that $\sigma$ be simple.  This amounts
to looking for $\sigma\in S_m$ with $m$ as small as possible.  In the example considered the diagram on
the right is the one with smallest $m$, with $m=4$.
The two goals ($m$ as large as possible and $m$
as small as possible) are clearly in opposition, and it may occur that when trying to reduce $\sigma$ to be as
small as possible, we arrive at $\sigma=\Id_{m}$ or $\wo[m]$, and then realize that we now have to look for
such a description with $m$ as large as possible.  Nonetheless, as Theorem~\ref{simple form theorem} guarantees,
every $\alpha$ has a unique representation as in inflation $\alpha=\sigma[\beta_1,\ldots, \beta_m]$
with either $\sigma$ simple with $m \geq 4$ or $\sigma$ one of $\Id_m$ or $\wo[m]$ and $m$ as large as possible.

\medskip
\noindent
{\bf Recursion for type $A$ maximal decompositions.} In \S\ref{enumerative section} we considered the problem
of enumerating the decompositions of $\Delta^{+}_n$ of maximal length, i.e., into a decomposition of $n-1$
nonempty inversion sets.  Here is a picture of such a decomposition with $n=8$.

% Diagram below made with

% DrawInv(Abloc([1,2],[1,2],8),8,0);       1x1
% DrawInv(Abloc([1,3],[2,5],8),8,0);       2x3
% DrawInv(Abloc([3,4],[3,4],8),8,0);       1x1
% DrawInv(Abloc([3,5],[4,5],8),8,0);       1x2
% DrawInv(Abloc([5,6],[1,9],8),8,0);       4x4
% DrawInv(Abloc([6,7],[6,8],8),8,0);       1x2
% DrawInv(Abloc([7,8],[7,8],8),8,0);       1x1
% DrawInv(Abloc([6,9],[8,9],8),8,0);       3x1

%\newgray{minusfill}{0.92}
%\newgray{minusedge}{0.87}
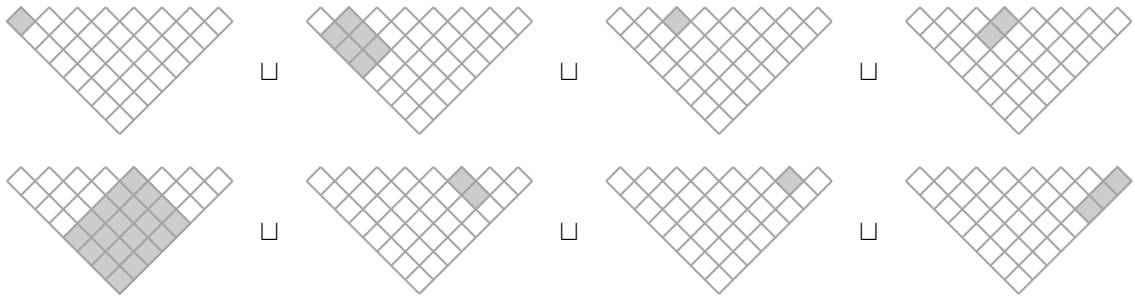
\newgray{numgray}{0.5}
\darkergrays
\psset{unit=0.50}
\newcommand{\decompnum}[1]{%
%\rput(0,-3){\color{numgray}\oldstylenums{#1}}
}

\noindent
\begin{tabular}{c@{$\sqcup$}c@{$\sqcup$}c@{$\sqcup$}c}
% DrawInv([1],8,0)
\begin{tabular}{c}
\begin{pspicture}(-0.5,-4.3)(8,1)
\SpecialCoor
\drawgrid{8}
\putminusbox{1}{2}
\decompnum{1}
\end{pspicture}
\end{tabular}
&
% DrawInv([3, 2, 1, 4, 3, 2],8,0)
\begin{tabular}{c}
\begin{pspicture}(-1,-4.3)(8,1)
\SpecialCoor
\drawgrid{8}
\putminusbox{1}{3}
\putminusbox{1}{4}
\putminusbox{1}{5}
\putminusbox{2}{3}
\putminusbox{2}{4}
\putminusbox{2}{5}
\decompnum{2}
\end{pspicture}
\end{tabular}
&
% DrawInv([3],8,0)
\begin{tabular}{c}
\begin{pspicture}(-1,-4.3)(8,1)
\SpecialCoor
\drawgrid{8}
\putminusbox{3}{4}
\decompnum{3}
\end{pspicture}
\end{tabular}
&
% DrawInv([3, 4],8,0)
\begin{tabular}{c}
\begin{pspicture}(-1,-4.3)(8,1)
\SpecialCoor
\drawgrid{8}
\putminusbox{3}{5}
\putminusbox{4}{5}
\decompnum{4}
\end{pspicture}
\end{tabular}
\\
%\multicolumn{4}{c}{$\bigsqcup$} \\
% DrawInv([4, 3, 2, 1, 5, 4, 3, 2, 6, 5, 4, 3, 7, 6, 5, 4, 8, 7, 6, 5],8,0)
\begin{tabular}{c}
\begin{pspicture}(-0.5,-4.3)(8,1)
\SpecialCoor
\drawgrid{8}
% \putnums{9}
\putminusbox{1}{6}
\putminusbox{1}{7}
\putminusbox{1}{8}
\putminusbox{1}{9}
\putminusbox{2}{6}
\putminusbox{2}{7}
\putminusbox{2}{8}
\putminusbox{2}{9}
\putminusbox{3}{6}
\putminusbox{3}{7}
\putminusbox{3}{8}
\putminusbox{3}{9}
\putminusbox{4}{6}
\putminusbox{4}{7}
\putminusbox{4}{8}
\putminusbox{4}{9}
\putminusbox{5}{6}
\putminusbox{5}{7}
\putminusbox{5}{8}
\putminusbox{5}{9}
\decompnum{5}
\end{pspicture}
\end{tabular}
&
% DrawInv([7, 6],8,0)
\begin{tabular}{c}
\begin{pspicture}(-1,-4.3)(8,1)
\SpecialCoor
\drawgrid{8}
\putminusbox{6}{7}
\putminusbox{6}{8}
\decompnum{6}
\end{pspicture}
\end{tabular}
&
% DrawInv([7],8,0)
\begin{tabular}{c}
\begin{pspicture}(-1,-4.3)(8,1)
\SpecialCoor
\drawgrid{8}
\putminusbox{7}{8}
\decompnum{7}
\end{pspicture}
\end{tabular}
&
% DrawInv([6, 7, 8],8,0)
\begin{tabular}{c}
\begin{pspicture}(-1,-4.3)(8,1)
\SpecialCoor
\drawgrid{8}
\putminusbox{6}{9}
\putminusbox{7}{9}
\putminusbox{8}{9}
\decompnum{8}
\end{pspicture}
\end{tabular}
\end{tabular}

\resetdiagramdefaults   %  Reset the defaults for colors and scale of the diagram

\noindent
The example is relatively small, but is enough to infer the general structure of the problem.

The key is the diagram containing the highest root $(1,n)$ (i.e, the bottom vertex of the triangle), which
in the example is the diagram at lower left. Because the inversion set also contains exactly one simple
root it follows that it must consist of the entire rectangle with corners $(1,n)$ and that simple root.
To see why we look at the example.  In the diagram at lower left the only simple root inverted is $(5,6)$.
This means that the numbers $\{1,2,3,4,5\}$
all retain their relative order when $\alpha$ is applied, and that the same holds for $\{6,7,8,9\}$.
Combined with the fact that the inversion set contains $(1,9)$, so that
$\alpha(9)<\alpha(1)$, we deduce that $\alpha$ swaps the two intervals, i.e., that $\alpha=(6,7,8,9,1,2,3,4,5)$, and
therefore that $\Inv\alpha$ is the rectangle with corners $(5,6)$ and $(1,9)$.

Returning to the general case,
removing the rectangle containing the highest root disconnects the diagram into two smaller diagrams,
each of which must be filled in by the other parts of the decomposition.
The number of maximal decompositions of each of these smaller rectangles may
be computed inductively.  Thus if we organize the counting of the number of maximal decompositions of
$\Delta^{+}_{n+1}$ by the rectangle containing the highest root,
we immediately arrive at the recursive relation $\CatA{n} = \sum_{k=1}^{n} \CatA{k-1}\CatA{n-k}$.
This leads quickly to the result that the enumerative problem is solved by the Catalan numbers.

By induction one also deduces that every diagram in a maximal decomposition is a rectangle.   In the example
all but two of these rectangles are reduced to lines or single squares, but this is simply because the example
is small.

\mysubsection{\texorpdfstring{\bf Diagrams for types $B$ and $C$}{\bf Diagrams for types B and C}}
For us the sign diagrams have been an extremely useful method of visualizing or discovering arguments in the
type $A$ case, and so it is natural to try and extend them to other types.
Our method of displaying the type $A$ inversion sets
arose from picturing what Weyl group elements $w$ do to an upper triangular Borel subgroup (this perspective has not
been explained in the appendix), and one could try and repeat this idea in the other cases.   However, in types
$B/C$ it turns out to be easier to use the group homomorphisms $\iota: \W(B_n)\hookrightarrow S_{2n+1}$ and
$\iota: \W(C_n)\hookrightarrow S_{2n}$ from \S\ref{section BCD}
and, rather than try and picture the inversion set of an element $\alpha\in \W(B_n)\cong\W(C_n)$ directly, to instead
study the inversion set of $\iota(\alpha)$, the image of $\alpha$ under one of the homomorphisms.  We first
briefly recall the groups and the homomorphisms.

The Weyl groups $\W(B_n)$ and $\W(C_n)$ can be identified with the {\em signed permutations} of
$\varepsilon_1$,\ldots, $\varepsilon_n$, i.e., we are allowed not only to permute the elements, but also multiply them
by $\pm$1.  Here is a sample element $\alpha$ of $\W(B_3)\cong \W(C_3)$:
$$
\alpha:\left\{
\begin{array}{c}
\ep_1 \longrightarrow -\ep_2 \\
\ep_2 \longrightarrow \phantom{-}\ep_3 \\
\ep_3 \longrightarrow \phantom{-}\ep_1 \\
\end{array}
\right..
$$

We can promote $\alpha\in \W(C_n)$ to an element $\iota(\alpha) \in S_{2n}$ by considering
$\ep_1$, $\ep_2$, \ldots, $\ep_n$, $-\ep_n$, $-\ep_{n-1}$,\ldots, $-\ep_1$ to be distinct symbols, and using
the rule given by $\alpha$ (and linearity) to deduce a permutation of these $2n$ elements.  For example, the element
$\alpha\in\W(C_3)$ shown above corresponds to
$$
\iota(\alpha):\left\{
\begin{array}{c}
\phantom{-}\ep_1 \longrightarrow -\ep_2 \\
\phantom{-}\ep_2 \longrightarrow \phantom{-}\ep_3 \\
\phantom{-}\ep_3 \longrightarrow \phantom{-}\ep_1 \\
-\ep_3 \longrightarrow -\ep_1 \\
-\ep_2 \longrightarrow -\ep_3 \\
-\ep_1 \longrightarrow \phantom{-}\ep_2 \\
\end{array}
\right..
$$
Using the order $\ep_1,\ep_2, \ep_3, -\ep_3,-\ep_2, -\ep_1$, this is the element
$\iota(\alpha)=(5,3,1,6,4,2)\in S_{6}$.  We can similarly obtain an element in $S_{2n+1}$ by adding the
element $0$ (in the order $\ep_1$, \ldots, $\ep_n$, 0, $-\ep_n$, \ldots, $-\ep_1$) and simply fixing $0$.
In the example considered, this gives the element $(6,3,1,4,7,5,2)\in S_{7}$.   (The way we have presented this
rule `adding $0$' seems somewhat arbitrary, but it does make sense from the natural description of the complete
flag variety of $B_n$ type as a subvariety of the complete flag variety of $A_{2n}$ type.)
We will use the symbol $\iota(\alpha)$ for the image of $\alpha\in \W(C_n)\cong\W(B_n)$ in $S_{2n}$ or
$S_{2n+1}$ under either of these homomorphisms, and trust that the resulting permutation
(of either an even or odd number of elements) will reveal which homomorphism was intended.

Here are the sign diagrams for $\iota(\alpha)$ (in $S_6$ and $S_7$) for the sample element of $\W(C_3)$
considered above.

% DrawInv([2, 3, 4, 3, 1, 5, 4, 2, 1],5,3);
% DrawInv([2, 3, 4, 5, 4, 1, 6, 5, 3, 2, 1],6,3);

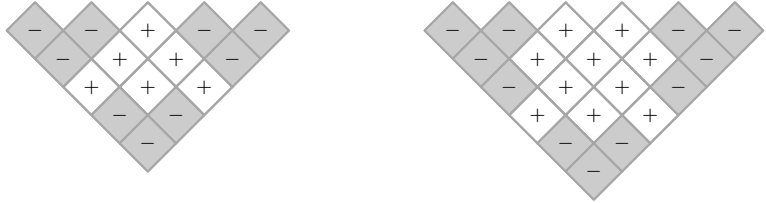
\begin{centering}
\begin{tabular}{ccc}
\begin{tabular}{c}
\begin{pspicture}(-1,-3.3)(5,1)
\SpecialCoor
\drawgrid{5}
\putminus{1}{2}
\putminus{1}{3}
\putplus{1}{4}
\putminus{1}{5}
\putminus{1}{6}
\putminus{2}{3}
\putplus{2}{4}
\putplus{2}{5}
\putminus{2}{6}
\putplus{3}{4}
\putplus{3}{5}
\putplus{3}{6}
\putminus{4}{5}
\putminus{4}{6}
\putminus{5}{6}
\end{pspicture}
\end{tabular}
&  &
\begin{tabular}{c}
\begin{pspicture}(-1,-3.3)(6,1)
\SpecialCoor
\drawgrid{6}
\putminus{1}{2}
\putminus{1}{3}
\putminus{1}{4}
\putplus{1}{5}
\putminus{1}{6}
\putminus{1}{7}
\putminus{2}{3}
\putplus{2}{4}
\putplus{2}{5}
\putplus{2}{6}
\putminus{2}{7}
\putplus{3}{4}
\putplus{3}{5}
\putplus{3}{6}
\putplus{3}{7}
\putplus{4}{5}
\putplus{4}{6}
\putminus{4}{7}
\putminus{5}{6}
\putminus{5}{7}
\putminus{6}{7}
\end{pspicture}
\end{tabular}
\end{tabular} \\
\end{centering}

\noindent
The images of the injective homomorphisms $\W(C_n) \longrightarrow S_{2n}$ and $\W(B_n)\longrightarrow S_{2n+1}$
turn out to be precisely those elements whose sign diagram is symmetric about the vertical centre line, and the
basic idea is to simply study such `symmetric' sign diagrams and the corresponding elements of $S_{2n}$ and
$S_{2n+1}$.

One point is worth stating explicitly: given $\alpha\in\W(C_n)$ or $\W(B_n)$, the inversion
set $\Inv\alpha$ is a subset of $\Delta^{+}_{C_n}$ (or $\Delta^{+}_{B_n}$), while the inversion set
$\Inv{\iota(\alpha)}$ is a subset of $\Delta^{+}_{2n}$ or $\Delta^{+}_{2n+1}$, and these sets can be
quite different. For instance these sets almost never have the same number of elements.  (This is evident in the
example above, where inversion sets for $\iota(\alpha)$ in $S_6$ and $S_7$ don't have the same number of elements
as each other and so could not both agree with the number of elements in $\Inv{\alpha}$.)  More importantly, the ideas
of ``indecomposable'', ``decomposition'', ``disjoint'', or  ``simple''
could potentially be quite different in $\Delta^{+}_{C_n}$ and $\Delta^{+}_{2n}$ (or $\Delta^{+}_{B_n}$ and
$\Delta_{2n+1}^+$), and one of the main things we need to check is that in fact they are not.

% Example based on

% DrawInflInv([1, 2, 4, 3, 2],[3, 4, 6, 4, 3],
%      [[1, 2], [3, 2, 1, 2], [2, 3, 4, 3, 1, 5, 4, 2, 1], [1, 2, 3, 2], [2, 1]],0);

%\newgray{minusfill}{0.92}
%\newgray{minusedge}{0.87}
\darkergrays
\psset{unit=0.4}

\hfill
\begin{tabular}{c}
\begin{pspicture}(-0.5,-9.8)(17,1)
\SpecialCoor
\drawgrid{19}
\putbigplusbox{1}{4}{3}{7}
\putbigplusbox{1}{8}{3}{13}
\putbigminusbox{1}{14}{3}{17}
\putbigplusbox{1}{18}{3}{20}
\putbigminusbox{4}{8}{7}{13}
\putbigminusbox{4}{14}{7}{17}
\putbigminusbox{4}{18}{7}{20}
\putbigminusbox{8}{14}{13}{17}
\putbigplusbox{8}{18}{13}{20}
\putbigplusbox{14}{18}{17}{20}
\putminusbox{1}{3}
\putminusbox{2}{3}
\putminusbox{4}{5}
\putminusbox{4}{6}
\putminusbox{4}{7}
\putminusbox{5}{6}
\putminusbox{8}{9}
\putminusbox{8}{10}
\putminusbox{8}{12}
\putminusbox{8}{13}
\putminusbox{9}{10}
\putminusbox{9}{13}
\putminusbox{11}{12}
\putminusbox{11}{13}
\putminusbox{12}{13}
\putminusbox{14}{17}
\putminusbox{15}{16}
\putminusbox{15}{17}
\putminusbox{16}{17}
\putminusbox{18}{19}
\putminusbox{18}{20}
\end{pspicture}
\end{tabular}

\resetdiagramdefaults   %  Reset the defaults for colors and scale of the diagram

\vspace{-3.3cm}
\parshape 4 0cm 9.0cm 0cm 9.5cm 0cm 10.0cm 0cm 10.5cm
In order to even define ``simple'' we need to have a notion of inflation, and to do this we simply use the
inflation procedure in $S_{2n}$ or $S_{2n+1}$ but require that all the data describing the inflation
be `symmetric'.  At right is an example.

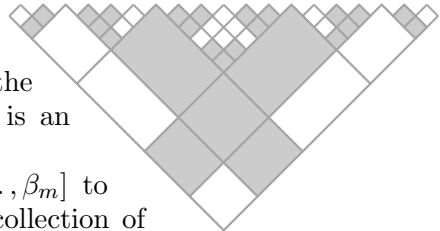
\parshape 4 0cm 11.0cm 0cm 11.5cm 0cm 12.0cm 0cm \textwidth
%\parshape 4 0cm 11.0cm 0cm 11.4cm 0cm 11.8cm 0cm \textwidth
For the data describing an inflation $\iota(\alpha)=\sigma[\beta_1,\ldots, \beta_m]$ to be symmetric means that:
(i) $\sigma$ is symmetric;
(ii) the collection of intervals $U_1$,\ldots, $U_m$ are symmetric (i.e, interchanged by the operation of reversing
$1$, \ldots, $2n$ or $1$, \ldots, $2n+1$);
(iii) if there is an interval $U_j$ which is itself symmetric (i.e., straddles the centre line)
then the corresponding $\beta_j$ must be symmetric; and (iv) for all other intervals $U_j$ the sign diagram
for $\beta_j$ must be the mirror image of the sign diagram for $\beta_{m+1-j}$.   The example presented above
has all these features.   Conditions (ii), (iii) and (iv) may be summarized by the condition that
$\beta_{m+1-j}=\wo[{{z_j}}]\beta_{j}\wo[{{z_j}}]$ for $j=1$,\ldots, $m$.

Propositions~\ref{proposition 43} and \ref{proposition 41}
contain the useful result that if a symmetric $\iota(\alpha)$ can be represented nontrivially as an inflation,
it can be represented nontrivially as an inflation with symmetric data.
With the idea of inflation in place, we now define
an element $\alpha\in\W(B_n)$ or $\W(C_n)$ to be simple if the corresponding $\iota(\alpha)$ is simple
in $S_{2n}$ or $S_{2n+1}$.

In \S\ref{section BCD} the following results are established showing that the intrinsic
notions for an inversion set $\Inv{\alpha}$ with $\alpha\in\W(C_n)\cong\W(B_n)$ in type $B/C$
agree with the the type $A$ notions of the corresponding element $\iota(\alpha)$ in $S_{2n}$ or $S_{2n+1}$.

\begin{itemize}
\item[(i)] Corollary~\ref{corollary 42}:
$\alpha$ is indecomposable if and only if $\iota(\alpha)$ is decomposable;
$\Inv{\alpha_1}$ and $\Inv{\alpha_2}$ are disjoint if and only if
$\Inv{\iota(\alpha_1)}$ and $\Inv{\iota(\alpha_2)}$ are disjoint; $\Delta^{+}_{B_n} = \sqcup_i \Inv{\alpha_i}$
if and only if $\Delta_{2n+1}^{+}=\sqcup_i \Inv{\iota(\alpha_i)}$ (respectively
$\Delta^{+}_{C_n} = \sqcup_i \Inv{\alpha_i}$
if and only if $\Delta_{2n}^{+}=\sqcup_i \Inv{\iota(\alpha_i)}$).

\medskip
\item[(ii)]  Proposition~\ref{proposition 44}:
$\alpha$ is simple if and only if $\iota(\alpha)$ is simple.
\end{itemize}

With these results, one deduces Theorem~\ref{symmetric main theorem} which is the type $B/C$ version of
Theorem~\ref{main theorem}.
The arguments and pictures used in \S\ref{revisit A} also extend in an appropriate way to the $B/C$ case.
For instance, one may also deduce a uniqueness statement for representation as a symmetric inflation, paralleling that
of Theorem~\ref{simple form theorem}, or recursion relations for the type $B/C$ Catalan numbers
(Proposition~\ref{B-C Catalan recursion}).

\bigskip
\bigskip
\bigskip

\setcounter{section}{9}
\section{Acknowledgements}
This work was partially supported by NSERC.  In particular, most of it was done
with the support of NSERC's Undergraduate Summer Research Awards program.

\end{document}